\documentclass[11pt,a4paper]{amsart}
\usepackage{amssymb,amsthm}

\usepackage[truedimen,margin=30truemm]{geometry}

\usepackage[pdftex,colorlinks=true,bookmarks=true,citecolor=blue,linkcolor=blue,pdfauthor={},pdftitle={}]{hyperref}

\theoremstyle{plain}
\newtheorem{theorem}{Theorem}[section]
\newtheorem{lemma}[theorem]{Lemma}

\theoremstyle{definition}

\newtheorem{assumption}[theorem]{Assumption}

\theoremstyle{remark}
\newtheorem{remark}[theorem]{Remark}

\numberwithin{equation}{section}

\begin{document}

\title[Rate of the enhanced dissipation for the two-jet Kolmogorov type flow]{Rate of the enhanced dissipation for the two-jet Kolmogorov type flow on the unit sphere}

\author[Y. Maekawa]{Yasunori Maekawa}
\address{Department of Mathematics, Kyoto University, Kitashirakawa Oiwake-cho, Sakyo-ku, Kyoto 606-8502, Japan}
\email{maekawa.yasunori.3n@kyoto-u.ac.jp}

\author[T.-H. Miura]{Tatsu-Hiko Miura}
\address{Department of Mathematics, Kyoto University, Kitashirakawa Oiwake-cho, Sakyo-ku, Kyoto 606-8502, Japan}
\email{t.miura@math.kyoto-u.ac.jp}



\begin{abstract}
  We study the enhanced dissipation for the two-jet Kolmogorov type flow which is a stationary solution to the Navier--Stokes equations on the two-dimensional unit sphere given by the zonal spherical harmonic function of degree two. Based on the pseudospectral bound method developed by Ibrahim, Maekawa, and Masmoudi \cite{IbMaMa19} and a modified version of the Gearhart--Pr\"{u}ss type theorem shown by Wei \cite{Wei21}, we derive an estimate for the resolvent of the linearized operator along the imaginary axis and show that a solution to the linearized equation rapidly decays at the rate $O(e^{-\sqrt{\nu}\,t})$ when the viscosity coefficient $\nu$ is sufficiently small as in the case of the plane Kolmogorov flow.
\end{abstract}

\maketitle

\section{Introduction} \label{S:Intro}
In this paper, as a continuation of \cite{Miu21pre}, we consider the incompressible Navier--Stokes equations on the two-dimensional (2D) unit sphere $S^2$ in $\mathbb{R}^3$:
\begin{align} \label{E:NS_Intro}
  \partial_t\mathbf{u}+\nabla_{\mathbf{u}}\mathbf{u}-\nu(\Delta_H\mathbf{u}+2\mathbf{u})+\nabla p = \mathbf{f}, \quad \mathrm{div}\,\mathbf{u} = 0 \quad\text{on}\quad S^2\times(0,\infty).
\end{align}
Here $\mathbf{u}$ is the tangential velocity of the fluid, $p$ is the scalar-valued pressure, and $\mathbf{f}$ is a given tangential external force. Also, $\nu>0$ is the viscosity coefficient, $\nabla_{\mathbf{u}}\mathbf{u}$ is the covariant derivative of $\mathbf{u}$ along itself, $\Delta_H$ is the Hodge Laplacian via identification of vector fields and one-forms, and $\nabla$ and $\mathrm{div}$ are the gradient and the divergence on $S^2$. Here the viscous term is taken to be the twice of the divergence of the deformation tensor $\mathrm{Def}\,\mathbf{u}$:
\begin{align*}
  2\,\mathrm{div}\,\mathrm{Def}\,\mathbf{u} = \Delta_H\mathbf{u}+\nabla(\mathrm{div}\,\mathbf{u})+2\,\mathrm{Ric}(\mathbf{u}) = \Delta_H\mathbf{u}+\nabla(\mathrm{div}\,\mathbf{u})+2\mathbf{u},
\end{align*}
where $\mathrm{Ric}\equiv1$ is the Ricci curvature of $S^2$. The Navier--Stokes equations on spheres and more general manifolds with this kind of viscous term have been studied by many authors (see e.g. \cite{Tay92,Prie94,Nag99,MitTay01,DinMit04,KheMis12,ChaCzu13,ChaYon13,SaTaYa13,SaTaYa15,Pie17,KohWen18,PrSiWi20,SamTuo20}). There are also several works on the Navier--Stokes equations on manifolds in which the viscous term is taken to be $\nu\Delta_H\mathbf{u}$ by analogy of the flat domain case (see e.g. \cite{IliFil88,Ili90,CaRaTi99,Ili04,Wir15,Lic16,Ski17}). We refer to \cite{EbiMar70,Aris89,DuMiMi06,Tay11_3,ChCzDi17} for the above identity and the choice of the viscous term in the Navier--Stokes equations on manifolds.

Identifying vector fields with one-forms and taking $\mathrm{rot}=\ast d$ in \eqref{E:NS_Intro}, where $\ast$ and $d$ are the Hodge star operator and the external derivative, we have the vorticity equation
\begin{align} \label{E:Vo_Intro}
  \partial_t\omega+\nabla_{\mathbf{u}}\omega-\nu(\Delta\omega+2\omega) = \mathrm{rot}\,\mathbf{f}, \quad \mathbf{u} = \mathbf{n}_{S^2}\times\nabla\Delta^{-1}\omega \quad\text{on}\quad S^2\times(0,\infty)
\end{align}
for the vorticity $\omega=\mathrm{rot}\,\mathbf{u}$ (see \cite{Miu21pre} for derivation). Here $\nabla_{\mathbf{u}}\omega$ is the directional derivative of $\omega$ along $\mathbf{u}$ and $\Delta$ is the Laplace--Beltrami operator on $S^2$ which is invertible on $L_0^2(S^2)$, the space of $L^2$ functions on $S^2$ with zero mean, $\mathbf{n}_{S^2}$ is the unit outward normal vector field on $S^2$, and $\times$ is the vector product in $\mathbb{R}^3$. Note that the vorticity equation \eqref{E:Vo_Intro} is equivalent to the Navier--Stokes equations \eqref{E:NS_Intro} since any closed one-form on $S^2$ is exact.

For $n\in\mathbb{Z}_{\geq0}$ and $|m|\leq n$ let $Y_n^m$ be the spherical harmonics which satisfy $-\Delta Y_n^m=\lambda_nY_n^m$ with $\lambda_n=n(n+1)$ (see Section \ref{S:Pre}). The vorticity equation \eqref{E:Vo_Intro} with external force $\mathrm{rot}\,\mathbf{f}_n^a=a\nu(\lambda_n-2)Y_n^0$ has a stationary solution with velocity field
\begin{align} \label{E:Kol_Intro}
  \omega_n^a(\theta,\varphi) = aY_n^0(\theta), \quad \mathbf{u}_n^a(\theta,\varphi) = -\frac{a}{\lambda_n\sin\theta}\frac{dY_n^0}{d\theta}(\theta)\partial_\varphi\mathbf{x}(\theta,\varphi)
\end{align}
for $n\in\mathbb{N}$ and $a\in\mathbb{R}$, where $\mathbf{x}(\theta,\varphi)$ is the parametrization of $S^2$ by the colatitude $\theta$ and the longitude $\varphi$ (see Section \ref{S:Pre}). The flow \eqref{E:Kol_Intro} can be seen as the spherical version of the Kolmogorov flow which is a stationary solution to the Navier--Stokes equations in a 2D flat torus with shear external force (see e.g. \cite{MesSin61,Iud65,Mar86,OkaSho93,MatMiy02} for the study of the stability of the plane Kolmogorov flow). Ilyin \cite{Ili04} called \eqref{E:Kol_Intro} the generalized Kolmogorov flow and studied its stability for the Navier--Stokes equations on $S^2$ with viscous term $\nu\Delta_H\mathbf{u}$. Also, Sasaki, Takehiro, and Yamada \cite{SaTaYa13,SaTaYa15} called \eqref{E:Kol_Intro} an $n$-jet zonal flow and investigated its stability for the Navier--Stokes equations on $S^2$ with viscous term $\nu(\Delta_H\mathbf{u}+2\mathbf{u})$. The stability of \eqref{E:Kol_Intro} for the Euler equations on $S^2$ was also studied by Taylor \cite{Tay16}. We call \eqref{E:Kol_Intro} the $n$-jet Kolmogorov type flow in order to emphasize both the similarity to the plane Kolmogorov flow and the number of jets.

In this paper we consider the linear stability of the two-jet Kolmogorov type flow. We substitute $\omega=\omega_2^a+\tilde{\omega}_2$ for \eqref{E:Vo_Intro} and omit the nonlinear term with respect to $\tilde{\omega}_2$ to get
\begin{align*}
  \partial_t\tilde{\omega}_2 = \nu(\Delta\tilde{\omega}_2+2\tilde{\omega}_2)-a_2\cos\theta\,\partial_\varphi(I+6\Delta^{-1})\tilde{\omega}_2, \quad a_2 = \frac{a}{4}\sqrt{\frac{5}{\pi}},
\end{align*}
where $I$ is the identity operator (see \cite{Miu21pre} for derivation of the linearized equation). Replacing $\tilde{\omega}_2$ and $a_2$ by $\omega$ and $a$, we rewrite this equation as
\begin{align} \label{E:Li2_Eq}
  \partial_t\omega = \mathcal{L}^{\nu,a}\omega = \nu A\omega-ia\Lambda\omega, \quad A = \Delta+2, \quad \Lambda = -i\cos\theta\,\partial_\varphi(I+6\Delta^{-1}),
\end{align}
which is considered in $L_0^2(S^2)$. Then since $-A$ is nonnegative and self-adjoint in $L_0^2(S^2)$ and $\Lambda$ is $A$-compact in $L_0^2(S^2)$, the operator $\mathcal{L}^{\nu,a}$ generates an analytic semigroup $\{e^{t\mathcal{L}^{\nu,a}}\}_{t\geq0}$ in $L_0^2(S^2)$ by a perturbation theory of semigroups (see \cite{EngNag00}) and thus the solution of \eqref{E:Li2_Eq} with initial data $\omega_0\in L_0^2(S^2)$ is given by $\omega(t)=e^{t\mathcal{L}^{\nu,a}}\omega_0$. In \cite{Miu21pre} the second author of the present paper proved the linear stability of the two-jet Kolmogorov type flow for all $\nu>0$ by getting the exponential decay of $e^{t\mathcal{L}^{\nu,a}}\omega_0$ towards a (not orthogonal) projection of $\omega_0$ onto the kernel of $\mathcal{L}^{\nu,a}$. Moreover, he showed that $\Lambda$ does not have eigenvalues in $\mathbb{C}\setminus\{0\}$ by making use of the mixing structure of $\Lambda$ expressed by a recurrence relation for $Y_n^m$, and applied it to find that the enhanced dissipation occurs for the rescaled flow $e^{\frac{t}{\nu}\mathcal{L}^{\nu,a}}\omega_0$, which is a solution to $\partial_t\omega=A\omega-i\alpha\Lambda\omega$ with $\alpha=a/\nu$, as in the case of an advection-diffusion equation \cite{CoKiRyZl08,Zla10,Wei21}. More precisely, let
\begin{align*}
  \mathcal{X} = \{u\in L_0^2(S^2) \mid (u,Y_n^0)_{L^2(S^2)}=(u,Y_1^m)_{L^2(S^2)}=0, \, n\geq1, \, |m|=0,1\},
\end{align*}
which is a closed subspace of $L_0^2(S^2)$ invariant under the actions of $A$ and $\Lambda$ (see Section \ref{SS:Ko_Basic}), and $\mathbb{Q}$ be the orthogonal projection from $\mathcal{X}$ onto the orthogonal complement of the kernel of $\Lambda$ restricted on $\mathcal{X}$. Then it was shown in \cite[Theorem 1.4]{Miu21pre} that
\begin{align*}
  \lim_{|a/\nu|\to\infty}\sup_{t\geq\tau}\|\mathbb{Q}e^{\frac{t}{\nu}\mathcal{L}^{\nu,a}}\|_{\mathcal{X}\to\mathcal{X}} = 0 \quad\text{for each}\quad \tau > 0.
\end{align*}
This result in particular gives the convergence of $\mathbb{Q}e^{\frac{t}{\nu}\mathcal{L}^{\nu,a}}\omega_0$ to zero in $L^2(S^2)$ as $\nu\to0$ for each fixed $t>0$ and $a\in\mathbb{R}$, but does not give the actual convergence rate. The purpose of this paper is to give an explicit decay rate of the original flow $\mathbb{Q}e^{t\mathcal{L}^{\nu,a}}\omega_0$.

In the plane case, Beck and Wayne \cite{BecWay13} numerically conjectured that a perturbation of the plane Kolmogorov flow decays at the rate $O(e^{-\sqrt{\nu}\,t})$ when the viscosity coefficient $\nu$ is sufficiently small. They also verified this enhanced dissipation for a linearized operator without a nonlocal term based on the hypocoercivity method developed by Villani \cite{Vil09}. Lin and Xu \cite{LinXu19} proved the enhanced dissipation for a full linearized operator but without an explicit decay rate by using the Hamiltonian structure of a perturbation operator and the RAGE theorem. The decay rate $O(e^{-\sqrt{\nu}\,t})$ was confirmed by Ibrahim, Maekawa, and Masmoudi \cite{IbMaMa19} based on the pseudospectral bound method, by Wei and Zhang \cite{WeiZha19} based on the hypocoercivity method, and by Wei, Zhang, and Zhao \cite{WeZhZh20} based on the wave operator method.

In this paper we show that $\mathbb{Q}e^{t\mathcal{L}^{\nu,a}}\omega_0$ decays at the rate $O(e^{-\sqrt{\nu}\,t})$ as in the plane case. For $m\in\mathbb{Z}$ let $\mathcal{P}_m$ be the orthogonal projection from $L^2(S^2)$ onto the space of $L^2$ functions on $S^2$ of the form $u=U(\theta)e^{im\varphi}$ (see \eqref{E:Def_Proj} for the precise definition). Note that $\mathcal{X}$ is orthogonally decomposed as $\mathcal{X}=\oplus_{m\in\mathbb{Z}\setminus\{0\}}\mathcal{P}_m\mathcal{X}$ and
\begin{align*}
  \mathbb{Q}\mathcal{P}_mu = \mathcal{P}_mu, \quad |m|\geq 3, \quad (I-\mathbb{Q})\mathcal{P}_mu=(\mathcal{P}_mu,Y_2^m)_{L^2(S^2)}Y_2^m, \quad |m|=1,2
\end{align*}
for $u\in\mathcal{X}$ (see Section \ref{SS:Ko_Basic}). The main result of this paper is as follows.

\begin{theorem} \label{T:OL_Dec}
  There exist constants $C_1,C_2>0$ such that
  \begin{align} \label{E:OL_Dec_Q}
    \|\mathbb{Q}\mathcal{P}_me^{t\mathcal{L}^{\nu,a}}\omega_0\|_{L^2(S^2)} \leq C_1e^{-C_2|a|^{1/2}\nu^{1/2}|m|^{2/3}t}\|\mathbb{Q}\mathcal{P}_m\omega_0\|_{L^2(S^2)}
  \end{align}
  for all $t\geq0$, $\omega_0\in\mathcal{X}$, $\nu>0$, $a\in\mathbb{R}$, and $m\in\mathbb{Z}\setminus\{0\}$. Also, if $|m|=1,2$ and $|a/\nu|$ is sufficiently large, then for all $t\geq0$ and $\omega_0\in\mathcal{X}$ we have
  \begin{multline*}
    \|(I-\mathbb{Q})\mathcal{P}_me^{t\mathcal{L}^{\nu,a}}\omega_0\|_{L^2(S^2)} \leq e^{-4\nu t}\|(I-\mathbb{Q})\mathcal{P}_m\omega_0\|_{L^2(S^2)} \\
    +C_3\left|\frac{am}{\nu}\right|^{1/3}\log\left(C_4\left|\frac{am}{\nu}\right|^{2/3}\right)e^{-2\nu t}\|\mathbb{Q}\mathcal{P}_m\omega_0\|_{L^2(S^2)},
  \end{multline*}
  where $C_3,C_4>0$ are constants independent of $t$, $\omega_0$, $\nu$, $a$, and $m$.
\end{theorem}

Here we note that the exponent of $|m|$ in \eqref{E:OL_Dec_Q} is $2/3$ in the sphere case, while it is $1/2$ in the plane case (see \cite[Corollary 3.14]{IbMaMa19}). This difference comes from the factor $|m|^{-1/2}$ appearing in the estimate for the $L^\infty(S^2)$-norm of $u=U(\theta)e^{im\varphi}$ by the $L^2(S^2)$-norm of $\nabla u$ (see Lemma \ref{L:Linf}) which yields a different coercive estimate for $\mathbb{Q}_m(\mu-\Lambda_m)$ with $\mu\in\mathbb{R}$, $\mathbb{Q}_m=\mathbb{Q}|_{\mathcal{P}_m\mathcal{X}}$, and $\Lambda_m=m^{-1}\Lambda|_{\mathcal{P}_m\mathcal{X}}$ (see Lemma \ref{L:Ko_Middle}).

To prove Theorem \ref{T:OL_Dec} we follow the argument of the work by Ibrahim, Maekawa, and Masmoudi \cite{IbMaMa19} which verified the enhanced dissipation for the plane Kolmogorov flow by the pseudospectral bound method. By rescaling time as $t\mapsto \nu t$ in \eqref{E:Li2_Eq} and taking the Fourier series with respect to the longitude $\varphi$, we consider the equation
\begin{align*}
  \partial_t\omega = L_{\alpha,m}\omega = A_m\omega-i\alpha m\Lambda_m\omega, \quad A_m = A|_{\mathcal{X}_m}, \quad \Lambda_m = M_{\cos\theta}(I+6\Delta^{-1})|_{\mathcal{X}_m}
\end{align*}
in $\mathcal{X}_m=\mathcal{P}_m\mathcal{X}$ for each $m\in\mathbb{Z}\setminus\{0\}$, where $\alpha=a/\nu\in\mathbb{R}$ and $M_{\cos\theta}$ is the multiplication operator by $\cos\theta$. By the definition of $\mathcal{X}_m$, the operator $-A_m$ is self-adjoint and positive in $\mathcal{X}_m$. Also, the kernel of $\Lambda_m$ is $\{cY_2^m\mid c\in\mathbb{C}\}$ when $|m|=1,2$ and trivial when $|m|\geq3$. We intend to estimate the semigroup $e^{tL_{\alpha,m}}$ generated by $L_{\alpha,m}$, especially $\mathbb{Q}_me^{tL_{\alpha,m}}$ or equivalently $e^{t\mathbb{Q}_mL_{\alpha,m}}$, where $\mathbb{Q}_m$ is the orthogonal projection from $\mathcal{X}_m$ onto $\mathcal{Y}_m$, the orthogonal complement of the kernel of $\Lambda_m$. To this end, we apply abstract results given in Section \ref{S:Abst} to $L_{\alpha,m}$ to derive an estimate for the quantity
\begin{align*}
  \left(\sup_{\lambda\in\mathbb{R}}\|(i\lambda-\mathbb{Q}_mL_{\alpha,m})^{-1}\|_{\mathcal{Y}_m\to\mathcal{Y}_m}\right)^{-1} = \left(\sup_{\lambda\in\mathbb{R}}\|(i\lambda+\mathbb{Q}_mL_{\alpha,m})^{-1}\|_{\mathcal{Y}_m\to\mathcal{Y}_m}\right)^{-1}
\end{align*}
which was introduced in \cite{GaGaNi09} and called the pseudospectral bound of $\mathbb{Q}_mL_{\alpha,m}$ in \cite{IbMaMa19}. Then we get an estimate for the semigroup generated by $\mathbb{Q}_mL_{\alpha,m}$ by combining the estimate for the pseudospectral bound of $\mathbb{Q}_mL_{\alpha,m}$ with an abstract theorem for an $m$-accretive operator on a weighted Hilbert space which is a version of the Gearhart--Pr\"{u}ss type theorem shown by Wei \cite{Wei21}. In order to use the abstract results, we need to confirm several assumptions for $A_m$ and $\Lambda_m$. The main effort is to verify a coercive estimate for $\mathbb{Q}_m(\mu-\Lambda_m)$ on $\mathcal{Y}_m$ with $\mu\in\mathbb{R}$ (see Section \ref{SS:Ko_As4}). This coercive estimate involves a loss of derivative, but thanks to a small factor in front it can be controlled by the smoothing effect of $-A_m$ in the estimate for the pseudospectral bound of $\mathbb{Q}_mL_{\alpha,m}$. When $|\mu|\geq1$, we easily get the coercive estimate by taking the $L^2(S^2)$-inner product of $\mathbb{Q}_m(\mu-\Lambda_m)u$ with $(I+6\Delta^{-1})u$ for $u\in\mathcal{Y}_m$ and using the coerciveness of $I+6\Delta^{-1}$ on $\mathcal{Y}_m$ (see Lemmas \ref{L:Ko_High} and \ref{L:Ko_Middle}). On the other hand, the proof is more involved when $|\mu|<1$ (see Lemma \ref{L:Ko_Low}). In this case, noting that $u\in\mathcal{Y}_m$ is of the form $u=U(\theta)e^{im\varphi}$, we analyze an ordinary differential equation (ODE) with variable $\theta$ corresponding to $\mathbb{Q}_m(\mu-\Lambda_m)$ and prove the coercive estimate by a contradiction argument after giving auxiliary statements. The main difficulty comes from the nonlocalities of the orthogonal projection $\mathbb{Q}_m$ which is nontrivial when $|m|=1,2$ and of the inverse operator $\Delta^{-1}$ appearing in $\Lambda_m$. These nonlocalities affect the analysis for different values of $\mu$. The nonlocality of $\mathbb{Q}_m$ with $|m|=1,2$ is relevant to the case where $\mu$ is closed to zero, which is the eigenvalue of $\Lambda_m$. On the other hand, the nonlocality of $\Delta^{-1}$ causes a difficulty when $\mu$ is close to $\pm1$, which are the critical values of $\cos\theta$ appearing in $\Lambda_m$, i.e. the derivative $(\cos\theta)'=-\sin\theta$ vanishes when $\cos\theta=\pm1$. Hence we can deal with these difficulties separately by using a contradiction argument. Moreover, since $I-\mathbb{Q}_m$ is the orthogonal projection onto the kernel of $\Lambda_m$ which is spanned by the smooth function $Y_2^m$, we can handle the nonlocality of $\mathbb{Q}_m$ by introducing a suitable auxiliary function which involves $Y_2^m$. Also, when $\mu$ is close to $\pm1$, the use of a contradiction argument enables us to focus on analysis of functions concentrating around the critical points of $\cos\theta$, i.e. $\theta=0,\pi$, for which the nonlocal term consisting of $\Delta^{-1}$ essentially becomes a small order by the smoothing effect of $\Delta^{-1}$. By these facts we can verify the coercive estimate for $|\mu|<1$, but the actual proof requires very long and careful calculations.

The coercive estimate for $\mathbb{Q}_m(\mu-\Lambda_m)$ given in this paper is basically the same as the one in the plane case \cite{IbMaMa19}. In the sphere case, the surface measure on $S^2$ is of the form $\sin\theta\,d\theta\,d\varphi$ under the spherical coordinate system. The weight function $\sin\theta$ vanishes at the critical points $\theta=0,\pi$ of $\cos\theta$ appearing in $\Lambda_m$, so one may expect to have a better coercive estimate than the one in the plane case when $|\mu|<1$ and $\mu$ is close to $\pm1$, but we cannot get such an estimate. In fact, for $\theta_\mu=\arccos\mu$ with $\mu$ close to $\pm1$, a small parameter $\varepsilon>0$, and a function on $S^2$ of the form $u=U(\theta)e^{im\varphi}$, we have a better bound $\varepsilon\sin\theta_\mu$ due to the weight function $\sin\theta$ when we estimate the $L^2$-norm of $u$ on a narrow band on $S^2$ corresponding to $\theta\in(\theta_\mu-\varepsilon,\theta_\mu+\varepsilon)$ by the $L^\infty$-norm of $U$ on $(\theta_\mu-\varepsilon,\theta_\mu+\varepsilon)$. However, we also have a factor $1/\sin\theta_\mu$ when we use an interpolation type inequality which bounds the $L^\infty$-norm $U$ on $(\theta_\mu-\varepsilon,\theta_\mu+\varepsilon)$ by the $L^2$- and $H^1$-norms of $u$ on $S^2$ (see Lemma \ref{L:Int_Fix}), so the resulting coercive estimate is the same as in the plane case.

Lastly, let us explain difference from the plane case \cite{IbMaMa19} in the proof of the coercive estimate. The authors of \cite{IbMaMa19} derived coercive estimates for $\mu-\Lambda_m$ on $\mathcal{X}_m$ with $\mu\in\mathbb{R}$ away from zero and for $\mathbb{Q}_m(\mu-\Lambda_m)$ on $\mathcal{Y}_m$ with $\mu\in\mathbb{R}$ close to zero in order to avoid the nonlocality of $\mathbb{Q}_m$ when $\mu$ is away from zero. In this paper, however, we deal with $\mathbb{Q}_m(\mu-\Lambda_m)$ for all $\mu\in\mathbb{R}$ since the nonlocality of $\mathbb{Q}_m$ can be handled by taking the inner product in $\mathcal{Y}_m$ when $|\mu|\geq1$ and by introducing an auxiliary function when $|\mu|<1$. Also, we encounter an additional difficulty due to the larger coefficient of the nonlocal operator $\Delta^{-1}$ in $\Lambda_m$. To prove the coercive estimate when $|\mu|<1$ and $\mu$ is close to $1$ (or $-1$), we analyze an auxiliary function $w$ on $S^2$ of the form $w=W(\theta)e^{im\varphi}$ and derive an estimate for the $L^2$-norm of $w$ on a subset of $S^2$ corresponding to $\theta\in(\theta_\mu,\pi)$ (or $\theta\in(0,\theta_\mu)$) with $\theta_\mu=\arccos\mu$ by using an ODE for $W$ (see Lemma \ref{L:KL_wLS}). In the proof of the estimate, we need to show $J+K\geq C\int_{\theta_\mu}^\pi|W(\theta)|^2\sin\theta\,d\theta$ with a constant $C>0$, where
\begin{align*}
  J &= \int_{\theta_\mu}^\pi(\mu-\cos\theta)\left(|W'(\theta)|^2+\frac{m^2}{\sin^2\theta}|W(\theta)|^{2}\right)\sin\theta\,d\theta, \\
  K &= 5\int_{\theta_\mu}^\pi|W(\theta)|^2\cos\theta\sin\theta\,d\theta.
\end{align*}
Here $J$ is always nonnegative but $K$ may become negative by the presence of $\cos\theta$. Moreover, the coefficient $5$ of $K$ comes from the coefficient $6$ of $\Delta^{-1}$ in $\Lambda_m$. In the plane case, one also needs to deal with integrals similar to $J$ and $K$, but the coefficient of $K$ becomes $1/2$ since the coefficient of $\Delta^{-1}$ in $\Lambda_m$ is $1$. Thus one can easily get an estimate just by using the zeroth order term of $W$ in $J$. In our case, however, the use of only the zeroth order term does not work for $|m|=1$ since the coefficient $5$ of $K$ is too large. To overcome this difficulty, we take into account the first order term of $W$ in $J$. In fact, it is natural to use both the zeroth and first order terms since
\begin{align*}
  |\nabla w|^2 = |W'(\theta)|^2+\frac{m^2}{\sin^2\theta}|W(\theta)|^2
\end{align*}
for the gradient of $w=W(\theta)e^{im\varphi}$ on $S^2$. In the actual proof, we split $K$ into the integrals $K_1$ and $K_2$ over $(\theta_\mu,\Theta)$ and $(\Theta,\pi)$ with $\Theta\in(\theta_\mu,\pi/2)$ and estimate them separately. To $K_1$ we just use $\cos\theta\geq\cos\Theta$ for $\theta\in(\theta_\mu,\Theta)$. Also, we carry out integration by parts for $K_2$ and apply Young's inequality to get the integral of a function involving the term $|\nabla w|^2$. Then, taking an appropriate $\Theta$ and estimating the integrand, we obtain an estimate for $K_2$ which gives the lower bound of $J+K$ when combined with the estimate for $K_1$.

The rest of this paper is organized as follows. In Section \ref{S:Pre} we fix notations and give auxiliary inequalities. Section \ref{S:Kol} is devoted to the study of the linearized operator for the two-jet Kolmogorov type flow. We provide settings and basic results in Section \ref{SS:Ko_Basic}, verify coercive estimates for $\mathbb{Q}_m(\mu-\Lambda_m)$ in Section \ref{SS:Ko_As4}, and derive estimates for the semigroup generated by $L_{\alpha,m}$ and prove Theorem \ref{T:OL_Dec} in Section \ref{SS:Ko_Semi}. Section \ref{S:Abst} gives abstract results used in Section \ref{S:Kol}. In Section \ref{S:DG} we show basic formulas of differential geometry on $S^2$.

\section{Preliminaries} \label{S:Pre}
In this section we fix notations and give auxiliary inequalities. We also refer to Section \ref{S:DG} for some notations and basic formulas of differential geometry.

Let $S^2$ be the unit sphere in $\mathbb{R}^3$ equipped with the Riemannian metric induced by the Euclidean metric of $\mathbb{R}^3$. We denote by $\theta$ and $\varphi$ the colatitude and the longitude so that $S^2$ is parametrized by $\mathbf{x}(\theta,\varphi)=(\sin\theta\cos\varphi,\sin\theta\sin\varphi,\cos\theta)$ for $\theta\in[0,\pi]$ and $\varphi\in[0,2\pi)$. For a (complex-valued) function $u$ on $S^2$, we sometimes abuse the notation
\begin{align*}
  u(\theta,\varphi) = u(\sin\theta\cos\varphi,\sin\theta\sin\varphi,\cos\theta), \quad \theta\in[0,\pi],\,\varphi\in[0,2\pi)
\end{align*}
when no confusion may occur. Thus the integral of $u$ over $S^2$ is given by
\begin{align*}
  \int_{S^2}u\,d\mathcal{H}^2 = \int_0^{2\pi}\left(\int_0^\pi u(\theta,\varphi)\sin\theta\,d\theta\right)\,d\varphi,
\end{align*}
where $\mathcal{H}^k$ is the Hausdorff measure of dimension $k\in\mathbb{N}$. As usual, we set
\begin{align*}
  (u,v)_{L^2(S^2)} = \int_{S^2}u\bar{v}\,d\mathcal{H}^2, \quad \|u\|_{L^2(S^2)} = (u,u)_{L^2(S^2)}^{1/2}, \quad u, v\in L^2(S^2),
\end{align*}
where $\bar{v}$ is the complex conjugate of $v$, and write $H^k(S^2)$, $k\in\mathbb{Z}_{\geq0}$ for the Sobolev spaces of $L^2$ functions on $S^2$ with $H^0(S^2)=L^2(S^2)$ (see \cite{Aub98}).

Let $\nabla$ and $\Delta$ be the gradient and Laplace--Beltrami operators on $S^2$. It is well known (see e.g. \cite{VaMoKh88,Tes14}) that $\lambda_n=n(n+1)$ is an eigenvalue of $-\Delta$ with multiplicity $2n+1$ for each $n\in\mathbb{Z}_{\geq0}$ and the corresponding eigenvectors are the spherical harmonics
\begin{align} \label{E:SpHa}
  Y_n^m = Y_n^m(\theta,\varphi) = \sqrt{\frac{2n+1}{4\pi}\frac{(n-m)!}{(n+m)!}} \, P_n^m(\cos\theta)e^{im\varphi}, \quad m=0,\pm1,\dots,\pm n.
\end{align}
Here $P_n^0$, $n\in\mathbb{Z}_{\geq0}$ are the Legendre polynomials defined as
\begin{align*}
  P_n^0(s) = P_n(s) = \frac{1}{2^nn!}\frac{d^n}{ds^n}(s^2-1)^n, \quad s\in(-1,1)
\end{align*}
and the associated Legendre functions $P_n^m$, $n\in\mathbb{Z}_{\geq0}$, $|m|\leq n$ are given by
\begin{align*}
  P_n^m(s) =
  \begin{cases}
    (-1)^m(1-s^2)^{m/2}\displaystyle\frac{d^m}{ds^m}P_n(s), &m\geq0, \\
    (-1)^{|m|}\displaystyle\frac{(n-|m|)!}{(n+|m|)!}P_n^{|m|}(s), &m = -|m| < 0
  \end{cases}
\end{align*}
so that $Y_n^{-m}=(-1)^m\overline{Y_n^m}$ (see \cite{Led72,DLMF}). Moreover, the set of all $Y_n^m$ forms an orthonormal basis of $L^2(S^2)$, i.e. $u=\sum_{n=0}^\infty\sum_{m=-n}^n(u,Y_n^m)_{L^2(S^2)}Y_n^m$ for each $u\in L^2(S^2)$. It is also known that the recurrence relation
\begin{align*}
  (n-m+1)P_{n+1}^m(s)-(2n+1)sP_n^m(s)+(n+m)P_{n-1}^m(s) = 0
\end{align*}
holds (see \cite[(7.12.12)]{Led72}) and thus (see also \cite[Section 5.7]{VaMoKh88})
\begin{align} \label{E:Y_Rec}
  \cos\theta\,Y_n^m = a_n^mY_{n-1}^m+a_{n+1}^mY_{n+1}^m, \quad a_n^m = \sqrt{\frac{(n-m)(n+m)}{(2n-1)(2n+1)}}
\end{align}
for $n\in\mathbb{Z}_{\geq0}$ and $|m|\leq n$, where we consider $Y_{|m|-1}^m\equiv0$. Note that the superscript $m$ of $a_n^m$ just corresponds to that of $Y_n^m$ and does not mean the $m$-th power.

Let $L_0^2(S^2)$ be the space of $L^2$ functions on $S^2$ with vanishing mean, i.e.
\begin{align*}
  L_0^2(S^2) = \left\{u\in L^2(S^2) ~\middle|~ \int_{S^2}u\,d\mathcal{H}^2 = 0\right\} = \{u\in L^2(S^2) \mid (u,Y_0^0)_{L^2(S^2)}=0\}.
\end{align*}
Then $\Delta$ is invertible and self-adjoint as a linear operator
\begin{align*}
  \Delta\colon D_{L_0^2(S^2)}(\Delta) \subset L_0^2(S^2) \to L_0^2(S^2), \quad D_{L_0^2(S^2)}(\Delta) = L_0^2(S^2)\cap H^2(S^2).
\end{align*}
Also, for $s\in\mathbb{R}$, the operator $(-\Delta)^s$ is defined on $L_0^2(S^2)$ by
\begin{align} \label{E:Def_Laps}
  (-\Delta)^su = \sum_{n=1}^\infty\sum_{m=-n}^n\lambda_n^s(u,Y_n^m)_{L^2(S^2)}Y_n^m, \quad u \in L_0^2(S^2).
\end{align}
We easily observe by a density argument and integration by parts that
\begin{align} \label{E:Lap_Half}
  \|(-\Delta)^{1/2}u\|_{L^2(S^2)} = \|\nabla u\|_{L^2(S^2)}, \quad u\in L_0^2(S^2)\cap H^1(S^2).
\end{align}
By this relation and Poincar\'{e}'s inequality (see e.g. \cite[Corollary 4.3]{Aub98})
\begin{align} \label{E:Poin}
  \|u\|_{L^2(S^2)} \leq C\|\nabla u\|_{L^2(S^2)}, \quad u\in L_0^2(S^2)\cap H^1(S^2)
\end{align}
we have the norm equivalence
\begin{align} \label{E:H1_Equiv}
  \|(-\Delta)^{1/2}u\|_{L^2(S^2)} \leq \|u\|_{H^1(S^2)} \leq C\|(-\Delta)^{1/2}u\|_{L^2(S^2)}, \quad u\in L_0^2(S^2)\cap H^1(S^2).
\end{align}
Let $u$ be a function on $S^2$. We write $u=U(\theta)e^{im\varphi}$ if $u$ is of the form
\begin{align*}
  u(\sin\theta\cos\varphi,\sin\theta\sin\varphi,\cos\theta) = U(\theta)e^{im\varphi}, \quad \theta\in[0,\pi], \, \varphi\in[0,2\pi)
\end{align*}
with some function $U$ of the colatitude $\theta$ and $m\in\mathbb{Z}$. In this case,
\begin{align} \label{E:L2_Mode}
  \begin{aligned}
    \|u\|_{L^2(S^2)}^2 &= 2\pi\int_0^\pi|U(\theta)|^2\sin\theta\,d\theta, \\
    \|\nabla u\|_{L^2(S^2)}^2 &= 2\pi\int_0^\pi\left(|U'(\theta)|^2+\frac{m^2}{\sin^2\theta}|U(\theta)|^2\right)\sin\theta\,d\theta, \\
    \|\nabla^2u\|_{L^2(S^2)}^2 &= 2\pi\int_0^\pi\{|U''(\theta)|^2+R(U)\}\sin\theta\,d\theta,
  \end{aligned}
\end{align}
by \eqref{E:Re_SC}, where $U'=\frac{dU}{d\theta}$ and $\nabla^2u$ is the covariant Hessian of $u$, and
\begin{align*}
  R(U) = \frac{2m^2}{\sin^2\theta}\left|U'(\theta)-\frac{\cos\theta}{\sin\theta}U(\theta)\right|^2+\frac{1}{\sin^4\theta}|U'(\theta)\sin\theta\cos\theta-m^2U(\theta)|^2 \geq 0.
\end{align*}
If $u=U(\theta)e^{im\varphi}\in L^2(S^2)$, then $(u,Y_n^{m'})_{L^2(S^2)}=0$ for $m'\neq m$. In particular, if $m\neq0$, then $u\in L_0^2(S^2)$ and we can apply \eqref{E:Lap_Half}--\eqref{E:H1_Equiv} to $u$. Also, if $u=U(\theta)e^{im\varphi}\in H^1(S^2)$ for $m\in\mathbb{Z}$, then $U\in H_{loc}^1(0,\pi)\subset C(0,\pi)$ by \eqref{E:L2_Mode}. We use these facts without mention.

For a function $u$ on $S^2$ and $m\in\mathbb{Z}$ we define a function $\mathcal{P}_mu$ on $S^2$ by
\begin{align} \label{E:Def_Proj}
  \mathcal{P}_mu(\theta,\varphi) = \frac{e^{im\varphi}}{2\pi}\int_0^{2\pi}u(\theta,\phi)e^{-im\phi}\,d\phi.
\end{align}
Note that $\mathcal{P}_mu=u$ and $\mathcal{P}_{m'}u=0$ for $m'\neq m$ if $u=U(\theta)e^{im\varphi}$.

The following results are shown in \cite[Section 2]{Miu21pre}.

\begin{lemma} \label{L:NS_Pole}
  If $u=U(\theta)e^{im\varphi}\in C(S^2)$ with $m\in\mathbb{Z}\setminus\{0\}$, then $U(0)=U(\pi)=0$.
\end{lemma}

\begin{lemma} \label{L:Proj_Bdd}
  For $\theta_1,\theta_2\in[0,\pi]$ with $\theta_1\leq\theta_2$ let
  \begin{align} \label{E:Def_Band}
    S^2(\theta_1,\theta_2)=\{(x_1,x_2,x_3)\in S^2 \mid x_3 = \cos\theta,\,\theta\in(\theta_1,\theta_2) \}.
  \end{align}
  Then for each $m\in\mathbb{Z}$ we have
  \begin{align} \label{E:Proj_Bdd}
    \begin{alignedat}{2}
      \|\mathcal{P}_mu\|_{L^2(S^2(\theta_1,\theta_2))} &\leq \|u\|_{L^2(S^2(\theta_1,\theta_2))}, &\quad &u\in L^2(S^2(\theta_1,\theta_2)), \\
      \|\nabla\mathcal{P}_mu\|_{L^2(S^2(\theta_1,\theta_2))} &\leq \|\nabla u\|_{L^2(S^2(\theta_1,\theta_2))}, &\quad &u\in H^1(S^2(\theta_1,\theta_2)).
    \end{alignedat}
  \end{align}
\end{lemma}

\begin{lemma} \label{L:Linf}
  For $m\in\mathbb{Z}\setminus\{0\}$ let $u=U(\theta)e^{im\varphi}\in H^1(S^2)$. Then
  \begin{align} \label{E:Linf}
    \|u\|_{L^\infty(S^2)}^2 = \|U\|_{L^\infty(0,\pi)}^2 \leq \frac{1}{\pi|m|}\|(-\Delta)^{1/2}u\|_{L^2(S^2)}^2.
  \end{align}
\end{lemma}

\begin{lemma} \label{L:Hardy}
  Let $\mu\in(-1,1)$ and $\theta_\mu=\arccos\mu\in(0,\pi)$. Then
  \begin{align} \label{E:Hardy}
    \int_{\theta_1}^{\theta_2}\left|\frac{U(\theta)\sqrt{\sin\theta}-U(\theta_\mu)\sqrt{\sin\theta_\mu}}{\mu-\cos\theta}\right|^2\,d\theta \leq \frac{16}{\pi\sin^2\theta_\mu}\|\nabla u\|_{L^2(S^2(\theta_1,\theta_2))}^2
  \end{align}
  for all $\theta_1\in[0,\theta_\mu]$, $\theta_2\in[\theta_\mu,\pi]$, and $u=U(\theta)e^{im\varphi}\in H^1(S^2(\theta_1,\theta_2))$ with $m\in\mathbb{Z}\setminus\{0\}$, where $S^2(\theta_1,\theta_2)$ is given by \eqref{E:Def_Band}.
\end{lemma}

Let us give interpolation type inequalities for $u=U(\theta)e^{im\varphi}$.

\begin{lemma} \label{L:Inter}
  Let $\beta\in(0,\pi/2)$ and $u=U(\theta)e^{im\varphi}\in H^1(S^2)$ with $m\in\mathbb{Z}\setminus\{0\}$. Then
  \begin{align}
    |U(\theta)|^2 &\leq \frac{1}{2\pi\sin\beta}\|u\|_{L^2(S^2)}\left(2\|(-\Delta)^{1/2}u\|_{L^2(S^2)}+\frac{1}{\pi-2\beta}\|u\|_{L^2(S^2)}\right), \label{E:Inter_U} \\
    |U(\theta)|^2 &\leq \frac{1}{2\pi\sin^2\beta}\|M_{\sin\theta}u\|_{L^2(S^2)}\left(2\|(-\Delta)^{1/2}u\|_{L^2(S^2)}+\frac{1}{\pi-2\beta}\|u\|_{L^2(S^2)}\right) \label{E:Inter_sinU}
  \end{align}
  for all $\theta\in(\beta,\pi-\beta)$, where $M_{\sin\theta}u=\sin\theta\,U(\theta)e^{im\varphi}$.
\end{lemma}

\begin{proof}
  By the mean value theorem for integrals, we have
  \begin{align*}
    \int_\beta^{\pi-\beta}|U(\vartheta)|^2\sin\vartheta\,d\vartheta = (\pi-2\beta)|U(\theta_\beta)|^2\sin\theta_\beta
  \end{align*}
  with some $\theta_\beta\in(\beta,\pi-\beta)$. Then we observe that
  \begin{align} \label{Pf_In:Alpha}
    |U(\theta_\beta)|^2 &= \frac{1}{(\pi-2\beta)\sin\theta_\beta}\int_\beta^{\pi-\beta}|U(\vartheta)|^2\sin\vartheta\,d\vartheta \leq \frac{1}{2\pi(\pi-2\beta)\sin\beta}\|u\|_{L^2(S^2)}^2
  \end{align}
  by $\sin\theta_\beta\geq\sin\beta$ and \eqref{E:L2_Mode}. Let $\theta\in(\beta,\pi-\beta)$. Then
  \begin{align} \label{Pf_In:Diff}
    \begin{aligned}
      |U(\theta)|^2-|U(\theta_\beta)|^2 &= \int_{\theta_\beta}^\theta\frac{d}{d\vartheta}|U(\vartheta)|^2\,d\vartheta \leq 2\int_\beta^{\pi-\beta}|U(\vartheta)||U'(\vartheta)|\,d\vartheta \\
      &\leq 2\left(\int_\beta^{\pi-\beta}\frac{|U(\vartheta)|^2}{\sin\vartheta}\,d\vartheta\right)^{1/2}\left(\int_\beta^{\pi-\beta}|U'(\vartheta)|^2\sin\vartheta\,d\vartheta\right)^{1/2}.
    \end{aligned}
  \end{align}
  Moreover, by $\sin\vartheta\geq\sin\beta$ for $\vartheta\in(\beta,\pi-\beta)$, \eqref{E:Lap_Half}, and \eqref{E:L2_Mode},
  \begin{align} \label{Pf_In:U_L2}
    \begin{gathered}
      \int_\beta^{\pi-\beta}\frac{|U(\vartheta)|^2}{\sin\vartheta}\,d\vartheta \leq \frac{1}{\sin^2\beta}\int_\beta^{\pi-\beta}|U(\vartheta)|^2\sin\vartheta\,d\vartheta \leq \frac{1}{2\pi\sin^2\beta}\|u\|_{L^2(S^2)}^2, \\
      \int_\beta^{\pi-\beta}|U'(\vartheta)|^2\sin\vartheta\,d\vartheta \leq \frac{1}{2\pi}\|\nabla u\|_{L^2(S^2)}^2 =\frac{1}{2\pi}\|(-\Delta)^{1/2}u\|_{L^2(S^2)}^2.
    \end{gathered}
  \end{align}
  Hence \eqref{E:Inter_U} follows from \eqref{Pf_In:Alpha}--\eqref{Pf_In:U_L2}. Also, for $M_{\sin\theta}u=\sin\theta\,U(\theta)e^{im\varphi}$,
  \begin{align*}
    \int_\beta^{\pi-\beta}|U(\vartheta)|^2\sin\vartheta\,d\vartheta &\leq \frac{1}{\sin\beta}\int_\beta^{\pi-\beta}|\sin\vartheta\,U(\vartheta)||U(\vartheta)|\sin\vartheta\,d\vartheta \\
    &\leq \frac{1}{2\pi\sin\beta}\|M_{\sin\theta}u\|_{L^2(S^2)}\|u\|_{L^2(S^2)}
  \end{align*}
  by $\sin\vartheta\geq\sin\beta$ for $\vartheta\in(\beta,\pi-\beta)$ and \eqref{E:L2_Mode}, and
  \begin{align*}
    \int_\beta^{\pi-\beta}\frac{|U(\vartheta)|^2}{\sin\vartheta}\,d\vartheta \leq \frac{1}{\sin^4\beta}\int_\beta^{\pi-\beta}|\sin\vartheta\,U(\vartheta)|^2\sin\vartheta\,d\vartheta \leq \frac{1}{2\pi\sin^4\beta}\|M_{\sin\theta}u\|_{L^2(S^2)}^2.
  \end{align*}
  We apply these estimates, $\sin\theta_\beta\geq\sin\beta$, and the second line of \eqref{Pf_In:U_L2} to \eqref{Pf_In:Alpha} and \eqref{Pf_In:Diff}, and then combine the resulting inequalities to obtain \eqref{E:Inter_sinU}.
\end{proof}

\begin{lemma} \label{L:Int_Fix}
  For $m\in\mathbb{Z}\setminus\{0\}$ let $u=U(\theta)e^{im\varphi}\in H^1(S^2)$. Also, let $\theta_\mu\in(0,\pi)$ and
  \begin{align*}
    0 < \varepsilon \leq \frac{1}{2}\sin\theta_\mu < \frac{1}{2}\min\{\theta_\mu,\pi-\theta_\mu\}.
  \end{align*}
  Then there exists a constant $C>0$ independent of $m$, $u$, $\theta_\mu$, and $\varepsilon$ such that
  \begin{align}
    |U(\theta)|^2 &\leq \frac{C}{\sin\theta_\mu}\|u\|_{L^2(S^2)}\|(-\Delta)^{1/2}u\|_{L^2(S^2)}, \label{E:IF_U} \\
    |U(\theta)|^2 &\leq \frac{C}{\sin^2\theta_\mu}\|M_{\sin\theta}u\|_{L^2(S^2)}\|(-\Delta)^{1/2}u\|_{L^2(S^2)} \label{E:IF_sinU}
  \end{align}
  for all $\theta\in(\theta_\mu-\varepsilon,\theta_\mu+\varepsilon)$, where $M_{\sin\theta}u=\sin\theta\,U(\theta)e^{im\varphi}$.
\end{lemma}

\begin{proof}
  When $\theta_\mu\in(0,\pi/6]$, we set $\beta=\theta_\mu-\varepsilon$. Then since
  \begin{align*}
    \theta_\mu+\varepsilon \leq \frac{3}{2}\theta_\mu \leq \frac{\pi}{4}, \quad \pi-\beta \geq \pi-2\beta \geq \pi-2\theta_\mu \geq \frac{2\pi}{3},
  \end{align*}
  we have $(\theta_\mu-\varepsilon,\theta_\mu+\varepsilon)\subset(\beta,\pi-\beta)$. Also, by Taylor's theorem for $\sin\theta$ at $\theta=\theta_0$,
  \begin{align*}
    \sin\beta = \sin(\theta_\mu-\varepsilon) \geq \sin\theta_\mu-\varepsilon \geq \frac{1}{2}\sin\theta_\mu.
  \end{align*}
  On the other hand, when $\theta_\mu\in[\pi/6,\pi/2]$,
  \begin{align*}
    \theta_\mu-\varepsilon \geq \frac{\theta_\mu}{2} \geq \frac{\pi}{12}, \quad \theta_\mu+\varepsilon \leq \frac{3}{2}\theta_\mu \leq \frac{3}{4}\pi \leq \frac{11}{12}\pi.
  \end{align*}
  Hence we set $\beta=\pi/12$ to get $(\theta_\mu-\varepsilon,\theta_\mu+\varepsilon)\subset(\beta,\pi-\beta)$ and
  \begin{align*}
    \pi-2\beta = \frac{5}{6}\pi, \quad \sin\beta = \sin\frac{\pi}{12} \geq \sin\frac{\pi}{12}\sin\theta_\mu.
  \end{align*}
  Therefore, in both cases $\theta_\mu\in(0,\pi/6]$ and $\theta_\mu\in[\pi/6,\pi/2]$, we have \eqref{E:IF_U} and \eqref{E:IF_sinU} by applying \eqref{E:Inter_U} and \eqref{E:Inter_sinU} with the above $\beta$ and using \eqref{E:H1_Equiv}. Similarly, we can prove \eqref{E:IF_U} and \eqref{E:IF_sinU} when $\theta_\mu\geq\pi/2$ by considering the cases $\theta_\mu\in[\pi/2,5\pi/6]$ and $\theta_\mu\in[5\pi/6,\pi)$ separately.
\end{proof}

\section{Analysis of the linearized operator} \label{S:Kol}
In this section we study the linearized operator for the two-jet Kolmogorov type flow.

\subsection{Settings and basic results} \label{SS:Ko_Basic}
Let $I$ be the identity operator and $M_f$ the multiplication operator by a function $f$ on $S^2$, i.e. $M_fu=fu$ for a function $u$ on $S^2$. We define linear operators $A$ and $\Lambda$ on $L_0^2(S^2)$ by
\begin{alignat*}{2}
  A &= \Delta+2, &\quad D_{L_0^2(S^2)}(A) &= L_0^2(S^2)\cap H^2(S^2), \\
  \Lambda &= -i\partial_\varphi M_{\cos\theta}B, &\quad D_{L_0^2(S^2)}(\Lambda) &= \{u\in L_0^2(S^2) \mid \partial_\varphi M_{\cos\theta}u\in L_0^2(S^2)\},
\end{alignat*}
where $B=I+6\Delta^{-1}$ on $L_0^2(S^2)$. Then $A$ is self-adjoint in $L_0^2(S^2)$. Also, $\Lambda$ is densely defined in $L_0^2(S^2)$ since its domain contains the dense subspace $L_0^2(S^2)\cap H^1(S^2)$ of $L_0^2(S^2)$. Since $M_{\cos\theta}B$ is a bounded operator on $L^2(S^2)$ (note that it does not map $L_0^2(S^2)$ into itself) and $\partial_\varphi$ is a closed operator from $L^2(S^2)$ into $L_0^2(S^2)$, we see that $\Lambda$ is closed in $L_0^2(S^2)$. Moreover, $\Lambda$ is $A$-compact in $L_0^2(S^2)$ since $H^2(S^2)$ is compactly embedded in $H^1(S^2)$ and
\begin{align} \label{E:Ko_A_Apr}
  \|u\|_{H^2(S^2)} \leq C\|\Delta u\|_{L^2(S^2)} \leq C\left(\|Au\|_{L^2(S^2)}+2\|u\|_{L^2(S^2)}\right), \quad u\in D(A)
\end{align}
by the elliptic regularity theorem. For $n\geq1$ and $|m|\leq n$ we have
\begin{align} \label{E:Lap_Y}
  \Delta Y_n^m = -\lambda_nY_n^m, \quad  BY_n^m = \left(1-\frac{6}{\lambda_n}\right)Y_n^m, \quad \partial_\varphi Y_n^m = imY_n^m.
\end{align}
By this equality and \eqref{E:Y_Rec}, we see that (here $Y_{|m|-1}^m\equiv0$)
\begin{align} \label{E:Ko_AL_Y}
  AY_n^m = (-\lambda_n+2)u, \quad \Lambda Y_n^m = m\left(1-\frac{6}{\lambda_n}\right)(a_n^mY_{n-1}^m+a_{n+1}^mY_{n+1}^m).
\end{align}
In particular, $\Lambda Y_2^m=0$ for $|m|=0,1,2$. Let
\begin{align*}
  \mathcal{X} = \{u\in L_0^2(S^2) \mid (u,Y_n^0)_{L^2(S^2)}=(u,Y_1^m)_{L^2(S^2)}=0, \, n\geq1, \, |m|=0,1\}.
\end{align*}
Then $\mathcal{X}$ is a closed subspace of $L_0^2(S^2)$ and each $u\in\mathcal{X}$ is expressed as
\begin{align} \label{E:Ko_u_X}
    u = \sum_{m\in\mathbb{Z}\setminus\{0\}}\sum_{n\geq N_m}(u,Y_n^m)_{L^2(S^2)}Y_n^m, \quad N_m = \max\{2,|m|\},
\end{align}
and we observe by \eqref{E:Ko_AL_Y} (in particular $\Lambda Y_2^m=0$) and \eqref{E:Ko_u_X} that $\mathcal{X}$ is invariant under the actions of $A$ and $\Lambda$. For the sake of simplicity, we write $A$ and $\Lambda$ for their restrictions on $\mathcal{X}$ with domains $D_{\mathcal{X}}(A)=\mathcal{X}\cap D_{L_0^2(S^2)}(A)$ and $D_{\mathcal{X}}(\Lambda)=\mathcal{X}\cap D_{L_0^2(S^2)}(\Lambda)$.

The follows lemmas are proved in \cite[Section 5]{Miu21pre}.

\begin{lemma} \label{L:Ko_ALam}
  The operator $A$ is self-adjoint in $\mathcal{X}$ and satisfies
  \begin{align} \label{E:Ko_A_Po}
    (-Au,u)_{L^2(S^2)} \geq 4\|u\|_{L^2(S^2)}^2, \quad u\in D_{\mathcal{X}}(A).
  \end{align}
  Also, $\Lambda$ is densely defined, closed, and $A$-compact in $\mathcal{X}$.
\end{lemma}

\begin{lemma} \label{L:Ko_NLam}
  For $u\in\mathcal{X}$ let $\mathbb{Q}u=u-\sum_{|m|=1,2}(u,Y_2^m)_{L^2(S^2)}Y_2^m$. Then
  \begin{align} \label{E:B2_LowB}
    \frac{1}{2}\|\mathbb{Q}u\|_{L^2(S^2)} \leq \|Bu\|_{L^2(S^2)} \leq \|u\|_{L^2(S^2)}, \quad u\in\mathcal{X}
  \end{align}
  and the kernel of $\Lambda$ in $\mathcal{X}$ is $N_{\mathcal{X}}(\Lambda)=\mathrm{span}\{Y_2^m\mid |m|=1,2\}$. Thus $\mathbb{Q}$ is the orthogonal projection from $\mathcal{X}$ onto $\mathcal{Y} = N_{\mathcal{X}}(\Lambda)^\perp=\{u\in\mathcal{X}\mid (u,Y_2^m)_{L^2(S^2)}=0,\,|m|=1,2\}$.
\end{lemma}

By Lemma \ref{L:Ko_ALam} and a perturbation theory of semigroups (see \cite{EngNag00}), the operator $L_\alpha=A-i\alpha\Lambda$ with domain $D_\mathcal{X}(L_\alpha)=D_{\mathcal{X}}(A)$ generates an analytic semigroup $\{e^{tL_\alpha}\}_{t\geq0}$ in $\mathcal{X}$ for each $\alpha\in\mathbb{R}$. Our aim is to study the decay rate of $e^{tL_\alpha}$ in terms of $\alpha$.

For $m\in\mathbb{Z}\setminus\{0\}$, let $\mathcal{X}_m=\mathcal{P}_m\mathcal{X}$ and $\mathcal{Y}_m=\mathcal{P}_m\mathcal{Y}$ with $\mathcal{P}_m$ given by \eqref{E:Def_Proj}. Then $\mathcal{X}_m$ and $\mathcal{Y}_m$ are closed subspaces of $\mathcal{X}$ and $\mathcal{Y}$, respectively, and the orthogonal decompositions
\begin{align*}
  \mathcal{X} = \oplus_{m\in\mathbb{Z}\setminus\{0\}}\mathcal{X}_m, \quad \mathcal{Y} = \oplus_{m\in\mathbb{Z}\setminus\{0\}}\mathcal{Y}_m
\end{align*}
hold. Moreover, since $u\in\mathcal{X}_m$ is of the form
\begin{align} \label{E:Ko_Xm}
  u = \sum_{n\geq N_m}(u,Y_n^m)_{L^2(S^2)}Y_n^m, \quad N_m = \max\{2,|m|\},
\end{align}
the subspace $\mathcal{X}_m$ is invariant under the actions of $A$ and $\Lambda$ by \eqref{E:Ko_AL_Y} (in particular $\Lambda Y_2^m=0$). Hence $L_\alpha=A-i\alpha\Lambda$ is diagonalized as
\begin{align*}
  L_\alpha = \oplus_{m\in\mathbb{Z}\setminus\{0\}}L_{\alpha,m}, \quad L_{\alpha,m}=L_\alpha|_{\mathcal{X}_m}.
\end{align*}
Moreover, we see by \eqref{E:Y_Rec}, \eqref{E:Lap_Y}, and \eqref{E:Ko_Xm} that $L_{\alpha,m}$ is expressed as
\begin{align*}
  L_{\alpha,m} = A_m-i\alpha m\Lambda_m, \quad D_{\mathcal{X}_m}(L_{\alpha,m}) = \mathcal{X}_m\cap H^2(S^2),
\end{align*}
where $A_m$ and $\Lambda_m$ are linear operators on $\mathcal{X}_m$ defined by
\begin{alignat*}{2}
  A_m &= A|_{\mathcal{X}_m} = (\Delta+2)|_{\mathcal{X}_m}, &\quad D_{\mathcal{X}_m}(A_m) &= \mathcal{X}_m\cap H^2(S^2), \\
  \Lambda_m &= M_{\cos\theta}B|_{\mathcal{X}_m} = M_{\cos\theta}(I+6\Delta^{-1})|_{\mathcal{X}_m}, &\quad D_{\mathcal{X}_m}(\Lambda_m) &= \mathcal{X}_m.
\end{alignat*}
Note that, by $|\cos\theta|\leq1$, \eqref{E:Lap_Y}, \eqref{E:Ko_Xm}, and $0\leq1-6/\lambda_n\leq1$ for $n\geq2$,
\begin{align} \label{E:Ko_Lm_Bo}
  \|\Lambda_mu\|_{L^2(S^2)} \leq \|Bu\|_{L^2(S^2)} \leq \|u\|_{L^2(S^2)}, \quad u\in\mathcal{X}_m.
\end{align}
Also, $\mathcal{Y}_m$ is a closed subspace of $\mathcal{X}_m$ and each $u\in\mathcal{Y}_m$ is of the form
\begin{align} \label{E:Ko_Ym}
  u = \sum_{n\geq N'_m}(u,Y_n^m)_{L^2(S^2)}Y_n^m, \quad N'_m = \max\{3,|m|\}.
\end{align}
Let $\mathbb{Q}_m$ be the orthogonal projection from $\mathcal{X}_m$ onto $\mathcal{Y}_m$. Then
\begin{align} \label{E:Ko_Qm}
  \mathbb{Q}_mu = \mathbb{Q}u =
  \begin{cases}
    u-(u,Y_2^m)_{L^2(S^2)}Y_2^m, &|m|=1,2, \\
    u, &|m|\geq3
  \end{cases}
\end{align}
for $u\in\mathcal{X}_m$ by \eqref{E:Ko_Xm} and \eqref{E:Ko_Ym}. We intend to derive decay estimates for the semigroup generated by $L_{\alpha,m}$ in $\mathcal{X}_m$ by applying abstract results given in Section \ref{S:Abst}. To this end, we verify Assumptions \ref{As:A}--\ref{As:La_02} and \ref{As:Est} for $A_m$ and $\Lambda_m$.

\begin{lemma} \label{L:Kom_As13}
  Let $m\in\mathbb{Z}\setminus\{0\}$. Then $A_m$ and $\Lambda_m$ satisfy Assumptions \ref{As:A}--\ref{As:La_02} in $\mathcal{X}_m$ with $B_{1,m}=M_{\cos\theta}$ on $\mathcal{H}_m=L^2(S^2)$ and $B_{2,m}=B|_{\mathcal{X}_m}=(I+6\Delta^{-1})|_{\mathcal{X}_m}$ on $\mathcal{X}_m$.
\end{lemma}

Note that $B_{2,m}$ maps $\mathcal{X}_m$ into itself by \eqref{E:Lap_Y} and \eqref{E:Ko_Xm}. Also, as in \eqref{E:Def_Laps}, we set
\begin{align} \label{E:Def_Ams}
  (-A_m)^su = \sum_{n\geq N_m}(\lambda_n-2)^s(u,Y_n^m)_{L^2(S^2)}Y_n^m, \quad N_m=\max\{2,|m|\}
\end{align}
for $s\in\mathbb{R}$ and $u\in\mathcal{X}_m$ of the form \eqref{E:Ko_Xm}.

\begin{proof}
  Since $\mathcal{X}_m$ is a closed subspace of $\mathcal{X}$, Lemma \ref{L:Ko_ALam} implies that $A_m$ satisfies Assumption \ref{As:A} in $\mathcal{X}_m$. The operator $\Lambda_m$ is densely defined, closed, and $A_m$-compact in $\mathcal{X}_m$ since it is a bounded operator on $\mathcal{X}_m$, the inequality \eqref{E:Ko_A_Apr} holds, and the embedding $H^2(S^2)\hookrightarrow L^2(S^2)$ is compact. If $\Lambda_mu=0$ for $u\in\mathcal{X}_m$, then $B_{2,m}u=0$ since $B_{1,m}=M_{\cos\theta}$ is injective on $L^2(S^2)$. Thus, by \eqref{E:B2_LowB}, \eqref{E:Ko_Xm}, and $B_{2,m}Y_2^m=0$ for $|m|=1,2$, we see that
  \begin{align} \label{Pf_K13:Ker}
    N_{\mathcal{X}_m}(\Lambda_m) = N_{\mathcal{X}_m}(B_{2,m}) =
    \begin{cases}
      \{cY_2^m \mid c\in\mathbb{C}\}, &|m|=1,2, \\
      \{0\}, &|m|\geq3
    \end{cases}
  \end{align}
  and we get $\mathcal{Y}_m=N_{\mathcal{X}_m}(\Lambda_m)^\perp$ by \eqref{E:Ko_Ym} and \eqref{Pf_K13:Ker}. Also, we have $\mathbb{Q}_mA_m\subset A_m\mathbb{Q}_m$ in $\mathcal{X}_m$ since $Y_2^m$ is smooth on $S^2$, the relations \eqref{E:Ko_AL_Y} and \eqref{E:Ko_Qm} hold, and $A_m$ is self-adjoint in $\mathcal{X}_m$. Since $u\in\mathcal{Y}_m$ is of the form \eqref{E:Ko_Ym}, we see by \eqref{E:Lap_Y}, \eqref{E:Ko_AL_Y}, \eqref{E:Def_Ams}, and $1-6/\lambda_n\geq 1/2$ for $n\geq 3$ that the inequalities \eqref{E:uB2u} and \eqref{E:AB2u} hold for $A_m$ and $B_{2,m}$ with constant $C=1/2$. Hence Assumptions \ref{As:La_01} and \ref{As:La_02} are valid.
\end{proof}

\begin{lemma} \label{L:B2_XmYm}
  For $m\in\mathbb{Z}\setminus\{0\}$ and $k=0,1$ we have
  \begin{alignat}{2}
    \|(-\Delta)^{k/2}B_{2,m}u\|_{L^2(S^2)}^2 &\leq \|(-\Delta)^{k/2}u\|_{L^2(S^2)}^2, &\quad u&\in\mathcal{X}_m\cap H^k(S^2), \label{E:B2m_Xmk} \\
    \frac{1}{4}\|u\|_{L^2(S^2)}^2 &\leq \|B_{2,m}u\|_{L^2(S^2)}^2, &\quad u&\in\mathcal{Y}_m. \label{E:B2m_Ym}
  \end{alignat}
\end{lemma}

\begin{proof}
  Since $u\in\mathcal{X}_m$ is of the form \eqref{E:Ko_Xm}, we have \eqref{E:B2m_Xmk} by \eqref{E:Def_Laps}, \eqref{E:Lap_Y}, and $|1-6/\lambda_n|\leq1$ for $n\geq2$. Also, if $u\in\mathcal{Y}_m$, then $u$ is of the form \eqref{E:Ko_Ym} and thus \eqref{E:B2m_Ym} follows from \eqref{E:Lap_Y} and $1-6/\lambda_n\geq1/2$ for $n\geq3$.
\end{proof}

The next result is crucial for the proof of Lemma \ref{L:Ko_Low} below.

\begin{lemma} \label{L:NoEi_Lam}
  Let $m\in\mathbb{Z}\setminus\{0\}$. Then $\Lambda_m$ in $\mathcal{X}_m$ has no eigenvalues in $\mathbb{C}\setminus\{0\}$.
\end{lemma}

\begin{proof}
  The statement follows from \cite[Theorem 4.1]{Miu21pre}.
\end{proof}

\subsection{Verification of Assumption \ref{As:Est}} \label{SS:Ko_As4}
Next we show that $A_m$ and $\Lambda_m$ satisfy Assumption \ref{As:Est} in several long steps. In what follows, we frequently use the fact that each function $u$ in $\mathcal{X}_m=\mathcal{P}_m\mathcal{X}$ is of the form $u=U(\theta)e^{im\varphi}$ (see \eqref{E:Def_Proj}).

\begin{lemma} \label{L:Ko_B3}
  Let $m\in\mathbb{Z}\setminus\{0\}$. For $u\in\mathcal{X}_m\cap H^1(S^2)$ we have
  \begin{align} \label{E:Ko_Amh}
    \frac{2}{3}\|(-\Delta)^{1/2}u\|_{L^2(S^2)}^2 \leq \|(-A_m)^{1/2}u\|_{L^2(S^2)}^2 \leq \|(-\Delta)^{1/2}u\|_{L^2(S^2)}^2.
  \end{align}
  Also, let $B_{3,m}=M_{\sin\theta}$ on $\mathcal{B}_m=L^2(S^2)$. Then for $u\in D_{\mathcal{X}_m}(A_m)$ we have
  \begin{align} \label{E:Ko_B3}
    \bigl|\mathrm{Im}(A_mu,\Lambda_mu)_{L^2(S^2)}\bigr| \leq \frac{\sqrt{6}}{2}\|(-A_m)^{1/2}u\|_{L^2(S^2)}\|B_{3,m}u\|_{L^2(S^2)}.
  \end{align}
\end{lemma}

\begin{proof}
  For $u\in\mathcal{X}_m\cap H^1(S^2)$ we have \eqref{E:Ko_Amh} by \eqref{E:Def_Laps} and \eqref{E:Def_Ams} since $u$ is of the form \eqref{E:Ko_Xm} and $2\lambda_n/3\leq\lambda_n-2\leq\lambda_n$ for $n\geq2$.

  Next let $u\in D_{\mathcal{X}_m}(A_m)$ and $v=\Delta^{-1}u$. Then
  \begin{multline*}
    (A_mu,\Lambda_mu)_{L^2(S^2)} = (\Delta u,x_3u)_{L^2(S^2)}+6(\Delta u,x_3v)_{L^2(S^2)} \\
    +2(u,x_3u)_{L^2(S^2)}+12(u,x_3v)_{L^2(S^2)}
  \end{multline*}
  with $x_3=\cos\theta$ on $S^2$. Moreover, we have
  \begin{align*}
    (\Delta u,x_3u)_{L^2(S^2)} &= -(\nabla u,x_3\nabla u)_{L^2(S^2)}-(\nabla u,u\nabla x_3)_{L^2(S^2)}, \\
    (\Delta u,x_3v)_{L^2(S^2)} &= -2(u,x_3v)_{L^2(S^2)}+2(u,\nabla x_3\cdot\nabla v)_{L^2(S^2)}+(u,x_3u)_{L^2(S^2)}
  \end{align*}
  by applying integration by parts once or twice, where we also used $\Delta x_3=-2x_3$ by \eqref{E:Re_SC} for $x_3=\cos\theta$ and $\Delta v=u$ in the second line. Hence
  \begin{multline*}
    (A_mu,\Lambda_mu)_{L^2(S^2)} = -(\nabla u,x_3\nabla u)_{L^2(S^2)}+8(u,x_3u)_{L^2(S^2)} \\
    -(\nabla u,u\nabla x_3)_{L^2(S^2)}+12(u,\nabla x_3\cdot\nabla v)_{L^2(S^2)}.
  \end{multline*}
  Moreover, since
  \begin{gather*}
    -(\nabla u,x_3\nabla u)_{L^2(S^2)}+8(u,x_3u)_{L^2(S^2)} = \int_{S^2}x_3(-|\nabla u|^2+8|u|^2)\,d\mathcal{H}^2 \in\mathbb{R}, \\
    \mathrm{Im}(u,\nabla x_3\cdot\nabla v)_{L^2(S^2)} =\mathrm{Im}\int_{S^2}u\nabla x_3\cdot\nabla\bar{v}\,d\mathcal{H}^2 = -\mathrm{Im}(\nabla v,u\nabla x_3)_{L^2(S^2)},
  \end{gather*}
  we have
  \begin{align} \label{Pf_KB3:Im}
    \begin{aligned}
      |\mathrm{Im}(A_mu,\Lambda_mu)_{L^2(S^2)}| &= |\mathrm{Im}(\nabla(u+12v),u\nabla x_3)_{L^2(S^2)}| \\
      &\leq \|\nabla(u+12v)\|_{L^2(S^2)}\|u\nabla x_3\|_{L^2(S^2)}.
    \end{aligned}
  \end{align}
  Now we recall that $v=\Delta^{-1}u$ and $u$ is of the form \eqref{E:Ko_Xm} to get
  \begin{align} \label{Pf_KB3:nab}
    \|\nabla(u+12v)\|_{L^2(S^2)}^2 &= \|(-\Delta)^{1/2}(I+12\Delta^{-1})u\|_{L^2(S^2)}^2 \leq \|(-\Delta)^{1/2}u\|_{L^2(S^2)}^2
  \end{align}
  by \eqref{E:Def_Laps}, \eqref{E:Lap_Half}, and $|1-12/\lambda_n|\leq1$ for $n\geq2$. Also, since $|\nabla x_3|=\sin\theta$ on $S^2$ by $x_3=\cos\theta$ and \eqref{E:Re_SC}, we have $\|u\nabla x_3\|_{L^2(S^2)}=\|B_{3,m}u\|_{L^2(S^2)}$ with $B_{3,m}=M_{\sin\theta}$. Thus we get \eqref{E:Ko_B3} by this equality, \eqref{E:Ko_Amh}, \eqref{Pf_KB3:Im}, and \eqref{Pf_KB3:nab}.
\end{proof}

Let us verify \eqref{E:u_h} and \eqref{E:B3u_h} in Assumption \ref{As:Est}. We consider three cases for $\mu\in\mathbb{R}$:
\begin{align*}
  |\mu| > 1, \quad 1-\kappa\delta^2 \leq |\mu| \leq 1+\kappa\delta^2, \quad |\mu| \leq 1-\kappa\delta^2.
\end{align*}
Here $\kappa>0$ is a small constant given in Lemma \ref{L:Ko_Middle} below and $\delta\in(0,1]$. In what follows, we write $C$ for a general positive constant independent of $m$, $u$, $\mu$, and $\delta$. Also, we frequently use the expressions $\cos\theta=x_3$ on $S^2$ and $\Lambda_mu=x_3B_{2,m}u$ for $u\in\mathcal{X}_m$.

\begin{lemma} \label{L:Ko_High}
  There exists a constant $C>0$ such that
  \begin{align}
    \|u\|_{L^2(S^2)}^2 &\leq \frac{C}{(|\mu|-1)^2}\|\mathbb{Q}_m(\mu-\Lambda_m)u\|_{L^2(S^2)}^2, \label{E:Ko_Hiu} \\
    \|M_{\sin\theta}u\|_{L^2(S^2)}^2 &\leq \frac{C}{|\mu|-1}\|\mathbb{Q}_m(\mu-\Lambda_m)u\|_{L^2(S^2)}^2 \label{E:Ko_HiS}
  \end{align}
  for all $m\in\mathbb{Z}\setminus\{0\}$, $u\in\mathcal{Y}_m$, and $\mu\in\mathbb{R}$ satisfying $|\mu|>1$.
\end{lemma}

\begin{proof}
  For $u\in\mathcal{Y}_m$ let $v=\Delta^{-1}u$ and $f=\mathbb{Q}_m(\mu-\Lambda_m)u$. Then
  \begin{align*}
    (f,B_{2,m}u)_{L^2(S^2)} &= \bigl((\mu-\Lambda_m)u,\mathbb{Q}_mB_{2,m}u\bigr)_{L^2(S^2)} = \bigl((\mu-x_3)B_{2,m}u-6\mu v,B_{2,m}u\bigr)_{L^2(S^2)}
  \end{align*}
  since $B_{2,m}u\in\mathcal{Y}_m$ by \eqref{E:Lap_Y} and \eqref{E:Ko_Ym}. Moreover,
  \begin{align*}
    (v,B_{2,m}u)_{L^2(S^2)} = (v,\Delta v)_{L^2(S^2)}+6\|v\|_{L^2(S^2)}^2 = -\|\nabla v\|_{L^2(S^2)}^2+6\|v\|_{L^2(S^2)}^2
  \end{align*}
  by $B_{2,m}u=\Delta v+6v$ and integration by parts. Hence
  \begin{align*}
    \bigl((\mu-x_3)B_{2,m}u,B_{2,m}u\bigr)_{L^2(S^2)}+6\mu\Bigl(\|\nabla v\|_{L^2(S^2)}^2-6\|v\|_{L^2(S^2)}^2\Bigr) = (f,B_{2,m}u)_{L^2(S^2)}
  \end{align*}
  and we divide both sides by $\mu$ and take the real part to get
  \begin{multline} \label{Pf_KH:Equ}
    \int_{S^2}\left(1-\frac{x_3}{\mu}\right)|B_{2,m}u|^2\,d\mathcal{H}^2+6\Bigl(\|\nabla v\|_{L^2(S^2)}^2-6\|v\|_{L^2(S^2)}^2\Bigr) \\
    = \frac{1}{\mu}\mathrm{Re}(f,B_{2,m}u)_{L^2(S^2)}.
  \end{multline}
  Moreover, since $u\in\mathcal{Y}_m$ is of the form \eqref{E:Ko_Ym} and $1-6/\lambda_n\geq1/2$ for $n\geq3$, we observe by $v=\Delta^{-1}u$, \eqref{E:Def_Laps}, and \eqref{E:Lap_Half} that
  \begin{align} \label{Pf_KH:nabla}
    \begin{aligned}
      \|\nabla v\|_{L^2(S^2)}^2-6\|v\|_{L^2(S^2)}^2 &= \|(-\Delta)^{-1/2}u\|_{L^2(S^2)}^2-6\|\Delta^{-1}u\|_{L^2(S^2)}^2 \\
      &\geq \frac{1}{2}\|(-\Delta)^{-1/2}u\|_{L^2(S^2)}^2.
    \end{aligned}
  \end{align}
  Applying \eqref{Pf_KH:nabla} and $1-x_3/\mu\geq1-1/|\mu|=(|\mu|-1)/|\mu|$ on $S^2$ to the left-hand side of \eqref{Pf_KH:Equ}, and using H\"{o}lder's and Young's inequalities to the right-hand side, we obtain
  \begin{multline*}
    \frac{|\mu|-1}{|\mu|}\|B_{2,m}u\|_{L^2(S^2)}^2+3\|(-\Delta)^{-1/2}u\|_{L^2(S^2)}^2 \\
    \leq \frac{1}{2}\cdot\frac{|\mu|-1}{|\mu|}\|B_{2,m}u\|_{L^2(S^2)}^2+\frac{1}{2}\cdot\frac{1}{|\mu|(|\mu|-1)}\|f\|_{L^2(S^2)}^2
  \end{multline*}
  and thus (note that $|\mu|>1$)
  \begin{align} \label{Pf_KH:B2_Up}
    \|B_{2,m}u\|_{L^2(S^2)}^2+\frac{6|\mu|}{|\mu|-1}\|(-\Delta)^{-1/2}u\|_{L^2(S^2)}^2 \leq \frac{1}{(|\mu|-1)^2}\|f\|_{L^2(S^2)}^2.
  \end{align}
  Hence we get \eqref{E:Ko_Hiu} by \eqref{E:B2m_Ym} and \eqref{Pf_KH:B2_Up}. To prove \eqref{E:Ko_HiS}, we see that
  \begin{align*}
    (f,u)_{L^2(S^2)} = \bigl((\mu-x_3)u-6x_3v,\mathbb{Q}_mu\bigr)_{L^2(S^2)} = \bigl((\mu-x_3)u-6x_3v,u\bigr)_{L^2(S^2)}
  \end{align*}
  by $u\in\mathcal{Y}_m$. Taking the real part of this equality, we get
  \begin{align} \label{Pf_KH:sinu}
    \int_{S^2}(\mu-x_3)|u|^2\,d\mathcal{H}^2 = 6\mathrm{Re}(x_3v,u)_{L^2(S^2)}+\mathrm{Re}(f,u)_{L^2(S^2)}.
  \end{align}
  Moreover, since $|\nabla x_3|=\sin\theta\leq1$ on $S^2$ by $x_3=\cos\theta$ and \eqref{E:Re_SC},
  \begin{align} \label{Pf_KH:x3v}
    \begin{aligned}
      \bigl|(x_3v,u)_{L^2(S^2)}\bigr| &= \bigl|(x_3v,\Delta v)_{L^2(S^2)}\bigr| = \bigl|(\nabla(x_3v),\nabla v)_{L^2(S^2)}\bigr| \\
      &\leq C\|v\|_{H^1(S^2)}^2 \leq C\|(-\Delta)^{1/2}v\|_{L^2(S^2)}^2 = C\|(-\Delta)^{-1/2}u\|_{L^2(S^2)}^2
    \end{aligned}
  \end{align}
  by \eqref{E:H1_Equiv} and $u=\Delta v$. Noting that
  \begin{align*}
    |\mu-x_3| =
    \begin{cases}
      \mu-x_3, &\mu>1, \\
      -\mu+x_3, &\mu<-1,
    \end{cases}
    \quad \sin^2\theta = 1-x_3^2 \leq 2|\mu-x_3| \quad\text{on}\quad S^2,
  \end{align*}
  we deduce from \eqref{Pf_KH:sinu} and \eqref{Pf_KH:x3v} that
  \begin{align*}
    \|M_{\sin\theta}u\|_{L^2(S^2)}^2 &\leq 2\int_{S^2}|\mu-x_3||u|^2\,d\mathcal{H}^2 \leq C\bigl(\bigl|(x_3v,u)_{L^2(S^2)}\bigr|+\bigl|(f,u)_{L^2(S^2)}\bigr|\bigl) \\
    &\leq C\Bigl(\|(-\Delta)^{-1/2}u\|_{L^2(S^2)}^2+\|f\|_{L^2(S^2)}\|u\|_{L^2(S^2)}\Bigr)
  \end{align*}
  and applying \eqref{E:Ko_Hiu} and \eqref{Pf_KH:B2_Up} to the last line we obtain \eqref{E:Ko_HiS}.
\end{proof}

\begin{lemma} \label{L:Ko_Middle}
  There exist constants $\kappa\in(0,1/2)$ and $C>0$ such that
  \begin{align}
    \|u\|_{L^2(S^2)}^2 &\leq C\left(\delta^{-4}\|\mathbb{Q}_m(\mu-\Lambda_m)u\|_{L^2(S^2)}^2+\frac{\delta^2}{|m|}\|(-\Delta)^{1/2}u\|_{L^2(S^2)}^2\right), \label{E:Ko_Midu} \\
    \|M_{\sin\theta}u\|_{L^2(S^2)}^2 &\leq C\left(\delta^{-2}\|\mathbb{Q}_m(\mu-\Lambda_m)u\|_{L^2(S^2)}^2+\frac{\delta^4}{|m|}\|(-\Delta)^{1/2}u\|_{L^2(S^2)}^2\right) \label{E:Ko_MidS}
  \end{align}
  for all $m\in\mathbb{Z}\setminus\{0\}$, $u\in\mathcal{Y}_m\cap H^1(S^2)$, $\delta\in(0,1]$, and $\mu\in\mathbb{R}$ satisfying
  \begin{align} \label{E:Ko_MidMu}
    1-\kappa\delta^2 \leq |\mu| \leq 1+\kappa\delta^2.
  \end{align}
\end{lemma}

\begin{proof}
  Let $\kappa\in(0,1/2)$ be a constant which will be determined later, and let $\delta\in(0,1]$ and $\mu\in\mathbb{R}$ satisfy \eqref{E:Ko_MidMu}. We may assume $\mu>0$ since the case $\mu<0$ is similarly handled. Let $v=\Delta^{-1}u$ and $f=\mathbb{Q}_m(\mu-\Lambda_m)u$. Since $u\in\mathcal{Y}_m$, we can use \eqref{Pf_KH:Equ} in the proof of Lemma \ref{L:Ko_High} to get
  \begin{multline*}
    \int_{S^2}(1-x_3)|B_{2,m}u|^2\,d\mathcal{H}^2+6\mu\Bigl(\|\nabla v\|_{L^2(S^2)}^2-6\|v\|_{L^2(S^2)}^2\Bigr) \\
    = (1-\mu)\|B_{2,m}u\|_{L^2(S^2)}^2+\mathrm{Re}(f,B_{2,m}u)_{L^2(S^2)}.
  \end{multline*}
  Then, by $\mu\geq1-\kappa\delta^2\geq1/2$, $1-\mu\leq\kappa\delta^2$, \eqref{Pf_KH:nabla}, and Young's inequality,
  \begin{multline} \label{Pf_KM:x3w}
    \int_{S^2}(1-x_3)|B_{2,m}u|^2\,d\mathcal{H}^2+\|(-\Delta)^{-1/2}u\|_{L^2(S^2)}^2 \\
    \leq C\Bigl(\kappa\delta^2\|B_{2,m}u\|_{L^2(S^2)}^2+\delta^{-2}\|f\|_{L^2(S^2)}^2\Bigr).
  \end{multline}
  Since $B_{2,m}u\in\mathcal{X}_m$, we can write $B_{2,m}u=U_{2,m}(\theta)e^{im\varphi}$. Then
  \begin{align*}
    \int_{S^2}(1-x_3)|B_{2,m}u|^2\,d\mathcal{H}^2 &= 2\pi\int_0^\pi(1-\cos\theta)|U_{2,m}(\theta)|^2\sin\theta\,d\theta \\
    &\geq 2\pi\int_\delta^\pi(1-\cos\theta)|U_{2,m}(\theta)|^2\sin\theta\,d\theta
  \end{align*}
  by \eqref{E:L2_Mode}. Moreover, since $1-\cos\theta\geq1-\cos\delta$ for $\theta\in(\delta,\pi)$ and
  \begin{align*}
    \int_\delta^\pi|U_{2,m}(\theta)|^2\sin\theta\,d\theta &= \int_0^\pi|U_{2,m}(\theta)|^2\sin\theta\,d\theta-\int_0^\delta|U_{2,m}(\theta)|^2\sin\theta\,d\theta \\
    &\geq \frac{1}{2\pi}\|B_{2,m}u\|_{L^2(S^2)}^2-\frac{1}{\pi|m|}(1-\cos\delta)\|(-\Delta)^{1/2}B_{2,m}u\|_{L^2(S^2)}^2
  \end{align*}
  by \eqref{E:L2_Mode}, \eqref{E:Linf}, and $\int_0^\delta\sin\theta\,d\theta=1-\cos\delta$, it follows that
  \begin{multline*}
    \int_{S^2}(1-x_3)|B_{2,m}u|^2\,d\mathcal{H}^2 \\
    \geq (1-\cos\delta)\|B_{2,m}u\|_{L^2(S^2)}^2-\frac{2}{|m|}(1-\cos\delta)^2\|(-\Delta)^{1/2}B_{2,m}u\|_{L^2(S^2)}^2.
  \end{multline*}
  We further observe that $\delta^2/4\leq1-\cos\delta\leq\delta^2/2$ by Taylor's theorem for $\cos\theta$ at $\theta=0$ and $\delta\in(0,1]\subset(0,\pi/3)$. Therefore,
  \begin{align*}
    \int_{S^2}(1-x_3)|B_{2,m}u|^2\,d\mathcal{H}^2 \geq \frac{1}{4}\left(\delta^2\|B_{2,m}u\|_{L^2(S^2)}^2-\frac{2\delta^4}{|m|}\|(-\Delta)^{1/2}B_{2,m}u\|_{L^2(S^2)}^2\right).
  \end{align*}
  We apply this inequality to \eqref{Pf_KM:x3w} and use \eqref{E:B2m_Xmk} with $k=1$ to find that
  \begin{multline} \label{Pf_KM:delw}
    \frac{1}{4}\delta^2\|B_{2,m}u\|_{L^2(S^2)}^2+\|(-\Delta)^{-1/2}u\|_{L^2(S^2)}^2 \\
    \leq C_1\kappa\delta^2\|B_{2,m}u\|_{L^2(S^2)}^2+C_2\left(\delta^{-2}\|f\|_{L^2(S^2)}^2+\frac{\delta^4}{|m|}\|(-\Delta)^{1/2}u\|_{L^2(S^2)}^2\right)
  \end{multline}
  with some constants $C_1,C_2>0$ independent of $m$, $u$, $\delta$, $\mu$, and $\kappa$. Now we define
  \begin{align} \label{Pf_KM:kappa}
    \kappa = \frac{1}{8}\min\left\{\frac{1}{C_1},1\right\} \in \left(0,\frac{1}{2}\right).
  \end{align}
  Then since $1/4-C_1\kappa\geq1/8$, it follows from \eqref{Pf_KM:delw} that
  \begin{multline} \label{Pf_KM:w_fin}
    \delta^2\|B_{2,m}u\|_{L^2(S^2)}^2+\|(-\Delta)^{-1/2}u\|_{L^2(S^2)}^2 \\
    \leq C\left(\delta^{-2}\|f\|_{L^2(S^2)}^2+\frac{\delta^4}{|m|}\|(-\Delta)^{1/2}u\|_{L^2(S^2)}^2\right)
  \end{multline}
  and we get \eqref{E:Ko_Midu} by \eqref{E:B2m_Ym} and \eqref{Pf_KM:w_fin}. Also, since
  \begin{align*}
    \int_{S^2}(1-x_3)|u|^2\,d\mathcal{H}^2 &= (1-\mu)\|u\|_{L^2(S^2)}^2+\int_{S^2}(\mu-x_3)|u|^2\,d\mathcal{H}^2 \\
    &\leq \kappa\delta^2\|u\|_{L^2(S^2)}^2+C\Bigl(\|(-\Delta)^{-1/2}u\|_{L^2(S^2)}^2+\|f\|_{L^2(S^2)}\|u\|_{L^2(S^2)}\Bigr)
  \end{align*}
  by $1-\mu\leq\kappa\delta^2$, \eqref{Pf_KH:sinu}, and \eqref{Pf_KH:x3v}, and since $\sin^2\theta=1-x_3^2\leq2(1-x_3)$ on $S^2$,
  \begin{align*}
    \|M_{\sin\theta}u\|_{L^2(S^2)}^2 &\leq 2\int_{S^2}(1-x_3)|u|^2\,d\mathcal{H}^2 \\
    &\leq C\Bigl(\delta^2\|u\|_{L^2(S^2)}^2+\|(-\Delta)^{-1/2}u\|_{L^2(S^2)}^2+\|f\|_{L^2(S^2)}\|u\|_{L^2(S^2)}\Bigr).
  \end{align*}
  Thus we apply \eqref{E:Ko_Midu} and \eqref{Pf_KM:w_fin} to the last line to get \eqref{E:Ko_MidS}.
\end{proof}

\begin{lemma} \label{L:Ko_Low}
  Let $\kappa$ be the constant given by \eqref{Pf_KM:kappa}. There exists a constant $C>0$ such that
  \begin{multline} \label{E:Ko_Lowu}
    \delta^2\|u\|_{L^2(S^2)}^2+\|(-\Delta)^{1/2}v\|_{L^2(S^2)}^2+\frac{1}{1-|\mu|}\|w\|_{L^2(S^2)}^2 \\
    \leq C\left(\delta^{-2}\|\mathbb{Q}_m(\mu-\Lambda_m)u\|_{L^2(S^2)}^2+\frac{\delta^6}{1-|\mu|}\|(-\Delta)^{1/2}u\|_{L^2(S^2)}^2\right)
  \end{multline}
  and
  \begin{align} \label{E:Ko_LowS}
    \|M_{\sin\theta}u\|_{L^2(S^2)}^2 \leq C\Bigl(\delta^{-2}\|\mathbb{Q}_m(\mu-\Lambda_m)u\|_{L^2(S^2)}^2+\delta^2(1-|\mu|)\|(-\Delta)^{1/2}u\|_{L^2(S^2)}^2\Bigr)
  \end{align}
  for all $m\in\mathbb{Z}\setminus\{0\}$, $u\in\mathcal{Y}_m\cap H^1(S^2)$, $\delta\in(0,1]$, and $\mu\in\mathbb{R}$ satisfying $|\mu|\leq 1-\kappa\delta^2$.
  Here
  \begin{align} \label{E:KL_vw}
    \begin{aligned}
      v &= \Delta^{-1}u \in \mathcal{Y}_m\cap H^3(S^2), \\
      w &= 6\mu v-\frac{a_3^m}{2}(u,Y_3^m)_{L^2(S^2)}Y_2^m \in \mathcal{X}_m\cap H^3(S^2),
    \end{aligned}
  \end{align}
  where $a_3^m$ is given by \eqref{E:Y_Rec} if $|m|=1,2$ and we consider $a_3^m=0$ if $|m|\geq3$ and $Y_n^m\equiv0$ if $|m|>n$.
\end{lemma}

The proof of \eqref{E:Ko_Lowu} relies on a contradiction argument. In order to focus on the part of getting a contradiction in that proof, we derive auxiliary estimates in the following lemmas. We assume $\mu\geq0$ in the sequel, since the case $\mu\leq0$ can be handled similarly.

In Lemmas \ref{L:KL_eps}--\ref{L:KL_wLS} below, let $\kappa$ be the constant given by \eqref{Pf_KM:kappa} and $\delta\in(0,1]$, $\mu\in[0,1-\kappa\delta^2]$, and $\theta_\mu=\arccos\mu\in(0,\pi/2]$. Note that
\begin{align} \label{E:KL_sinmu}
  (1-\mu)^{1/2} \leq \sin\theta_\mu \leq 1, \quad \kappa\delta^2 \leq \sin^2\theta_\mu \leq 2(1-\mu)
\end{align}
by $\sin^2\theta_\mu=1-\mu^2=(1-\mu)(1+\mu)$ and $\kappa\delta^2\leq1-\mu$. For $\varepsilon>0$ let
\begin{align} \label{E:KL_Sme}
  S_{\mu,\varepsilon}^2 = \{(x_1,x_2,x_3)\in S^2 \mid x_3=\cos\theta, \, \theta\in(\theta_\mu-\varepsilon,\theta_\mu+\varepsilon)\},
\end{align}
i.e. $S_{\mu,\varepsilon}^2=S^2(\theta_\mu-\varepsilon,\theta_\mu+\varepsilon)$ in \eqref{E:Def_Band}. Also, let
\begin{align} \label{E:KL_fR}
  \begin{gathered}
    f = \delta^{-1}\mathbb{Q}_m(\mu-\Lambda_m)u \in \mathcal{Y}_m\cap H^1(S^2), \\
    R(u,\delta,\mu) = \|f\|_{L^2(S^2)}+\frac{\delta^3}{(1-\mu)^{1/2}}\|(-\Delta)^{1/2}u\|_{L^2(S^2)}
  \end{gathered}
\end{align}
and $v$ and $w$ be given by \eqref{E:KL_vw} for $u\in\mathcal{Y}_m\cap H^1(S^2)$ with $m\in\mathbb{Z}\setminus\{0\}$. Hereafter we write $C$ for a general positive constant independent of $\delta$, $\mu$, $m$, and $u$.

\begin{lemma} \label{L:KL_eps}
  Let $\varepsilon$ be a positive constant satisfying
  \begin{align} \label{E:KL_eps}
    0 < \varepsilon \leq \frac{1}{2}\sin\theta_\mu < \frac{1}{2}\theta_\mu.
  \end{align}
  Also, let $\theta\in(\theta_\mu-\varepsilon,\theta_\mu+\varepsilon)$. Then
  \begin{gather}
    \frac{1}{2}\sin\theta_\mu \leq \sin\theta \leq \frac{3}{2}\sin\theta_\mu, \quad |\mu-\cos\theta| \leq \frac{5}{4}\varepsilon\sin\theta_\mu, \label{E:KL_theta} \\
    \frac{3}{4}\varepsilon\sin\theta_\mu \leq |\mu-\cos(\theta_\mu\pm\varepsilon)| \leq \frac{5}{4}\varepsilon\sin\theta_\mu. \label{E:KL_cosep}
  \end{gather}
  Note that $\cos(\theta_\mu+\varepsilon)<\mu<\cos(\theta_\mu-\varepsilon)$. Also,
  \begin{align} \label{E:KL_mucos}
    |\mu-\cos\theta| \geq \frac{1}{2}|\theta-\theta_\mu|(\sin\theta+\sin\theta_\mu), \quad \theta\in[0,\pi].
  \end{align}
\end{lemma}

\begin{proof}
  For $\theta\in(\theta_\mu-\varepsilon,\theta_\mu+\varepsilon)$ we have \eqref{E:KL_theta} since
  \begin{align*}
    |\sin\theta-\sin\theta_\mu| \leq \varepsilon \leq \frac{1}{2}\sin\theta_\mu, \quad |\mu-\cos\theta| \leq \varepsilon\sin\theta+\frac{1}{2}\varepsilon^2 \leq \frac{5}{4}\varepsilon\sin\theta_\mu
  \end{align*}
  by Taylor's theorem for $\sin\theta$ and $\cos\theta$ at $\theta=\theta_\mu$ and \eqref{E:KL_eps}. Also,
  \begin{align*}
    |\cos(\theta_\mu\pm\varepsilon)-(\mu\mp\varepsilon\sin\theta_\mu)| \leq \frac{1}{2}\varepsilon^2 \leq \frac{1}{4}\varepsilon\sin\theta_\mu
  \end{align*}
  by Taylors' theorem and \eqref{E:KL_eps}. Hence \eqref{E:KL_cosep} follows. For $\theta\in[0,\pi]$ since
  \begin{align*}
    |\mu-\cos\theta| = \int_{\theta_{\min}}^{\theta_{\max}}\sin\vartheta\,d\vartheta = (\theta_{\max}-\theta_{\min})\int_0^1\sin\bigl((1-t)\theta_{\min}+t\theta_{\max}\bigr)\,dt,
  \end{align*}
  where $\theta_{\min}=\min\{\theta_\mu,\theta\}$ and $\theta_{\max}=\max\{\theta_\mu,\theta\}$, and since $\sin\vartheta$ is concave for $\vartheta\in[0,\pi]$, we easily find that \eqref{E:KL_mucos} holds.
\end{proof}

\begin{lemma} \label{L:KL_w_Eq}
  We have
  \begin{alignat}{2}
    (\mu-x_3)B_{2,m}u &= w+\delta f &\quad\text{on}\quad &S^2, \label{E:KL_B2u_Eq} \\
    (\mu-x_3)(\Delta w+6w) &= 6\mu(w+\delta f) &\quad\text{on}\quad &S^2. \label{E:KL_w_Eq}
  \end{alignat}
\end{lemma}

\begin{proof}
  If $|m|\geq3$, then $\mathbb{Q}_m=I$ and $w=6\mu v$. Thus
  \begin{align*}
    \delta f=(\mu-\Lambda_m)u=(\mu-x_3)B_{2,m}u-w
  \end{align*}
  by $\Lambda_mu=x_3B_{2,m}u$ and $B_{2,m}u=(I+6\Delta^{-1})u=u+6v$. Hence \eqref{E:KL_B2u_Eq} holds.

  Let $|m|=1,2$ and $u\in\mathcal{Y}_m$ be of the form \eqref{E:Ko_Ym} with $N'_m=3$. Then it follows from \eqref{E:Y_Rec}, \eqref{E:Lap_Y} and $\lambda_3=12$ that
  \begin{align*}
    \Lambda_mu &= \sum_{n\geq 3}\left(1-\frac{6}{\lambda_n}\right)(u,Y_n^m)_{L^2(S^2)}(a_n^mY_{n-1}^m+a_{n+1}^mY_{n+1}^m) \\
    &= \frac{a_3^m}{2}(u,Y_3^m)_{L^2(S^2)}Y_2^m+\tilde{u}_{\geq3}, \quad \tilde{u}_{\geq3} \in \overline{\mathrm{span}}\{Y_n^m \mid n\geq3\}
  \end{align*}
  and thus $\mathbb{Q}_m\Lambda_mu=\Lambda_mu-\frac{1}{2}a_3^m(u,Y_3^m)_{L^2(S^2)}Y_2^m$ by \eqref{E:Ko_Qm}. Hence
  \begin{align*}
    \delta f &= \mathbb{Q}_m(\mu-\Lambda_m)u = \mu u-\mathbb{Q}_m\Lambda_mu = \mu u-\Lambda_mu+\frac{a_3^m}{2}(u,Y_3^m)_{L^2(S^2)}Y_2^m \\
    &= (\mu-x_3)B_{2,m}u-\left\{6\mu v-\frac{a_3^m}{2}(u,Y_3^m)_{L^2(S^2)}Y_2^m\right\} = (\mu-x_3)B_{2,m}u-w
  \end{align*}
  by $u\in\mathcal{Y}_m$, $\Lambda_mu=x_3B_{2,m}u$, and $B_{2,m}u=u+6v$.

  For each $m\in\mathbb{Z}\setminus\{0\}$ the equality \eqref{E:KL_w_Eq} follows from \eqref{E:KL_vw}, \eqref{E:KL_B2u_Eq}, $B_{2,m}u=\Delta v+6v$, and $\Delta Y_2^m=-\lambda_2Y_2^m=-6Y_2^m$, where $Y_2^m\equiv0$ if $|m|\geq3$.
\end{proof}

\begin{lemma} \label{L:KL_w_est}
  We have
  \begin{align} \label{E:KL_w_est}
    \begin{aligned}
      \|w\|_{L^2(S^2)} &\leq C\|v\|_{L^2(S^2)}, \\
      \|w\|_{H^1(S^2)} &\leq C\|(-\Delta)^{1/2}w\|_{L^2(S^2)} \leq C\|(-\Delta)^{1/2}v\|_{L^2(S^2)}.
    \end{aligned}
  \end{align}
\end{lemma}

\begin{proof}
  By $\Delta Y_3^m=-\lambda_3Y_3^m$ and $\|Y_3^m\|_{L^2(S^2)}\leq1$,
  \begin{align*}
    \bigl|(u,Y_3^m)_{L^2(S^2)}\bigr| = \bigl|(\Delta v,Y_3^m)_{L^2(S^2)}\bigr| = \left|\bigl(v,\Delta Y_3^m\bigr)_{L^2(S^2)}\right| \leq \lambda_3\|v\|_{L^2(S^2)},
  \end{align*}
  where $Y_n^m\equiv0$ if $|m|>n$.
  We apply this inequality, $|\mu|\leq1$, $(-\Delta)^{1/2}Y_2^m=\lambda_2^{1/2}Y_2^m$, and $\|Y_2^m\|_{L^2(S^2)}\leq1$ to $w$ given by \eqref{E:KL_vw} to find that
  \begin{align} \label{Pf_KLwe:w_bd}
    \|(-\Delta)^{k/2}w\|_{L^2(S^2)} \leq C\left(\|(-\Delta)^{k/2}v\|_{L^2(S^2)}+\|v\|_{L^2(S^2)}\right), \quad k=0,1.
  \end{align}
  Hence we have \eqref{E:KL_w_est} by this inequality and \eqref{E:H1_Equiv}.
\end{proof}

\begin{lemma} \label{L:KL_u}
  We have
  \begin{align} \label{E:KL_u}
    \delta\|u\|_{L^2(S^2)} \leq C\left(R(u,\delta,\mu)+\|(-\Delta)^{1/2}v\|_{L^2(S^2)}\right).
  \end{align}
\end{lemma}

\begin{proof}
  Let $\varepsilon=\kappa\delta^2/2\sin\theta_\mu$, which satisfies \eqref{E:KL_eps} and
  \begin{align} \label{Pf_KLu:eps}
    \varepsilon \leq \frac{C\delta^2}{(1-\mu)^{1/2}}
  \end{align}
  by \eqref{E:KL_sinmu}. Also, let $S_{\mu,\varepsilon}^2$ be given by \eqref{E:KL_Sme}. Since $B_{2,m}u\in\mathcal{X}_m$, we can write $B_{2,m}u=U_{2,m}(\theta)e^{im\varphi}$. Then we see by \eqref{E:IF_U} and \eqref{E:B2m_Xmk} that
  \begin{align*}
    |U_{2,m}(\theta)|^2 &\leq \frac{C}{\sin\theta_\mu}\|B_{2,m}u\|_{L^2(S^2)}\|(-\Delta)^{1/2}B_{2,m}u\|_{L^2(S^2)} \\
    &\leq \frac{C}{\sin\theta_\mu}\|u\|_{L^2(S^2)}\|(-\Delta)^{1/2}u\|_{L^2(S^2)}
  \end{align*}
  for $\theta\in(\theta_\mu-\varepsilon,\theta_\mu-\varepsilon)$. By this inequality and \eqref{E:KL_cosep},
  \begin{align} \label{Pf_KLu:B2_In}
    \|B_{2,m}u\|_{L^2(S_{\mu,\varepsilon}^2)}^2 = 2\pi\int_{\theta_\mu-\varepsilon}^{\theta_\mu+\varepsilon}|U_{2,m}(\theta)|^2\sin\theta\,d\theta \leq C\varepsilon\|u\|_{L^2(S^2)}\|(-\Delta)^{1/2}u\|_{L^2(S^2)}.
  \end{align}
  Also, since $B_{2,m}u=(w+\delta f)/(\mu-x_3)$ for $x_3\neq\mu$ by \eqref{E:KL_B2u_Eq},
  \begin{multline} \label{Pf_KLu:B2_Out_01}
    \|B_{2,m}u\|_{L^2(S^2\setminus S_{\mu,\varepsilon}^2)} \leq \|w\|_{L^\infty(S^2)}\left\|\frac{1}{\mu-x_3}\right\|_{L^2(S^2\setminus S_{\mu,\varepsilon}^2)} \\
    +\delta\|f\|_{L^2(S^2)}\left\|\frac{1}{\mu-x_3}\right\|_{L^\infty(S^2\setminus S_{\mu,\varepsilon}^2)}.
  \end{multline}
  Moreover, since $\sin\theta(\mu-\cos\theta)^{-2}=-\frac{d}{d\theta}(\mu-\cos\theta)^{-1}$, we have
  \begin{align*}
    \left\|\frac{1}{\mu-x_3}\right\|_{L^2(S^2\setminus S_{\mu,\varepsilon}^2)}^2 &= 2\pi\left(\int_0^{\theta_\mu-\varepsilon}+\int_{\theta_\mu+\varepsilon}^\pi\right)\frac{\sin\theta}{(\mu-\cos\theta)^2}\,d\theta \\
    &\leq 2\pi\left(\frac{1}{\cos(\theta_\mu-\varepsilon)-\mu}+\frac{1}{\mu-\cos(\theta_\mu+\varepsilon)}\right) \leq \frac{C}{\varepsilon\sin\theta_\mu} \leq C\delta^{-2}
  \end{align*}
  by \eqref{E:L2_Mode}, \eqref{E:KL_cosep}, and $\varepsilon\sin\theta_\mu=\kappa\delta^2/2$. Also, we see by \eqref{E:KL_mucos} that
  \begin{align*}
    \left\|\frac{1}{\mu-x_3}\right\|_{L^\infty(S^2\setminus S_{\mu,\varepsilon}^2)} = \sup_{|\theta-\theta_\mu|\geq\varepsilon}\frac{1}{|\mu-\cos\theta|} \leq \frac{2}{\varepsilon\sin\theta_\mu} \leq C\delta^{-2}.
  \end{align*}
  We apply these inequalities to \eqref{Pf_KLu:B2_Out_01} and use \eqref{E:Linf} and \eqref{E:KL_w_est} to $w$ to get
  \begin{align} \label{Pf_KLu:B2_Out_02}
    \|B_{2,m}u\|_{L^2(S^2\setminus S_{\mu,\varepsilon}^2)} \leq C\delta^{-1}\Bigl(\|f\|_{L^2(S^2)}+\|(-\Delta)^{1/2}v\|_{L^2(S^2)}\Bigr).
  \end{align}
  By \eqref{E:B2m_Ym}, \eqref{Pf_KLu:B2_In}, \eqref{Pf_KLu:B2_Out_02}, and Young's inequality, we find that
  \begin{align*}
    \delta\|u\|_{L^2(S^2)} \leq \frac{1}{2}\delta\|u\|_{L^2(S^2)}+C\Bigl(\delta\varepsilon\|(-\Delta)^{1/2}u\|_{L^2(S^2)}+\|f\|_{L^2(S^2)}+\|(-\Delta)^{1/2}v\|_{L^2(S^2)}\Bigr).
  \end{align*}
  Hence we have \eqref{E:KL_u} by this inequality and \eqref{Pf_KLu:eps}.
\end{proof}

\begin{lemma} \label{L:KL_gradf}
  Let $\varepsilon$ satisfy \eqref{E:KL_eps} and $S_{\mu,\varepsilon}^2$ be given by \eqref{E:KL_Sme}. Then
  \begin{align} \label{E:KL_gradf}
    \begin{aligned}
      \delta\|\nabla f\|_{L^2(S_{\mu,\varepsilon}^2)} &\leq C\Bigl(\|M_{\sin\theta}u\|_{L^2(S_{\mu,\varepsilon}^2)}\Bigr. \\
      &\qquad\qquad \Bigl.+\varepsilon\sin\theta_\mu\|(-\Delta)^{1/2}u\|_{L^2(S^2)}+\|(-\Delta)^{1/2}v\|_{L^2(S^2)}\Bigr).
    \end{aligned}
  \end{align}
\end{lemma}

\begin{proof}
  By \eqref{E:KL_B2u_Eq} and $B_{2,m}u=u+6v$,
  \begin{align*}
    \delta\nabla f = -u\nabla x_3-6v\nabla x_3+(\mu-x_3)\nabla B_{2,m}u-\nabla w \quad\text{on}\quad S^2.
  \end{align*}
  We take the $L^2(S_{\mu,\varepsilon}^2)$-norms of both sides, apply \eqref{E:KL_theta} to $\mu-x_3=\mu-\cos\theta$ on $S_{\mu,\varepsilon}^2$, and use $|\nabla x_3|=\sin\theta\leq1$ on $S^2$ by $x_3=\cos\theta$ and \eqref{E:Re_SC}. Then, replacing the $L^2$-norms of $v$, $\nabla B_{2,m}u$, and $\nabla w$ on $S_{\mu,\varepsilon}^2$ by those on $S^2$, we get
  \begin{align*}
    \delta\|\nabla f\|_{L^2(S_{\mu,\varepsilon}^2)} &\leq C\Bigl(\|M_{\sin\theta}u\|_{L^2(S_{\mu,\varepsilon}^2)}+\|v\|_{L^2(S^2)}\Bigr. \\
    &\qquad\qquad \Bigl.+\varepsilon\sin\theta_\mu\|\nabla B_{2,m}u\|_{L^2(S^2)}+\|\nabla w\|_{L^2(S^2)}\Bigr).
  \end{align*}
  In this inequality, we apply \eqref{E:H1_Equiv} and \eqref{E:KL_w_est} to $v$ and $\nabla w$, respectively, and use \eqref{E:Lap_Half} and \eqref{E:B2m_Xmk} with $k=1$ to $\nabla B_{2,m}u$. Then we obtain \eqref{E:KL_gradf}.
\end{proof}

\begin{lemma} \label{L:KL_WF}
  We have
  \begin{align} \label{E:KL_WL_form}
    w = W(\theta)e^{im\varphi}, \quad f = F(\theta)e^{im\varphi}, \quad W\in C^1([0,\pi]), \quad F\in H_{loc}^1(0,\pi) \subset C(0,\pi)
  \end{align}
  and $W(\theta_\mu)+\delta F(\theta_\mu)=0$. Moreover, for all $\sigma_1\in(0,1/\sqrt{2})$,
  \begin{align} \label{E:KL_W_thmu}
    |W(\theta_\mu)| \leq C\left(\sigma_1^{-1}R(u,\delta,\mu)+\sigma_1\|(-\Delta)^{1/2}v\|_{L^2(S^2)}\right).
  \end{align}
\end{lemma}

\begin{proof}
  Since $w\in\mathcal{X}_m\cap H^3(S^2) \subset \mathcal{X}_m\cap C^1(S^2)$ by the Sobolev embedding theorem (see \cite{Aub98}), and since $f\in\mathcal{X}_m\cap H^1(S^2)$, we can write \eqref{E:KL_WL_form}. Also, since $w$ and $f$ are continuous near $x_3=\mu$ and $B_{2,m}u=(w+\delta f)/(\mu-x_3)$ for $x_3\neq\mu$ by \eqref{E:KL_B2u_Eq}, we have $w+\delta f=0$ at $x_3=\mu$, otherwise $B_{2,m}u$ does not belong to $L^2(S^2)$. Hence $W(\theta_\mu)+\delta F(\theta_\mu)=0$.

  Let us show \eqref{E:KL_W_thmu}. For $\sigma_1\in(0,1/\sqrt{2})$ let $\varepsilon=\sigma_1^2\kappa\delta^2/\sin\theta_\mu$, which satisfies \eqref{E:KL_eps} and \eqref{Pf_KLu:eps} by \eqref{E:KL_sinmu}. By the mean value theorem for integrals, there exists $\theta'_\mu\in(\theta_\mu-\varepsilon,\theta_\mu+\varepsilon)$ such that
  \begin{align*}
    \int_{\theta_\mu-\varepsilon}^{\theta_\mu+\varepsilon}|F(\theta)|^2\sin\theta\,d\theta = 2\varepsilon|F(\theta'_\mu)|^2\sin\theta'_\mu.
  \end{align*}
  Then since $2\varepsilon\sin\theta'_\mu\geq \varepsilon\sin\theta_\mu=\sigma_1^2\kappa\delta^2$ by \eqref{E:KL_theta}, we have
  \begin{align} \label{Pf_WF:prime}
    |F(\theta'_\mu)|^2 \leq \frac{1}{\varepsilon\sin\theta_\mu}\int_{\theta_\mu-\varepsilon}^{\theta_\mu+\varepsilon}|F(\theta)|^2\sin\theta\,d\theta \leq \frac{C}{\sigma_1^2\delta^2}\|f\|_{L^2(S^2)}^2
  \end{align}
  by \eqref{E:L2_Mode}. Also, we use H\"{o}lder's inequality, \eqref{E:L2_Mode}, and \eqref{E:KL_theta} to get
  \begin{align} \label{Pf_WF:diff}
    \begin{aligned}
      |F(\theta_\mu)-F(\theta'_\mu)| = \left|\int_{\theta'_\mu}^{\theta_\mu}F'(\theta)\,d\theta\right| &\leq C\left(\frac{\varepsilon}{\sin\theta_\mu}\right)^{1/2}\left(\int_{\theta_\mu-\varepsilon}^{\theta_\mu+\varepsilon}|F'(\theta)|^2\sin\theta\,d\theta\right)^{1/2} \\
      &\leq \frac{C\sigma_1\delta}{\sin\theta_\mu}\|\nabla f\|_{L^2(S_{\mu,\varepsilon}^2)},
    \end{aligned}
  \end{align}
  where $S_{\mu,\varepsilon}^2$ is given by \eqref{E:KL_Sme}. Moreover, since
  \begin{align*}
    \|M_{\sin\theta}u\|_{L^2(S_{\mu,\varepsilon}^2)} \leq C\sin\theta_\mu\|u\|_{L^2(S_{\mu,\varepsilon}^2)} \leq C\sin\theta_\mu\|u\|_{L^2(S^2)}, \quad \frac{\delta}{\sin\theta_\mu} \leq C
  \end{align*}
  by \eqref{E:KL_sinmu} and \eqref{E:KL_theta}, we see by \eqref{E:KL_u}, \eqref{Pf_KLu:eps}, and \eqref{E:KL_gradf} that
  \begin{align} \label{Pf_WF:grf}
    \begin{aligned}
      \frac{\delta^2}{\sin\theta_\mu}\|\nabla f\|_{L^2(S_{\mu,\varepsilon}^2)} &\leq C\left(\frac{\delta}{\sin\theta_\mu}\|M_{\sin\theta}u\|_{L^2(S_{\mu,\varepsilon}^2)}\right. \\
      &\qquad \left.+\delta\varepsilon\|(-\Delta)^{1/2}u\|_{L^2(S^2)}+\frac{\delta}{\sin\theta_\mu}\|(-\Delta)^{1/2}v\|_{L^2(S^2)}\right) \\
      &\leq C\Bigl(R(u,\delta,\mu)+\|(-\Delta)^{1/2}v\|_{L^2(S^2)}\Bigr).
    \end{aligned}
  \end{align}
  By $W(\theta_\mu)+\delta F(\theta_\mu)=0$, \eqref{Pf_WF:prime}, and \eqref{Pf_WF:diff}, we have
  \begin{align*}
    |W(\theta_\mu)| \leq \delta|F(\theta'_\mu)|+\delta|F(\theta_\mu)-F(\theta'_\mu)| \leq C\left(\sigma_1^{-1}\|f\|_{L^2(S^2)}+\frac{\sigma_1\delta^2}{\sin\theta_\mu}\|\nabla f\|_{L^2(S_{\mu,\varepsilon}^2)}\right)
  \end{align*}
  and we apply \eqref{Pf_WF:grf} and $\sigma_1<1$ to the last term to get \eqref{E:KL_W_thmu}.
\end{proof}

\begin{lemma} \label{L:KL_vH1}
  We have
  \begin{multline} \label{E:KL_vH1}
    \bigl|(\nabla v,\nabla\psi)_{L^2(S^2)}-6(v,\psi)_{L^2(S^2)}\bigr| \\
    \leq C\left(\sigma_2^{-2}R(u,\delta,\mu)+\sigma_2\|(-\Delta)^{1/2}v\|_{L^2(S^2)}\right)\|\nabla\psi\|_{L^2(S^2)} \\
    +\frac{C}{(1-\mu)^{3/4}}\left(\sigma_2^{-4}R(u,\delta,\mu)+\|(-\Delta)^{1/2}v\|_{L^2(S^2)}\right)\|\psi\|_{L^2(S^2)}
  \end{multline}
  for all $\psi\in H^1(S^2)$ and $\sigma_2\in(0,1/2)$. Moreover,
  \begin{multline} \label{E:KL_vLaph}
    3\mu\|(-\Delta)^{1/2}v\|_{L^2(S^2)}^2 \leq C\left(\sigma_2^{-2}R(u,\delta,\mu)+\sigma_2\|(-\Delta)^{1/2}v\|_{L^2(S^2)}\right)\|(-\Delta)^{1/2}v\|_{L^2(S^2)} \\
    +\frac{C}{\sigma_2^2(1-\mu)}\|w\|_{L^2(S^2(0,\theta_\mu))}^2,
  \end{multline}
  where $S^2(0,\theta_\mu)$ is given by \eqref{E:Def_Band}.
\end{lemma}

\begin{proof}
  By Lemma \ref{L:KL_WF} we have \eqref{E:KL_WL_form} and $W(\theta_\mu)+\delta F(\theta_\mu)=0$. Since
  \begin{align*}
    -\Delta v-6v = -B_{2,m}u = -\frac{w+\delta f}{\mu-x_3}
  \end{align*}
  by \eqref{E:KL_B2u_Eq}, we take the $L^2(S^2)$-inner products of both sides with $\psi\in H^1(S^2)$ and carry out integration by parts for $(\Delta v,\psi)_{L^2(S^2)}$ to get
  \begin{align} \label{Pf_vH1:IBP}
    (\nabla v,\nabla\psi)_{L^2(S^2)}-6(v,\psi)_{L^2(S^2)} = -\left(\frac{w+\delta f}{\mu-x_3},\psi\right)_{L^2(S^2)}.
  \end{align}
  Let us estimate the right-hand side. We may assume $\psi=\Psi(\theta)e^{im\varphi}$ since
  \begin{align*}
    \frac{w+\delta f}{\mu-x_3} = \left\{\frac{W(\theta)+\delta F(\theta)}{\mu-\cos\theta}\right\}e^{im\varphi}.
  \end{align*}
  For $\sigma_2\in(0,1/2)$ let $\varepsilon=\sigma_2^2\kappa\delta^2/\sin\theta_\mu$ and $\varepsilon' = 2\sigma_2^2\sin\theta_\mu$. Then $\varepsilon<\varepsilon'$ by \eqref{E:KL_sinmu}. Also, $\varepsilon'$ and thus $\varepsilon$ satisfy \eqref{E:KL_eps} by $\sigma_2^2<1/4$. Let
  \begin{align*}
    I_1 &= (\theta_\mu-\varepsilon,\theta_\mu+\varepsilon), \\
    I_2 &= (\theta_\mu-\varepsilon',\theta_\mu-\varepsilon]\cup[\theta_\mu+\varepsilon,\theta_\mu+\varepsilon'), \\
    I_3 &= (0,\theta_\mu-\varepsilon']\cup[\theta_\mu+\varepsilon',\pi).
  \end{align*}
  Then we can write
  \begin{align} \label{Pf_vH1:decom}
    -\left(\frac{w+\delta f}{\mu-x_3},\psi\right)_{L^2(S^2)} = -2\pi\int_0^\pi\frac{W(\theta)+\delta F(\theta)}{\mu-\cos\theta}\,\overline{\Psi(\theta)}\sin\theta\,d\theta = 2\pi\sum_{k=1}^3J_k[\psi],
  \end{align}
  where
  \begin{align*}
    J_k[\psi] = -\int_{I_k}\frac{W(\theta)+\delta F(\theta)}{\mu-\cos\theta}\,\overline{\Psi(\theta)}\sin\theta\,d\theta, \quad k=1,2,3.
  \end{align*}
  We estimate $J_1[\psi]$, $J_2[\psi]$, and $J_3[\psi]$ separately. For $J_1[\psi]$, we have
  \begin{align*}
    |J_1[\psi]| \leq \left(\int_{I_1}\left|\frac{\{W(\theta)+\delta F(\theta)\}\sqrt{\sin\theta}}{\mu-\cos\theta}\right|^2\,d\theta\right)^{1/2}\left(\int_{I_1}|\Psi(\theta)|^2\sin\theta\,d\theta\right)^{1/2}.
  \end{align*}
  Since $W(\theta_\mu)+\delta F(\theta_\mu)=0$, we can use \eqref{E:Hardy} to get
  \begin{align*}
    \int_{I_1}\left|\frac{\{W(\theta)+\delta F(\theta)\}\sqrt{\sin\theta}}{\mu-\cos\theta}\right|^2\,d\theta \leq \frac{C}{\sin^2\theta_\mu}\|\nabla w+\delta\nabla f\|_{L^2(S_{\mu,\varepsilon}^2)}^2,
  \end{align*}
  where $S_{\mu,\varepsilon}^2$ is given by \eqref{E:KL_Sme}. Also, we see by \eqref{E:Linf} and \eqref{E:KL_cosep} that
  \begin{align*}
    \int_{I_1}|\Psi(\theta)|^2\sin\theta\,d\theta &\leq \|\Psi\|_{L^\infty(0,\pi)}^2\{\cos(\theta_\mu-\varepsilon)-\cos(\theta_\mu+\varepsilon)\} \\
    &\leq C\varepsilon\sin\theta_\mu\|(-\Delta)^{1/2}\psi\|_{L^2(S^2)}^2 = C\sigma_2^2\kappa\delta^2\|(-\Delta)^{1/2}\psi\|_{L^2(S^2)}^2.
  \end{align*}
  We use these inequalities, \eqref{E:KL_w_est}, \eqref{Pf_WF:grf}, and $\delta/\sin\theta_\mu\leq C$ by \eqref{E:KL_sinmu} to deduce that
  \begin{align} \label{Pf_vH1:J1}
    \begin{aligned}
      |J_1[\psi]| &\leq C\sigma_2\Bigl(\frac{\delta}{\sin\theta_\mu}\|\nabla w\|_{L^2(S^2)}+\frac{\delta^2}{\sin\theta_\mu}\|\nabla f\|_{L^2(S_{\mu,\varepsilon}^2)}\Bigr)\|(-\Delta)^{1/2}\psi\|_{L^2(S^2)} \\
      &\leq C\sigma_2\Bigl(R(u,\delta,\mu)+\|(-\Delta)^{1/2}v\|_{L^2(S^2)}\Bigr)\|(-\Delta)^{1/2}\psi\|_{L^2(S^2)}.
    \end{aligned}
  \end{align}
  Note that we can use \eqref{Pf_WF:grf} since $\varepsilon$ satisfies \eqref{E:KL_eps} and \eqref{Pf_KLu:eps} by \eqref{E:KL_sinmu}.

  To estimate $J_2[\psi]$, we decompose $J_2[\psi]=\sum_{k=1}^5J_{2,k}[\psi]$ as
  \begin{align*}
    J_{2,1}[\psi] &= -\int_{I_2}\frac{W(\theta)\sqrt{\sin\theta}-W(\theta_\mu)\sqrt{\sin\theta_\mu}}{\mu-\cos\theta}\,\overline{\Psi(\theta)}\sqrt{\sin\theta}\,d\theta, \\
    J_{2,2}[\psi] &= -W(\theta_\mu)\sqrt{\sin\theta_\mu}\int_{I_2}\frac{\overline{\Psi(\theta)}\sqrt{\sin\theta}-\overline{\Psi(\theta_\mu)}\sqrt{\sin\theta_\mu}}{\mu-\cos\theta}\,d\theta, \\
    J_{2,3}[\psi] &= -W(\theta_\mu)\overline{\Psi(\theta_\mu)}\int_{I_2}\frac{\sin\theta_\mu-\sin\theta}{\mu-\cos\theta}\,d\theta, \\
    J_{2,4}[\psi] &= -W(\theta_\mu)\overline{\Psi(\theta_\mu)}\int_{I_2}\frac{\sin\theta}{\mu-\cos\theta}\,d\theta, \\
    J_{2,5}[\psi] &= -\int_{I_2}\frac{\delta F(\theta)}{\mu-\cos\theta}\,\overline{\Psi(\theta)}\sin\theta\,d\theta.
  \end{align*}
  We apply H\"{o}lder's inequality and \eqref{E:Hardy} to $J_{2,1}[\psi]$ to get
  \begin{align*}
    |J_{2,1}[\psi]| \leq \frac{C}{\sin\theta_\mu}\|\nabla w\|_{L^2(S^2)}\|\Psi\|_{L^\infty(0,\pi)}\sqrt{\cos(\theta_\mu-\varepsilon')-\cos(\theta_\mu+\varepsilon')}.
  \end{align*}
  Moreover, by \eqref{E:Linf}, \eqref{E:KL_cosep}, \eqref{E:KL_w_est}, and $\varepsilon'\sin\theta_\mu=2\sigma_2^2\sin^2\theta_\mu$,
  \begin{align} \label{Pf_vH1:J21}
    |J_{2,1}[\psi]| \leq C\sigma_2\|(-\Delta)^{1/2}v\|_{L^2(S^2)}\|(-\Delta)^{1/2}\psi\|_{L^2(S^2)}.
  \end{align}
  Similarly, by H\"{o}lder's inequality, \eqref{E:Lap_Half}, \eqref{E:Hardy}, and $|I_2|\leq 2\varepsilon'=4\sigma_2^2\sin\theta_\mu$,
  \begin{align} \label{Pf_vH1:J22}
    |J_{2,2}[\psi]| \leq C\sigma_2|W(\theta_\mu)|\|(-\Delta)^{1/2}\psi\|_{L^2(S^2)}.
  \end{align}
  Here $|I_2|$ is the length of $I_2$. For $J_{2,3}[\psi]$, since
  \begin{align*}
    |\sin\theta_\mu-\sin\theta| \leq |\theta-\theta_\mu|, \quad |\mu-\cos\theta| \geq \frac{1}{2}|\theta-\theta_\mu|\sin\theta_\mu, \quad \theta\in[0,\pi]
  \end{align*}
  by \eqref{E:KL_mucos}, it follows from \eqref{E:Linf} and $|I_2|\leq 2\varepsilon'=4\sigma_2^2\sin\theta_\mu$ that
  \begin{align} \label{Pf_vH1:J23}
    |J_{2,3}[\psi]| \leq \frac{2|I_2|}{\sin\theta_\mu}|W(\theta_\mu)|\|\Psi\|_{L^\infty(0,\pi)} \leq C\sigma_2^2|W(\theta_\mu)|\|(-\Delta)^{1/2}\psi\|_{L^2(S^2)}.
  \end{align}
  To estimate $J_{2,4}[\psi]$, we see that
  \begin{align*}
    \int_{I_2}\frac{\sin\theta}{\mu-\cos\theta}\,d\theta &= \int_{\theta_\mu-\varepsilon'}^{\theta_\mu-\varepsilon}\frac{d}{d\theta}\bigl(\log(\cos\theta-\mu)\bigr)\,d\theta+\int_{\theta_\mu+\varepsilon}^{\theta_\mu+\varepsilon'}\frac{d}{d\theta}\bigl(\log(\mu-\cos\theta)\bigr)\,d\theta \\
    &= \log\frac{\cos(\theta_\mu-\varepsilon)-\mu}{\mu-\cos(\theta_\mu+\varepsilon)}-\log\frac{\cos(\theta_\mu-\varepsilon')-\mu}{\mu-\cos(\theta_\mu+\varepsilon')}.
  \end{align*}
  Moreover, since $\varepsilon$ and $\varepsilon'$ satisfy \eqref{E:KL_eps}, we can use \eqref{E:KL_cosep} to get
  \begin{align*}
    \left|\int_{I_2}\frac{\sin\theta}{\mu-\cos\theta}\,d\theta\right| \leq \sum_{\eta=\varepsilon,\varepsilon'}\left|\log\frac{\cos(\theta_\mu-\eta)-\mu}{\mu-\cos(\theta_\mu+\eta)}\right| \leq 2\log\frac{5}{3}.
  \end{align*}
  Hence we deduce from this inequality and \eqref{E:Linf} that
  \begin{align} \label{Pf_vH1:J24}
    |J_{2,4}[\psi]| \leq C|W(\theta_\mu)|\|\Psi\|_{L^\infty(0,\pi)} \leq C|W(\theta_\mu)|\|(-\Delta)^{1/2}\psi\|_{L^2(S^2)}.
  \end{align}
  For $J_{2,5}[\psi]$, it follows from \eqref{E:L2_Mode} and \eqref{E:Linf} that
  \begin{align*}
    |J_{2,5}[\psi]| &\leq \delta\|\Psi\|_{L^\infty(0,\pi)}\left(\int_0^\pi|F(\theta)|^2\sin\theta\,d\theta\right)^{1/2}\left(\int_{I_2}\frac{\sin\theta}{(\mu-\cos\theta)^2}\,d\theta\right)^{1/2} \\
    &\leq C\delta\|f\|_{L^2(S^2)}\|(-\Delta)^{1/2}\psi\|_{L^2(S^2)}\left(\int_{I_2}\frac{\sin\theta}{(\mu-\cos\theta)^2}\,d\theta\right)^{1/2}.
  \end{align*}
  Moreover, since $\sin\theta(\mu-\cos\theta)^{-2}=-\frac{d}{d\theta}(\mu-\cos\theta)^{-1}$,
  \begin{align} \label{Pf_vH1:I2_int}
    \begin{aligned}
      \int_{I_2}\frac{\sin\theta}{(\mu-\cos\theta)^2}\,d\theta &\leq \frac{1}{\cos(\theta_\mu-\varepsilon)-\mu}+\frac{1}{\mu-\cos(\theta_\mu+\varepsilon)} \\
      &\leq \frac{C}{\varepsilon\sin\theta_\mu} \leq C\sigma_2^{-2}\delta^{-2}
    \end{aligned}
  \end{align}
  by \eqref{E:KL_cosep} and $\varepsilon\sin\theta_\mu=\sigma_2^2\kappa\delta^2$. Hence
  \begin{align} \label{Pf_vH1:J25}
    |J_{2,5}[\psi]| \leq C\sigma_2^{-1}\|f\|_{L^2(S^2)}\|(-\Delta)^{1/2}\psi\|_{L^2(S^2)}.
  \end{align}
  Noting that $\sigma_2<1$, we apply \eqref{Pf_vH1:J21}--\eqref{Pf_vH1:J25} to $J_2[\psi]=\sum_{k=1}^5J_{2,k}[\psi]$ to get
  \begin{align} \label{Pf_vH1:J2}
    |J_2[\psi]| \leq C\Bigl(|W(\theta_\mu)|+\sigma_2^{-1}\|f\|_{L^2(S^2)}+\sigma_2\|(-\Delta)^{1/2}v\|_{L^2(S^2)}\Bigr)\|(-\Delta)^{1/2}\psi\|_{L^2(S^2)}.
  \end{align}
  Let us estimate $J_3[\psi]$. We write $J_3[\psi]=J_{3,1}[\psi]+J_{3,2}[\psi]$ with
  \begin{align} \label{Pf_vH1:J3_dec}
    \begin{aligned}
      J_{3,1}[\psi] &= -\int_{I_3}\frac{W(\theta)}{\mu-\cos\theta}\,\overline{\Psi(\theta)}\sin\theta\,d\theta, \\
      J_{3,2}[\psi] &= -\int_{I_3}\frac{\delta F(\theta)}{\mu-\cos\theta}\,\overline{\Psi(\theta)}\sin\theta\,d\theta.
    \end{aligned}
  \end{align}
  To estimate $J_{3,1}[\psi]$, we split $J_{3,1}[\psi]=K_1[\psi]+K_2[\psi]$ into
  \begin{align*}
    K_1[\psi] &= -\int_{I_3}\frac{W(\theta)\sqrt{\sin\theta}-W(\theta_\mu)\sqrt{\sin\theta_\mu}}{\mu-\cos\theta}\,\overline{\Psi(\theta)}\sqrt{\sin\theta}\,d\theta, \\
    K_2[\psi] &= -W(\theta_\mu)\sqrt{\sin\theta_\mu}\int_{I_3}\frac{\overline{\Psi(\theta)}\sqrt{\sin\theta}}{\mu-\cos\theta}\,d\theta.
  \end{align*}
  We use H\"{o}lder's inequality, \eqref{E:Lap_Half}, \eqref{E:L2_Mode}, and \eqref{E:Hardy} to $K_1[\psi]$ to get
  \begin{align*}
    |K_1[\psi]| \leq \frac{C}{\sin\theta_\mu}\|(-\Delta)^{1/2}w\|_{L^2(S^2)}\|\psi\|_{L^2(S^2)}.
  \end{align*}
  Also, if $\theta\in I_3$, i.e. $|\theta-\theta_\mu|\geq\varepsilon'$, then it follows from \eqref{E:KL_mucos} that
  \begin{align*}
    |\mu-\cos\theta| \geq \frac{1}{2}|\theta-\theta_\mu|\sin\theta_\mu \geq \frac{1}{2}\varepsilon'\sin\theta_\mu = \sigma_2^2\sin^2\theta_\mu.
  \end{align*}
  By this inequality, H\"{o}lder's inequality, and \eqref{E:L2_Mode}, we have
  \begin{align*}
    |K_2[\psi]| \leq \frac{C|W(\theta_\mu)|\|\psi\|_{L^2(S^2)}}{\sigma_2^2\sin^{3/2}\theta_\mu}.
  \end{align*}
 These estimates, \eqref{E:KL_sinmu}, and \eqref{E:KL_w_est} imply that
  \begin{align} \label{Pf_vH1:J31}
    \begin{aligned}
      |J_{3,1}[\psi]| &\leq \frac{C}{\sin^{3/2}\theta_\mu}\Bigl(\sigma_2^{-2}|W(\theta_\mu)|+\sqrt{\sin\theta_\mu}\|(-\Delta)^{1/2}w\|_{L^2(S^2)}\Bigr)\|\psi\|_{L^2(S^2)} \\
      &\leq \frac{C}{(1-\mu)^{3/4}}\Bigl(\sigma_2^{-2}|W(\theta_\mu)|+\|(-\Delta)^{1/2}v\|_{L^2(S^2)}\Bigr)\|\psi\|_{L^2(S^2)}.
    \end{aligned}
  \end{align}
  Also, as in \eqref{Pf_vH1:I2_int}, we see by \eqref{E:KL_cosep} and $\varepsilon'=2\sigma_2^2\sin\theta_\mu$ that
  \begin{align*}
    \int_{I_3}\frac{\sin\theta}{(\mu-\cos\theta)^2}\,d\theta \leq \frac{1}{\cos(\theta_\mu-\varepsilon')-\mu}+\frac{1}{\mu-\cos(\theta_\mu+\varepsilon')} \leq \frac{C}{\varepsilon'\sin\theta_\mu} \leq \frac{C}{\sigma_2^2\sin^2\theta_\mu}.
  \end{align*}
  We deduce from this inequality, \eqref{E:L2_Mode}, \eqref{E:Linf}, and $\delta/\sin\theta_\mu\leq C$ by \eqref{E:KL_sinmu} that
  \begin{align} \label{Pf_vH1:J32}
    \begin{aligned}
      |J_{3,2}[\psi]| &\leq \delta\|\Psi\|_{L^\infty(0,\pi)}\left(\int_0^\pi|F(\theta)|^2\sin\theta\,d\theta\right)^{1/2}\left(\int_{I_3}\frac{\sin\theta}{(\mu-\cos\theta)^2}\,d\theta\right)^{1/2} \\
      &\leq C\sigma_2^{-1}\|f\|_{L^2(S^2)}\|(-\Delta)^{1/2}\psi\|_{L^2(S^2)}.
    \end{aligned}
  \end{align}
  Applying \eqref{Pf_vH1:J31} and \eqref{Pf_vH1:J32} to $J_3[\psi]=J_{3,1}[\psi]+J_{3,2}[\psi]$, we get
  \begin{multline} \label{Pf_vH1:J3}
    |J_3[\psi]| \leq \frac{C}{(1-\mu)^{3/4}}\Bigl(\sigma_2^{-2}|W(\theta_\mu)|+\|(-\Delta)^{1/2}v\|_{L^2(S^2)}\Bigr)\|\psi\|_{L^2(S^2)} \\
    +C\sigma_2^{-1}\|f\|_{L^2(S^2)}\|(-\Delta)^{1/2}\psi\|_{L^2(S^2)}.
  \end{multline}
  Now we deduce from \eqref{Pf_vH1:IBP}--\eqref{Pf_vH1:J1}, \eqref{Pf_vH1:J2}, \eqref{Pf_vH1:J3}, and $\sigma_2<1$ that
  \begin{multline*}
    \bigl|(\nabla v,\nabla\psi)_{L^2(S^2)}-6(v,\psi)_{L^2(S^2)}\bigr| \leq 2\pi(|J_1[\psi]|+|J_2[\psi]|+|J_3[\psi]|) \\
    \leq C\Bigl(|W(\theta_\mu)|+\sigma_2^{-1}R(u,\delta,\mu)+\sigma_2\|(-\Delta)^{1/2}v\|_{L^2(S^2)}\Bigr)\|(-\Delta)^{1/2}\psi\|_{L^2(S^2)} \\
    +\frac{C}{(1-\mu)^{3/4}}\Bigl(\sigma_2^{-2}|W(\theta_\mu)|+\|(-\Delta)^{1/2}v\|_{L^2(S^2)}\Bigr)\|\psi\|_{L^2(S^2)}.
  \end{multline*}
  In this inequality, we further use \eqref{E:KL_W_thmu} with $\sigma_1=\sigma_2^2$ and apply \eqref{E:Lap_Half} to $\psi$ (note that we assume $\psi=\Psi(\theta)e^{im\varphi}$ with $m\neq0$). Then we obtain \eqref{E:KL_vH1} since $\sigma_2<1$.

  Let us show \eqref{E:KL_vLaph}. We set $\psi=w$ for $J_{3,1}[\psi]$ (see \eqref{Pf_vH1:J3_dec}) to get
  \begin{align*}
    J_{3,1}[w] = -\int_{I_3}\frac{|W(\theta)|^2}{\mu-\cos\theta}\sin\theta\,d\theta = \left(\int_0^{\theta_\mu-\varepsilon'}+\int_{\theta_\mu+\varepsilon'}^\pi\right)\frac{|W(\theta)|^2}{\cos\theta-\mu}\sin\theta\,d\theta \in \mathbb{R}.
  \end{align*}
  Moreover, since $\cos\theta-\mu\leq0$ for $\theta\in(\theta_\mu+\varepsilon',\pi)$ and
  \begin{align*}
    \cos\theta-\mu \geq \frac{1}{2}(\theta_\mu-\theta)\sin\theta_\mu \geq \frac{1}{2}\varepsilon'\sin\theta_\mu = \sigma_2^2\sin^2\theta_\mu \geq \sigma_2^2(1-\mu)
  \end{align*}
  for $\theta\in(0,\theta_\mu-\varepsilon')$ by \eqref{E:KL_sinmu} and \eqref{E:KL_mucos}, we see by \eqref{E:L2_Mode} that
  \begin{align} \label{Pf_vH1:J31_w}
    J_{3,1}[w] \leq \int_0^{\theta_\mu-\varepsilon'}\frac{|W(\theta)|^2}{\cos\theta-\mu}\sin\theta\,d\theta \leq \frac{C}{\sigma_2^2(1-\mu)}\|w\|_{L^2(S^2(0,\theta_\mu))}^2,
  \end{align}
  where $S^2(0,\theta_\mu)$ is given by \eqref{E:Def_Band}. Also, since $v=\Delta^{-1}u\in\mathcal{Y}_m$,
  \begin{align*}
    (v,Y_2^m)_{L^2(S^2)} = 0, \quad (\nabla v,\nabla Y_2^m)_{L^2(S^2)} = -(v,\Delta Y_2^m)_{L^2(S^2)} = \lambda_2(v,Y_2^m)_{L^2(S^2)} = 0.
  \end{align*}
  Note that $Y_2^m\equiv0$ when $|m|\geq3$. By this fact and $w=6\mu v-\beta_mY_2^m$,
  \begin{align*}
    (\nabla v,\nabla w)_{L^2(S^2)}-6(v,w)_{L^2(S^2)} = 6\mu\Bigl(\|\nabla v\|_{L^2(S^2)}^2-6\|v\|_{L^2(S^2)}^2\Bigr) \in \mathbb{R}.
  \end{align*}
  Thus, setting $\psi=w$ in \eqref{Pf_vH1:IBP} and using \eqref{Pf_vH1:decom}, we have
  \begin{align*}
    6\mu\Bigl(\|\nabla v\|_{L^2(S^2)}^2-6\|v\|_{L^2(S^2)}^2\Bigr) &= 2\pi\sum_{k=1}^3J_k[w] = 2\pi\{\mathrm{Re}(J_1[w]+J_2[w]+J_{3,2}[w])+J_{3,1}[w]\} \\
    &\leq 2\pi(|J_1[w]|+|J_2[w]|+|J_{3,2}[w]|+J_{3,1}[w])
  \end{align*}
  and we use \eqref{Pf_vH1:J1}, \eqref{Pf_vH1:J2}, \eqref{Pf_vH1:J32}, \eqref{Pf_vH1:J31_w}, and $\sigma_2<1$ to find that
  \begin{multline} \label{Pf_vH1:Laph_v}
    6\mu\Bigl(\|\nabla v\|_{L^2(S^2)}^2-6\|v\|_{L^2(S^2)}^2\Bigr) \\
    \leq C\Bigl(|W(\theta_\mu)|+\sigma_2^{-1}R(u,\delta,\mu)+\sigma_2\|(-\Delta)^{1/2}v\|_{L^2(S^2)}\Bigr)\|(-\Delta)^{1/2}w\|_{L^2(S^2)} \\
    +\frac{C}{\sigma_2^2(1-\mu)}\|w\|_{L^2(S^2(0,\theta_\mu))}^2.
  \end{multline}
  Moreover, since $u\in\mathcal{Y}_m$, we can use \eqref{Pf_KH:nabla} in the proof of Lemma \ref{L:Ko_High} to get
  \begin{align*}
    6\mu\Bigl(\|\nabla v\|_{L^2(S^2)}^2-6\|v\|_{L^2(S^2)}^2\Bigr) \geq 3\mu\|(-\Delta)^{-1/2}u\|_{L^2(S^2)}^2 = 3\mu\|(-\Delta)^{1/2}v\|_{L^2(S^2)}^2.
  \end{align*}
  We apply this inequality, \eqref{E:KL_w_est} for $(-\Delta)^{1/2}w$, and \eqref{E:KL_W_thmu} with $\sigma_1=\sigma_2^2$ to \eqref{Pf_vH1:Laph_v}, and then use $\sigma_2<1$ to obtain \eqref{E:KL_vLaph}.
\end{proof}

\begin{lemma} \label{L:KL_wLS}
  There exist constants $\mu_0\in(0,1)$ and $C>0$ such that
  \begin{align} \label{E:KL_wLS}
    \frac{1}{1-\mu}\|w\|_{L^2(S^2(\theta_\mu,\pi))}^2 \leq C\left(\sigma_3^{-1}R(u,\delta,\mu)^2+\sigma_3\|(-\Delta)^{1/2}v\|_{L^2(S^2)}^2\right)
  \end{align}
  for all $\mu\in[\mu_0,1)$ and $\sigma_3\in(0,1/2)$, where $S^2(\theta_\mu,\pi)$ is given by \eqref{E:Def_Band}.
\end{lemma}

\begin{proof}
  Since $(\mu-x_3)\Delta w = 6x_3w+6\mu\delta f$ by \eqref{E:KL_w_Eq}, we use \eqref{E:KL_WL_form}, \eqref{E:Re_SC}, and $x_3=\cos\theta$ to rewrite this equation as
  \begin{align*}
    (\mu-\cos\theta)\left\{\frac{1}{\sin\theta}\frac{d}{d\theta}\bigl(\sin\theta\,W'(\theta)\bigr)-\frac{m^2}{\sin^2\theta}W(\theta)\right\} = 6\cos\theta\,W(\theta)+6\mu\delta F(\theta)
  \end{align*}
  for $\theta\in(0,\pi)$. We multiply both sides by $\overline{W(\theta)}\sin\theta$ and integrate them over $(\theta_\mu,\pi)$. Then we carry out integration by parts and use $\mu-\cos\theta_\mu=0$ and $\sin\pi=0$ to get
  \begin{multline*}
    -\int_{\theta_\mu}^\pi W'(\theta)\overline{W(\theta)}\sin^2\theta\,d\theta-\int_{\theta_\mu}^\pi (\mu-\cos\theta)\left(|W'(\theta)|^2+\frac{m^2}{\sin^2\theta}|W(\theta)|^2\right)\sin\theta\,d\theta \\
    = 6\int_{\theta_\mu}^\pi|W(\theta)|^2\cos\theta\sin\theta\,d\theta+6\mu\delta\int_{\theta_\mu}^\pi F(\theta)\overline{W(\theta)}\sin\theta\,d\theta.
  \end{multline*}
  We take the real part of this equality and use
  \begin{align*}
    \mathrm{Re}\int_{\theta_\mu}^\pi W'(\theta)\overline{W(\theta)}\sin^2\theta\,d\theta = -\frac{1}{2}|W(\theta_\mu)|^2\sin^2\theta_\mu-\int_{\theta_\mu}^\pi|W(\theta)|^2\cos\theta\sin\theta\,d\theta
  \end{align*}
  by integration by parts and $\sin\pi=0$ to get
  \begin{align} \label{Pf_wLS:Int_Eq}
    J+K = -6\mu\delta\,\mathrm{Re}\int_{\theta_\mu}^\pi F(\theta)\overline{W(\theta)}\sin\theta\,d\theta+\frac{1}{2}|W(\theta_\mu)|^2\sin^2\theta_\mu,
  \end{align}
  where
  \begin{align} \label{Pf_wLS:Def_JK}
    \begin{aligned}
      J &= \int_{\theta_\mu}^\pi(\mu-\cos\theta)\left(|W'(\theta)|^2+\frac{m^2}{\sin^2\theta}|W(\theta)|^2\right)\sin\theta\,d\theta, \\
      K &= 5\int_{\theta_\mu}^\pi|W(\theta)|^2\cos\theta\sin\theta\,d\theta.
    \end{aligned}
  \end{align}
  Now we claim that there exist $\mu_0\in(0,1)$ and $C>0$ such that
  \begin{align} \label{Pf_wLS:JK_geq}
    J+K \geq C\|w\|_{L^2(S^2(\theta_\mu,\pi))}^2
  \end{align}
  for all $\mu\in[\mu_0,1)$. If this claim is valid, then we apply H\"{o}lder's and Young's inequalities, \eqref{E:L2_Mode}, \eqref{Pf_wLS:JK_geq}, $\mu\leq1$, and $\|f\|_{L^2(S^2(\theta_\mu,\pi))}\leq\|f\|_{L^2(S^2)}$ to \eqref{Pf_wLS:Int_Eq} to get
  \begin{align*}
    \|w\|_{L^2(S^2(\theta_\mu,\pi))}^2 \leq \frac{1}{2}\|w\|_{L^2(S^2(\theta_\mu,\pi))}^2+C\left(\delta^2\|f\|_{L^2(S^2)}^2+|W(\theta_\mu)|^2\sin^2\theta_\mu\right)
  \end{align*}
  and we subtract $\frac{1}{2}\|w\|_{L^2(S^2(\theta_\mu,\pi))}^2$ from both sides, dividing them by $1-\mu$, and use \eqref{E:KL_sinmu} and \eqref{E:KL_W_thmu} with $\sigma_1=\sigma_3^{1/2}$ to obtain \eqref{E:KL_wLS}.

  Let us show \eqref{Pf_wLS:JK_geq}. We split $K=K_1+K_2$ into
  \begin{align*}
    K_1 = 5\int_{\theta_\mu}^\Theta|W(\theta)|^2\cos\theta\sin\theta\,d\theta, \quad K_2 = 5\int_\Theta^\pi|W(\theta)|^2\cos\theta\sin\theta\,d\theta,
  \end{align*}
  where $\Theta\in(\theta_\mu,\pi/2)$ is fixed later. By $\cos\theta\geq\cos\Theta$ for $\theta\in[\theta_\mu,\Theta]$,
  \begin{align} \label{Pf_wLS:K1}
    K_1 \geq 5\cos\Theta\int_{\theta_\mu}^\Theta|W(\theta)|^2\sin\theta\,d\theta.
  \end{align}
  Since $w=W(\theta)e^{im\varphi}\in H^3(S^2)\subset C^1(S^2)$ and $m\neq0$, we have $W(\pi)=0$ by Lemma \ref{L:NS_Pole}. By this fact, $2\cos\theta\sin\theta=-\frac{d}{d\theta}(\cos^2\theta-\cos^2\Theta)$, and integration by parts, we have
  \begin{align*}
    K_2 &= -\frac{5}{2}\Bigl[(\cos^2\theta-\cos^2\Theta)|W(\theta)|^2\Bigr]_\Theta^\pi+5\int_\Theta^\pi(\cos^2\theta-\cos^2\Theta)\mathrm{Re}\Bigl(W'(\theta)\overline{W(\theta)}\Bigr)\,d\theta \\
    &= 5\int_\Theta^\pi\chi(\theta)(\mu-\cos\theta)\mathrm{Re}\Bigl(W'(\theta)\overline{W(\theta)}\Bigr)\,d\theta,
  \end{align*}
  where $\chi(\theta)=(\cos^2\theta-\cos^2\Theta)(\mu-\cos\theta)^{-1}$. Hence
  \begin{align} \label{Pf_wLS:K2}
    |K_2| \leq 5\|\chi\|_{L^\infty(\Theta,\pi)}\int_\Theta^\pi(\mu-\cos\theta)|W'(\theta)||W(\theta)|\,d\theta \leq \frac{5}{2}\|\chi\|_{L^\infty(\Theta,\pi)}\cdot J
  \end{align}
  by Young's inequality and $m^2\geq1$ (see \eqref{Pf_wLS:Def_JK}). To estimate $\|\chi\|_{L^\infty(\Theta,\pi)}$, let
  \begin{align*}
    \rho(t) = \chi(\arccos t) = \frac{t^2-T^2}{\mu-t}, \quad t \in[-1,T],
  \end{align*}
  where $T=\cos\Theta\in(0,\mu)$. Then
  \begin{align} \label{Pf_wLS:drho}
    \frac{d\rho}{dt}(t) = -\frac{(t-t_-)(t-t_+)}{(\mu-t)^2}, \quad t_{\pm} = \mu\pm\sqrt{\mu^2-T^2}.
  \end{align}
  Now let $\mu_0=\sqrt{23}/5$ and $T=\cos\Theta=\sqrt{7}/5$. Then for $\mu\in[\mu_0,1)$ we see by the mean value theorem for $\mu-\sqrt{\mu^2-s}$ with $s\in[0,T^2]$ and $\mu^2-T^2\geq16/25$ that
  \begin{align} \label{Pf_wLS:inq_T}
    \mu-\sqrt{\mu^2-T^2} \leq \frac{T^2}{2\sqrt{\mu^2-T^2}} \leq \frac{5}{8}T^2
  \end{align}
  and thus $0<t_-<T<t_+$ by $T<\mu<1$. Hence, by \eqref{Pf_wLS:drho} and $\rho(T)=0$,
  \begin{align*}
    \|\chi\|_{L^\infty(\Theta,\pi)} &= \|\rho\|_{L^\infty(-1,T)} = \max\{|\rho(-1)|,|\rho(t_-)|\} \\
    &= \max\left\{\frac{1-T^2}{1+\mu},2\left(\mu-\sqrt{\mu^2-T^2}\right)\right\}.
  \end{align*}
  Moreover, by $\mu\geq\mu_0=\sqrt{23}/5$, $T^2=7/25$, \eqref{Pf_wLS:inq_T}, and $4<\sqrt{23}<5$,
  \begin{align*}
    \frac{5}{2}\cdot\frac{1-T^2}{1+\mu} \leq \frac{9}{5+\sqrt{23}}, \quad \frac{5}{2}\cdot 2\left(\mu-\sqrt{\mu^2-T^2}\right) \leq \frac{7}{8}, \quad \frac{7}{8} < \frac{9}{10} < \frac{9}{5+\sqrt{23}} < 1.
  \end{align*}
  Hence $\frac{5}{2}\|\chi\|_{L^\infty(\Theta,\pi)}\leq9/(5+\sqrt{23})$ and
  \begin{align} \label{Pf_wLS:JK2}
    J+K_2 \geq J-|K_2| \geq C_0J
  \end{align}
  with $C_0=1-9/(5+\sqrt{23})>0$ by \eqref{Pf_wLS:K2}. Moreover,
  \begin{align} \label{Pf_wLS:Jgeq}
    J \geq \int_\Theta^\pi\frac{\mu-\cos\theta}{\sin^2\theta}|W(\theta)|^2\sin\theta\,d\theta \geq \frac{\sqrt{23}-\sqrt{7}}{5}\int_\Theta^\pi|W(\theta)|^2\sin\theta\,d\theta
  \end{align}
  by $m^2\geq1$, $\sin^2\theta\leq1$ and $\cos\theta\leq\cos\Theta=\sqrt{7}/5$ for $\theta\in[\Theta,\pi]$, and $\mu\geq\mu_0=\sqrt{23}/5$. Thus, noting that $K=K_1+K_2$, we combine \eqref{Pf_wLS:K1}, \eqref{Pf_wLS:JK2}, and \eqref{Pf_wLS:Jgeq} and apply \eqref{E:L2_Mode} to $w=W(\theta)e^{im\varphi}$ to get \eqref{Pf_wLS:JK_geq}.
\end{proof}

Now let us prove Lemma \ref{L:Ko_Low}. We give the proofs of \eqref{E:Ko_Lowu} and \eqref{E:Ko_LowS} separately.

\begin{proof}[Proof of \eqref{E:Ko_Lowu}]
  As we mentioned above, we assume $\mu\geq0$ and prove \eqref{E:Ko_Lowu} by contradiction. Assume that for each $k\in\mathbb{N}$ there exist $m_k\in\mathbb{Z}\setminus\{0\}$, $u_k\in\mathcal{Y}_{m_k}\cap H^1(S^2)$, $\delta_k\in(0,1]$, and $\mu_k\in[0,1-\kappa\delta_k^2]$ such that, if we define $v_k$, $w_k$, $f_k$, and $R_k=R(u_k,\delta_k,\mu_k)$ by \eqref{E:KL_vw} and \eqref{E:KL_fR} with $m$, $u$, $\delta$, and $\mu$ replaced by $m_k$, $u_k$, $\delta_k$, and $\mu_k$, then
  \begin{gather}
    \delta_k^2\|u_k\|_{L^2(S^2)}^2+\|(-\Delta)^{1/2}v_k\|_{L^2(S^2)}^2+\frac{1}{1-\mu_k}\|w_k\|_{L^2(S^2)}^2 = 1, \label{Pf_KL:As_u} \\
    \lim_{k\to\infty}R_k = \lim_{k\to\infty}\left(\|f_k\|_{L^2(S^2)}+\frac{\delta_k^3}{(1-\mu_k)^{1/2}}\|(-\Delta)^{1/2}u_k\|_{L^2(S^2)}\right) = 0. \label{Pf_KL:As_Rk}
  \end{gather}
  Taking subsequences, we may assume that
  \begin{gather*}
    \lim_{k\to\infty}m_k = m_\infty \in \{\pm\infty\}\cup\mathbb{Z}\setminus\{0\}, \\
    \lim_{k\to\infty}\delta_k = \delta_\infty \in [0,1], \quad \lim_{k\to\infty}\mu_k = \mu_\infty \in [0,1-\kappa\delta_\infty^2].
  \end{gather*}
  Suppose first that $\delta_\infty>0$. Then $\delta_k\geq\delta_\infty/2>0$ for sufficiently large $k$. Moreover, we see by \eqref{E:KL_w_est} and the boundedness of $\Delta^{-1}$ on $L_0^2(S^2)$ that
  \begin{align*}
    \|w_k\|_{L^2(S^2)} \leq C\|(-\Delta)^{1/2}v_k\|_{L^2(S^2)} = C\|(-\Delta)^{-1/2}u_k\|_{L^2(S^2)} \leq C\|(-\Delta)^{1/2}u_k\|_{L^2(S^2)}.
  \end{align*}
  By these facts, $(1-\mu_k)^{-1}\geq1$, \eqref{E:H1_Equiv}, and \eqref{Pf_KL:As_Rk}, we find that
  \begin{multline*}
    \delta_k^2\|u_k\|_{L^2(S^2)}^2+\|(-\Delta)^{1/2}v_k\|_{L^2(S^2)}^2+\frac{1}{1-\mu_k}\|w_k\|_{L^2(S^2)}^2 \leq \frac{C\delta_k^6}{1-\mu_k}\|(-\Delta)^{1/2}u_k\|_{L^2(S^2)}^2 \to 0
  \end{multline*}
  as $k\to\infty$, which contradicts \eqref{Pf_KL:As_u}.

  Now let $\delta_\infty=0$. We see by \eqref{E:KL_u} for $u_k$, \eqref{Pf_KL:As_u}, and \eqref{Pf_KL:As_Rk} that
  \begin{align} \label{Pf_KL:As_linf}
    \liminf_{k\to\infty}\left(\|(-\Delta)^{1/2}v_k\|_{L^2(S^2)}^2+\frac{1}{1-\mu_k}\|w_k\|_{L^2(S^2)}^2\right) \geq C > 0.
  \end{align}
  Let $\theta_k=\arccos\mu_k\in(0,\pi/2]$. Then we see by Lemma \ref{L:KL_WF} that
  \begin{align} \label{Pf_KL:Wk_Fk}
    w_k = W_k(\theta)e^{im_k\varphi}, \quad f_k = F_k(\theta)e^{im_k\varphi}, \quad W_k(\theta_k)+\delta_kF_k(\theta_k) = 0
  \end{align}
  Also, by \eqref{E:KL_W_thmu} for $W_k(\theta_k)$, \eqref{Pf_KL:As_u}, and \eqref{Pf_KL:As_Rk},
  \begin{align*}
    \limsup_{k\to\infty}|W_k(\theta_k)| \leq C\limsup_{k\to\infty}\Bigl(\sigma^{-1}R_k+\sigma\|(-\Delta)^{1/2}v_k\|_{L^2(S^2)}\Bigr) \leq C\sigma
  \end{align*}
  for all $\sigma\in(0,1/\sqrt{2})$ and thus
  \begin{align} \label{Pf_KL:Wk_Thk}
    \lim_{k\to\infty}W_k(\theta_k) = 0.
  \end{align}
  In what follows, we consider two cases.

  \textbf{Case 1:} $0\leq\mu_\infty<1$. If $m_\infty=\pm\infty$, then $|m_k|\geq3$ for sufficiently large $k\in\mathbb{N}$. Then since $v_k\in\mathcal{Y}_{m_k}$ is of the form \eqref{E:Ko_Ym} with $N_{m_k}=|m_k|$, it follows from \eqref{E:Def_Laps}, $\lambda_n=n(n+1)\geq|m_k|^2$ for $n\geq|m_k|$, and \eqref{Pf_KL:As_u} that
  \begin{align} \label{Pf_KL:0in_vk}
    \|v_k\|_{L^2(S^2)} \leq \frac{1}{|m_k|}\|(-\Delta)^{1/2}v_k\|_{L^2(S^2)} \leq \frac{1}{|m_k|} \to 0 \quad\text{as}\quad k\to\infty.
  \end{align}
  We set $\psi=v_k$ in \eqref{E:KL_vH1} for $v_k$ and use \eqref{E:Lap_Half} and \eqref{Pf_KL:As_u} to find that
  \begin{align*}
    \|(-\Delta)^{1/2}v_k\|_{L^2(S^2)}^2 \leq 6\|v_k\|_{L^2(S^2)}^2+C(\sigma^{-2}R_k+\sigma)+\frac{C}{(1-\mu_k)^{3/4}}(\sigma^{-4}R_k+1)\|v_k\|_{L^2(S^2)}
  \end{align*}
  for all $\sigma\in(0,1/2)$. Then we send $k\to\infty$ and use \eqref{Pf_KL:As_Rk}, \eqref{Pf_KL:0in_vk}, and $\mu_\infty<1$ to get
  \begin{align*}
    \limsup_{k\to\infty}\|(-\Delta)^{1/2}v_k\|_{L^2(S^2)}^2 \leq C\sigma \quad\text{for all}\quad \sigma\in\left(0,\frac{1}{2}\right).
  \end{align*}
  Hence $\lim_{k\to\infty}\|(-\Delta)^{1/2}v_k\|_{L^2(S^2)}=0$. Also, $\lim_{k\to\infty}\|w_k\|_{L^2(S^2)}=0$ by \eqref{E:KL_w_est} and \eqref{Pf_KL:0in_vk}. By these facts and $\mu_\infty<1$ we get a contradiction with \eqref{Pf_KL:As_linf}.

  Next suppose that $m_\infty\in\mathbb{Z}\setminus\{0\}$. Taking a subsequence, we may assume $m_k=m_\infty$ and thus $v_k\in\mathcal{Y}_{m_\infty}$ for all $k\in\mathbb{N}$. Since $\{v_k\}_{k=1}^\infty$ is bounded in $H^1(S^2)$ by \eqref{E:H1_Equiv} and \eqref{Pf_KL:As_u}, and since $H^1(S^2)$ is compactly embedded into $L^2(S^2)$, we may assume that
  \begin{align} \label{Pf_KL:0fi_vkli}
    \lim_{k\to\infty}v_k = v_\infty \quad\text{strongly in $L^2(S^2)$ and weakly in $H^1(S^2)$}
  \end{align}
  with some $v_\infty\in H^1(S^2)$ by taking a subsequence again. Then $v_\infty\in\mathcal{Y}_{m_\infty}$ since $v_k\in\mathcal{Y}_{m_\infty}$ for all $k\in\mathbb{N}$ and $\mathcal{Y}_{m_\infty}$ is closed in $L^2(S^2)$. If $v_\infty=0$, i.e. $\|v_k\|_{L^2(S^2)}\to0$ as $k\to\infty$, then we can get a contradiction with \eqref{Pf_KL:As_linf} as in the case of $m_\infty=\pm\infty$. Suppose that $v_\infty\neq0$. Since $u_k=\Delta v_k$ and $\Delta Y_3^{m_\infty}=-\lambda_3Y_3^{m_\infty}$, where we consider $Y_n^{m_\infty}\equiv0$ if $|m_\infty|>n$,
  \begin{align*}
    (u_k,Y_3^{m_\infty})_{L^2(S^2)} = (\Delta v_k,Y_3^{m_\infty})_{L^2(S^2)} = (v_k,\Delta Y_3^{m_\infty})_{L^2(S^2)} = -\lambda_3(v_k,Y_3^{m_\infty})_{L^2(S^2)}.
  \end{align*}
  Hence $w_k=6\mu_kv_k+\frac{1}{2}a_3^{m_\infty}\lambda_3(v_k,Y_3^{m_\infty})_{L^2(S^2)}Y_2^{m_\infty}$ by \eqref{E:KL_vw} for $w_k$, and
  \begin{align} \label{Pf_KL:0fi_wkde}
    \lim_{k\to\infty}w_k = w_\infty = 6\mu_\infty v_\infty+\frac{1}{2}a_3^{m_\infty}\lambda_3(v_\infty,Y_3^{m_\infty})_{L^2(S^2)}Y_2^{m_\infty} \quad\text{strongly in $L^2(S^2)$}
  \end{align}
  by \eqref{Pf_KL:0fi_vkli}, where $a_3^{m_\infty}$ is given by \eqref{E:Y_Rec} if $|m_\infty|=1,2$ and $a_3^{m_\infty}=0$ if $|m_\infty|\geq3$. For each $\psi\in H^1(S^2)$ and $\sigma\in(0,1/2)$, we see by \eqref{E:KL_vH1} for $v_k$ and \eqref{Pf_KL:As_u} that
  \begin{multline*}
    \bigl|(\nabla v_k,\nabla\psi)_{L^2(S^2)}\bigr| \leq 6\|v_k\|_{L^2(S^2)}\|\psi\|_{L^2(S^2)}+C(\sigma^{-2}R_k+\sigma)\|\nabla\psi\|_{L^2(S^2)}\\
    +\frac{C}{(1-\mu_k)^{3/4}}(\sigma^{-4}R_k+1)\|\psi\|_{L^2(S^2)}.
  \end{multline*}
  We send $k\to\infty$, use \eqref{Pf_KL:As_Rk}, \eqref{Pf_KL:0fi_vkli}, and $\mu_\infty<1$, and then let $\sigma\to0$ to get
  \begin{align*}
    \bigl|(\nabla v_\infty,\nabla\psi)_{L^2(S^2)}\bigr| \leq C\left\{\frac{1}{(1-\mu_\infty)^{3/4}}+\|v_\infty\|_{L^2(S^2)}\right\}\|\psi\|_{L^2(S^2)}
  \end{align*}
  for all $\psi\in H^1(S^2)$, which shows that $\Delta v_\infty\in L^2(S^2)$ since $H^1(S^2)$ is dense in $L^2(S^2)$. Hence $v_\infty\in H^2(S^2)$ by the elliptic regularity theorem. Now let $\psi\in C^\infty(S^2)$ satisfy
  \begin{align} \label{Pf_KL:0fi_test}
    \psi = 0 \quad\text{near}\quad \{(x_1,x_2,x_3)\in S^2 \mid x_3=\mu_\infty\}.
  \end{align}
  Then, as in the proof of Lemma \ref{L:KL_vH1}, we can get (see \eqref{Pf_vH1:IBP})
  \begin{align*}
    (\nabla v_k,\nabla\psi)_{L^2(S^2)}-6(v_k,\psi)_{L^2(S^2)} = -\left(w_k+\delta_kf_k,\frac{\psi}{\mu_k-x_3}\right)_{L^2(S^2)}.
  \end{align*}
  We send $k\to\infty$ and use \eqref{Pf_KL:As_Rk}, \eqref{Pf_KL:0fi_vkli}, \eqref{Pf_KL:0fi_wkde}, and $\psi/(\mu_k-x_3)\to\psi/(\mu_\infty-x_3)$ uniformly on $S^2$ by \eqref{Pf_KL:0fi_test}. Then we get
  \begin{align*}
    (\nabla v_\infty,\nabla\psi)_{L^2(S^2)}-6(v_\infty,\psi)_{L^2(S^2)} = -\left(w_\infty,\frac{\psi}{\mu_\infty-x_3}\right)_{L^2(S^2)} = -\left(\frac{w_\infty}{\mu_\infty-x_3},\psi\right)_{L^2(S^2)}
  \end{align*}
  for all $\psi\in C^\infty(S^2)$ satisfying \eqref{Pf_KL:0fi_test}. Since $v_\infty\in H^2(S^2)$, this equality shows that
  \begin{align} \label{Pf_KL:0fi_viEq}
    \Delta v_\infty+6v_\infty = \frac{w_\infty}{\mu_\infty-x_3} \quad\text{on}\quad \{(x_1,x_2,x_3)\in S^2 \mid x_3\neq\mu_\infty\}.
  \end{align}
  Let $u_\infty=\Delta v_\infty$. Then $u_\infty\in\mathcal{Y}_{m_\infty}$ and $u_\infty\neq0$ since $v_\infty\in\mathcal{Y}_{m_\infty}\cap H^2(S^2)$, $v_\infty\neq0$, and $\Delta$ maps $\mathcal{Y}_{m_\infty}$ into itself and is invertible on $\mathcal{Y}_{m_\infty}$ by \eqref{E:Lap_Y} and \eqref{E:Ko_Ym}. Moreover, we see by \eqref{Pf_KL:0fi_wkde}, \eqref{Pf_KL:0fi_viEq}, and $B_{2,m_{\infty}}=(I+6\Delta^{-1})|_{\mathcal{X}_{m_\infty}}$ that $u_\infty$ satisfies
  \begin{align} \label{Pf_KL:0fi_uinf}
    \Lambda_{m_\infty}u_\infty = x_3B_{2,m_\infty}u_\infty = \mu_\infty u_\infty+\beta_\infty Y_2^{m_\infty} \quad\text{in}\quad \mathcal{X}_{m_\infty},
  \end{align}
  where $\beta_\infty=-\frac{1}{2}a_3^{m_\infty}\lambda_3(v_\infty,Y_3^{m_\infty})_{L^2(S^2)}$.

  Now suppose that $|m_\infty|\geq3$. Then $\Lambda_{m_\infty}u_\infty=\mu_\infty u_\infty$ in $\mathcal{X}_{m_\infty}$ by \eqref{Pf_KL:0fi_uinf} and $Y_2^{m_\infty}\equiv0$. Thus, if $\mu_\infty>0$, then $\mu_\infty$ is a nonzero eigenvalue of $\Lambda_{m_\infty}$ since $u_\infty\neq0$, but this contradicts Lemma \ref{L:NoEi_Lam}. Also, if $\mu_\infty=0$, then $u_\infty$ is in the kernel of $\Lambda_{m_\infty}$ and thus $u_\infty=0$ by \eqref{Pf_K13:Ker}, which contradicts $u_\infty\neq0$.

  Let $|m_\infty|=1,2$ and $\mu_\infty>0$. Then, by \eqref{Pf_KL:0fi_uinf} and $\Lambda_{m_\infty}Y_2^{m_\infty}=0$,
  \begin{align*}
    \Lambda_{m_\infty}(u_\infty+\mu_\infty^{-1}\beta_\infty Y_2^{m_\infty}) = \Lambda_{m_\infty}u_\infty = \mu_\infty u_\infty+\beta_\infty Y_2^{m_\infty} = \mu_\infty(u_\infty+\mu_\infty^{-1}\beta_\infty Y_2^{m_\infty}).
  \end{align*}
  Moreover, $u_\infty+\mu_\infty^{-1}\beta_\infty Y_2^{m_\infty}\neq0$ since $u_\infty\neq0$ and $(u_\infty,Y_2^{m_\infty})_{L^2(S^2)}=0$ by $u_\infty\in\mathcal{Y}_{m_\infty}$. Hence $\mu_\infty$ is a nonzero eigenvalue of $\Lambda_{m_\infty}$, which contradicts Lemma \ref{L:NoEi_Lam}.

  Suppose that $|m_\infty|=1$ and $\mu_\infty=0$. Since $Y_2^{\pm1}=\cos\theta\,Y_1^{\pm1}/a_2^{\pm1}$ by \eqref{E:Y_Rec} with $(n,m)=(1,\pm1)$ and $Y_0^{\pm1}\equiv0$, we apply this equality and $x_3=\cos\theta$ to \eqref{Pf_KL:0fi_uinf} to get
  \begin{align*}
    x_3\left\{B_{2,\pm1}u_\infty-\frac{\beta_\infty}{a_2^{\pm1}}Y_1^{\pm1}\right\} = 0, \quad\text{i.e.}\quad B_{2,\pm1}u_\infty-\frac{\beta_\infty}{a_2^{\pm1}}Y_1^{\pm1} = 0 \quad\text{on}\quad S^2.
  \end{align*}
  Then since $u\in\mathcal{Y}_{\pm1}$ is of the form \eqref{E:Ko_Ym}, we have
  \begin{align*}
    \|B_{2,\pm1}u_\infty\|_{L^2(S^2)}^2 &= \frac{\beta_\infty}{a_2^{\pm1}}(B_{2,\pm1}u_\infty,Y_1^{\pm1})_{L^2(S^2)} \\
    &= \frac{\beta_\infty}{a_2^{\pm1}}\sum_{n\geq3}\left(1-\frac{6}{\lambda_n}\right)(u_\infty,Y_n^{\pm1})_{L^2(S^2)}(Y_n^{\pm1},Y_1^{\pm1})_{L^2(S^2)} = 0
  \end{align*}
  by \eqref{E:Lap_Y} (note that $B_{2,\pm1}=B|_{\mathcal{X}_{\pm1}}$). Since $u_\infty\in\mathcal{Y}_{\pm1}$, we find by this equality and \eqref{E:B2m_Ym} that $u_\infty=0$, which contradicts $u_\infty\neq0$.

  Let $|m_\infty|=2$ and $\mu_\infty=0$. Since $Y_2^{\pm2}$ is of the form $Y_2^{\pm2}=C_{\pm2}\sin^2\theta e^{\pm2i\varphi}$ with a constant $C_{\pm2}\in\mathbb{R}\setminus\{0\}$ by \eqref{E:SpHa}, we see by \eqref{Pf_KL:0fi_uinf} and $x_3=\cos\theta$ that
  \begin{align*}
    B_{2,\pm2}u_\infty(\theta,\varphi) = \beta_\infty C_{\pm2}\frac{\sin^2\theta}{\cos\theta}e^{\pm2i\varphi}, \quad (\theta,\varphi)\in[0,\pi]\times[0,2\pi), \, \theta \neq \frac{\pi}{2}
  \end{align*}
  and thus $\beta_\infty=0$, otherwise $B_{2,\pm2}u_\infty$ does not belong to $L^2(S^2)$. Hence $B_{2,\pm2}u_\infty=0$ by the above equality. Since $u_\infty\in\mathcal{Y}_{\pm2}$, this fact and \eqref{E:B2m_Ym} imply that $u_\infty=0$, which contradicts $u_\infty\neq0$. This completes the proof in Case 1.

  \textbf{Case 2:} $\mu_\infty=1$. In this case, $\mu_k\geq\mu_0$ for sufficiently large $k\in\mathbb{N}$, where $\mu_0\in(0,1)$ is given in Lemma \ref{L:KL_wLS}. Then we see by \eqref{E:KL_wLS} for $w_k$, \eqref{Pf_KL:As_u}, and \eqref{Pf_KL:As_Rk} that
  \begin{align*}
    \limsup_{k\to\infty}\frac{1}{1-\mu_k}\|w_k\|_{L^2(S^2(\theta_k,\pi))}^2 \leq C\limsup_{k\to\infty}\Bigl(\sigma^{-1}R_k^2+\sigma\|(-\Delta)^{1/2}v_k\|_{L^2(S^2)}^2\Bigr) \leq C\sigma
  \end{align*}
  for all $\sigma\in(0,1/2)$, where $S^2(\theta_k,\pi)$ is given by \eqref{E:Def_Band}. Hence
  \begin{align} \label{Pf_KL:1_wLS}
    \lim_{k\to\infty}\frac{1}{1-\mu_k}\|w_k\|_{L^2(S^2(\theta_k,\pi))}^2 = 0.
  \end{align}
  Suppose $\lim_{k\to\infty}(1-\mu_k)^{-1}\|w_k\|_{L^2(S^2(0,\theta_k))}^2=0$. Then
  \begin{align} \label{Pf_KL:1_wkco}
    \lim_{k\to\infty}\frac{1}{1-\mu_k}\|w_k\|_{L^2(S^2)}^2 = \lim_{k\to\infty}\frac{1}{1-\mu_k}\Bigl(\|w_k\|_{L^2(S^2(0,\theta_k))}^2+\|w_k\|_{L^2(S^2(\theta_k,\pi))}^2\Bigr) = 0.
  \end{align}
  Also, for each $\sigma\in(0,1/2)$ we see by \eqref{E:KL_vLaph} for $v_k$ and \eqref{Pf_KL:As_u} that
  \begin{align*}
    3\mu_k\|(-\Delta)^{1/2}v_k\|_{L^2(S^2)}^2 \leq C\left(\sigma^{-2}R_k+\sigma+\frac{1}{\sigma^2(1-\mu_k)}\|w_k\|_{L^2(S^2(0,\theta_k))}^2\right)
  \end{align*}
  We send $k\to\infty$ and use \eqref{Pf_KL:As_Rk} and $(1-\mu_k)^{-1}\|w_k\|_{L^2(S^2(0,\theta_k))}^2\to0$ to get
  \begin{align*}
    3\mu_\infty\limsup_{k\to\infty}\|(-\Delta)^{1/2}v_k\|_{L^2(S^2)}^2 = 3\limsup_{k\to\infty}\|(-\Delta)^{1/2}v_k\|_{L^2(S^2)}^2 \leq C\sigma
  \end{align*}
  for all $\sigma\in(0,1/2)$. Thus $\lim_{k\to\infty}\|(-\Delta)^{1/2}v_k\|_{L^2(S^2)}^2=0$ and we combine this with \eqref{Pf_KL:1_wkco} to get a contradiction with \eqref{Pf_KL:As_linf}. Now we assume that
  \begin{align} \label{Pf_KL:1_winf}
    \inf_{k\in\mathbb{N}}\frac{1}{1-\mu_k}\|w_k\|_{L^2(S^2(0,\theta_k))}^2 > 0.
  \end{align}
  Since $\mu_k\in[0,1)$ and $\theta_k=\arccos\mu_k\in(0,\pi/2]$, we have
  \begin{align*}
    \theta_k \leq \frac{\pi}{2}\sin\theta_k = \frac{\pi}{2}(1-\mu_k^2)^{1/2} = \frac{\pi}{2}(1-\mu_k)^{1/2}(1+\mu_k)^{1/2} \leq \frac{\sqrt{2}}{2}\pi(1-\mu_k)^{1/2}.
  \end{align*}
  Moreover, by Taylor's theorem for $\cos\theta$ at $\theta=0$,
  \begin{align*}
    \mu_k = \cos\theta_k = 1-\frac{1}{2}\theta_k^2+\frac{1}{24}\theta_k^4\cos(\tau_k\theta_k), \quad \tau_k\in(0,1).
  \end{align*}
  By the above relations, $|\cos(\tau_k\theta_k)|\leq1$, and $\mu_\infty=1$, we find that
  \begin{align*}
    \left|\frac{\theta_k^2}{1-\mu_k}-2\right| = \left|\frac{\theta_k^4\cos(\tau_k\theta_k)}{12(1-\mu_k)}\right| \leq C(1-\mu_k) \to 0 \quad\text{as}\quad k\to\infty.
  \end{align*}
  Hence $\theta_k/(1-\mu_k)^{1/2}\to\sqrt{2}$ as $k\to\infty$. Based on this observation, we set
  \begin{align} \label{Pf_KL:1_Def_ga}
    \gamma_k = \frac{4\sqrt{2}}{\pi}(1-\mu_k)^{1/2}, \quad \tilde{\theta}_k = \frac{\theta_k}{\gamma_k} = \frac{\pi}{4\sqrt{2}}\cdot\frac{\theta_k}{(1-\mu_k)^{1/2}}.
  \end{align}
  Then
  \begin{align} \label{Pf_KL:1_GtTht}
    \lim_{k\to\infty}\gamma_k = 0, \quad \lim_{k\to\infty}\tilde{\theta}_k = \tilde{\theta}_\infty = \frac{\pi}{4}
  \end{align}
  and thus $\gamma_k\in(0,1)$ and $\tilde{\theta}_k\in(\pi/6,\pi/3)$ for sufficiently large $k\in\mathbb{N}$. Also,
  \begin{align} \label{Pf_KL:1_sin}
    \frac{2}{\pi}\cdot\frac{1}{\gamma_k}\sin\theta \leq \sin\left(\frac{\theta}{\gamma_k}\right) \leq \frac{\pi}{2}\cdot\frac{1}{\gamma_k}\sin\theta, \quad \theta \in \left(0,\frac{\pi}{2}\gamma_k\right)
  \end{align}
  by $2\theta/\pi\leq\sin\theta\leq\theta$ for $\theta\in(0,\pi/2)$. We define
  \begin{gather}
    \widetilde{W}_k(\tilde{\theta}) = W_k(\gamma_k\tilde{\theta}), \quad \widetilde{F}_k(\tilde{\theta}) = F_k(\gamma_k\tilde{\theta}), \quad \tilde{\theta} \in \left(0,\frac{\pi}{2}\right), \label{Pf_KL:Def_Wtk} \\
    \tilde{w}_k = \widetilde{W}_k(\tilde{\theta})e^{im_k\varphi}, \quad \tilde{f}_k = \widetilde{F}_k(\tilde{\theta})e^{im_k\varphi} \quad\text{on}\quad S_+^2 = \{(x_1,x_2,x_3)\in S^2 \mid x_3 > 0\}. \notag
  \end{gather}
  Then, by $\widetilde{W}_k(\tilde{\theta}_k)=W_k(\theta_k)$, $\widetilde{F}_k(\tilde{\theta}_k)=F_k(\theta_k)$, \eqref{Pf_KL:Wk_Fk}, and \eqref{Pf_KL:Wk_Thk},
  \begin{align} \label{Pf_KL:1_tilWF}
    \widetilde{W}_k(\tilde{\theta}_k)+\delta_k\widetilde{F}_k(\tilde{\theta}_k) = 0, \quad k\in\mathbb{N}, \quad \lim_{k\to\infty}\widetilde{W}_k(\tilde{\theta}_k) = 0.
  \end{align}
  Also, we use \eqref{E:L2_Mode}, set $\theta=\gamma_k\tilde{\theta}$, and apply \eqref{Pf_KL:1_Def_ga} and \eqref{Pf_KL:1_sin} to get
  \begin{align} \label{Pf_KL:1_Int_L2}
    \begin{aligned}
      \|\tilde{w}_k\|_{L^2(S_+^2)}^2 &= 2\pi\int_0^{\frac{\pi}{2}}\Bigl|\widetilde{W}_k(\tilde{\theta})\Bigr|^2\sin\tilde{\theta}\,d\tilde{\theta} = 2\pi\int_0^{\frac{\pi}{2}\gamma_k}|W_k(\theta)|^2\sin\left(\frac{\theta}{\gamma_k}\right)\frac{d\theta}{\gamma_k} \\
      &\leq \frac{C}{\gamma_k^2}\int_0^{\frac{\pi}{2}\gamma_k}|W_k(\theta)|^2\sin\theta\,d\theta \leq \frac{C}{1-\mu_k}\|w_k\|_{L^2(S^2)}^2.
    \end{aligned}
  \end{align}
  Similarly, since $\widetilde{W}'_k(\tilde{\theta})=\gamma_kW'_k(\gamma_k\tilde{\theta})$, we have
  \begin{align} \label{Pf_KL:1_Int_H1}
    \begin{aligned}
      \|\nabla\tilde{w}_k\|_{L^2(S_+^2)}^2 &= 2\pi\int_0^{\frac{\pi}{2}}\left(\Bigl|\widetilde{W}'_k(\tilde{\theta})\Bigr|^2+\frac{m_k^2}{\sin^2\tilde{\theta}}\Bigl|\widetilde{W}_k(\tilde{\theta})\Bigr|^2\right)\sin\tilde{\theta}\,d\tilde{\theta} \\
      &= 2\pi\int_0^{\frac{\pi}{2}\gamma_k}\left(\gamma_k^2|W'_k(\theta)|^2+\frac{m_k^2}{\sin^2(\theta/\gamma_k)}|W_k(\theta)|^2\right)\sin\left(\frac{\theta}{\gamma_k}\right)\frac{d\theta}{\gamma_k} \\
      &\leq C\int_0^{\frac{\pi}{2}\gamma_k}\left(|W'_k(\theta)|^2+\frac{m_k^2}{\sin^2\theta}|W_k(\theta)|^2\right)\sin\theta\,d\theta \leq C\|\nabla w_k\|_{L^2(S^2)}^2.
    \end{aligned}
  \end{align}
  By these inequalities, \eqref{E:KL_w_est}, and \eqref{Pf_KL:As_u}, we find that
  \begin{align} \label{Pf_KL:1_wtk_H1}
    \|\tilde{w}_k\|_{H^1(S_+^2)}^2 \leq C\left(\frac{1}{1-\mu_k}\|w_k\|_{L^2(S^2)}^2+\|(-\Delta)^{1/2}v_k\|_{L^2(S^2)}^2\right) \leq C.
  \end{align}
  Also, as in \eqref{Pf_KL:1_Int_L2}, we use \eqref{E:L2_Mode}, \eqref{Pf_KL:1_Def_ga}, \eqref{Pf_KL:1_sin}, and $\gamma_k\tilde{\theta}_k=\theta_k$ to get
  \begin{align*}
    \|\tilde{w}_k\|_{L^2(S^2(0,\tilde{\theta}_k))}^2 \geq \frac{C}{1-\mu_k}\|w_k\|_{L^2(S^2(0,\theta_k))}^2.
  \end{align*}
  By this inequality and \eqref{Pf_KL:1_winf}, we obtain
  \begin{align} \label{Pf_KL:1_wtk_inf}
    \inf_{k\in\mathbb{N}}\|\tilde{w}_k\|_{L^2(S_+^2)}^2 \geq \inf_{k\in\mathbb{N}}\|\tilde{w}_k\|_{L^2(S^2(0,\tilde{\theta}_k))}^2 \geq C\inf_{k\in\mathbb{N}}\frac{1}{1-\mu_k}\|w_k\|_{L^2(S^2(0,\theta_k))}^2 > 0.
  \end{align}
  Suppose that $m_\infty=\pm\infty$. Then since
  \begin{align*}
    \|\nabla\tilde{w}_k\|_{L^2(S_+^2)}^2 &= 2\pi\int_0^{\frac{\pi}{2}}\left(\Bigl|\widetilde{W}'_k(\tilde{\theta})\Bigr|^2+\frac{m_k^2}{\sin^2\tilde{\theta}}\Bigl|\widetilde{W}_k(\tilde{\theta})\Bigr|^2\right)\sin\tilde{\theta}\,d\tilde{\theta} \\
    &\geq m_k^2\cdot2\pi\int_0^{\frac{\pi}{2}}\Bigl|\widetilde{W}_k(\tilde{\theta})\Bigr|^2\sin\tilde{\theta}\,d\tilde{\theta} = m_k^2\|\tilde{w}_k\|_{L^2(S_+^2)}^2
  \end{align*}
  by \eqref{E:L2_Mode} and $\sin^2\tilde{\theta}\leq1$, it follows from \eqref{Pf_KL:1_wtk_H1} that
  \begin{align*}
    \|\tilde{w}_k\|_{L^2(S_+^2)}^2 \leq \frac{1}{m_k^2}\|\nabla\tilde{w}_k\|_{L^2(S_+^2)}^2 \leq \frac{C}{m_k^2} \to 0 \quad\text{as}\quad k\to\infty,
  \end{align*}
  which contradicts \eqref{Pf_KL:1_wtk_inf}.

  Now let $m_\infty\in\mathbb{Z}\setminus\{0\}$. Taking a subsequence, we may assume that
  \begin{align} \label{Pf_KL:1_minf}
    m_k = m_\infty, \quad \tilde{w}_k = \widetilde{W}_k(\tilde{\theta})e^{im_\infty\varphi}, \quad \tilde{f}_k = \widetilde{F}_k(\tilde{\theta})e^{im_\infty\varphi} \quad\text{for all}\quad k\in\mathbb{N}.
  \end{align}
  Since $\{\tilde{w}_k\}_{k=1}^\infty$ is bounded in $H^1(S_+^2)$ by \eqref{Pf_KL:1_wtk_H1} and $H^1(S_+^2)$ is compactly embedded into $L^2(S_+^2)$, we may assume, by taking a subsequence again, that
  \begin{align} \label{Pf_KL:1_wtk_co}
    \lim_{k\to\infty}\tilde{w}_k = \tilde{w}_\infty \quad\text{strongly in $L^2(S_+^2)$ and weakly in $H^1(S_+^2)$}.
  \end{align}
  with some $\tilde{w}_\infty\in H^1(S_+^2)$. Then $\tilde{w}_\infty\neq0$ by \eqref{Pf_KL:1_wtk_inf} and \eqref{Pf_KL:1_wtk_co}. For each $m\in\mathbb{Z}$ we see by \eqref{E:Proj_Bdd} and \eqref{Pf_KL:1_wtk_co} that $\{\mathcal{P}_m\tilde{w}_k\}_{k=1}^\infty$ converges to $\mathcal{P}_m\tilde{w}_\infty$ strongly in $L^2(S_+^2)$. Moreover, $\mathcal{P}_{m_\infty}\tilde{w}_k=\tilde{w}_k$ and $\mathcal{P}_m\tilde{w}_k=0$ for $m\neq m_\infty$ by \eqref{Pf_KL:1_minf}. Hence
  \begin{align*}
    \mathcal{P}_m\tilde{w}_\infty = 0 \quad\text{for}\quad m\neq m_\infty, \quad \tilde{w}_\infty = \mathcal{P}_{m_\infty}\tilde{w}_\infty = \widetilde{W}_\infty(\tilde{\theta})e^{im_\infty\varphi}.
  \end{align*}
  Let us show $\widetilde{W}_\infty=0$. First we prove
  \begin{align} \label{Pf_KL:1_Wtin_H2}
    \widetilde{W}_\infty \in H^2(\Theta,\pi/2) \subset C^1([\Theta,\pi/2]) \quad\text{for all}\quad \Theta \in \left(0,\frac{\pi}{2}\right).
  \end{align}
  We see that $\widetilde{W}_\infty\in H^1(\Theta,\pi/2)$ by $\tilde{w}_\infty\in H^1(S_+^2)$, \eqref{E:L2_Mode}, and
  \begin{align} \label{Pf_KL:1_sTh}
    \sin\tilde{\theta} \geq \sin\Theta > 0, \quad \tilde{\theta} \in \left(\Theta,\frac{\pi}{2}\right).
  \end{align}
  Moreover, by \eqref{E:L2_Mode}, \eqref{Pf_KL:1_wtk_co}, and \eqref{Pf_KL:1_sTh},
  \begin{align} \label{Pf_KL:1_L2Thst}
    \lim_{k\to\infty}\widetilde{W}_k = \widetilde{W}_\infty \quad\text{strongly in}\quad L^2(\Theta,\pi/2).
  \end{align}
  For each $\Psi\in L^2(\Theta,\pi/2)$, the linear functional
  \begin{align*}
    L(\tilde{w}) = \frac{1}{2\pi}\int_0^{2\pi}\left(\int_\Theta^{\frac{\pi}{2}}\partial_{\tilde{\theta}}\tilde{w}(\tilde{\theta},\varphi)\,\overline{\Psi(\tilde{\theta})}\,d\tilde{\theta}\right)e^{-im_\infty\varphi}\,d\varphi, \quad \tilde{w} \in H^1(S_+^2)
  \end{align*}
  is bounded on $H^1(S_+^2)$. Indeed, by H\"{o}lder's inequality, \eqref{Pf_KL:1_sTh}, and \eqref{E:Re_SC},
  \begin{align*}
    |L(\tilde{w})| \leq \frac{1}{\sqrt{2\pi\sin\Theta}}\|\nabla\tilde{w}\|_{L^2(S_+^2)}\|\Psi\|_{L^2(\Theta,\pi/2)}.
  \end{align*}
  Hence $\lim_{k\to\infty}L(\tilde{w}_k)=L(\tilde{w}_\infty)$ by \eqref{Pf_KL:1_wtk_co}. By this fact and
  \begin{align*}
    L(\tilde{w}_k) = \int_\Theta^{\frac{\pi}{2}}\widetilde{W}'_k(\tilde{\theta})\,\overline{\Psi(\tilde{\theta})}\,d\tilde{\theta} = \Bigl(\widetilde{W}'_k,\Psi\Bigr)_{L^2(\Theta,\pi/2)}, \quad k\in\mathbb{N}\cup\{\infty\},
  \end{align*}
  we find that
  \begin{align} \label{Pf_KL:1_H1Thwe}
    \lim_{k\to\infty}\widetilde{W}'_k = \widetilde{W}'_\infty \quad\text{weakly in}\quad L^2(\Theta,\pi/2).
  \end{align}
  To get $\widetilde{W}''_\infty\in L^2(\Theta,\pi/2)$, we derive an ODE for $\widetilde{W}_k$, $k\in\mathbb{N}$. By \eqref{Pf_KL:Def_Wtk} we have $\widetilde{W}_k^{(j)}(\tilde{\theta})=\gamma_k^jW_k^{(j)}(\gamma_k\tilde{\theta})$ for $j=0,1,2$, where $\widetilde{W}_k^{(j)}$ and $W_k^{(j)}$ are the $j$-th derivatives of $\widetilde{W}_k$ and $W_k$. We use this relation and \eqref{Pf_KL:1_sTh} and calculate as in \eqref{Pf_KL:1_Int_L2} to get
  \begin{align*}
    \Bigl\|\widetilde{W}_k^{(j)}\Bigr\|_{L^2(\Theta,\pi/2)}^2 &\leq \frac{\gamma_k^{2j}}{\sin\Theta}\int_\Theta^{\frac{\pi}{2}}|W_k^{(j)}(\gamma_k\tilde{\theta})|^2\sin\tilde{\theta}\,d\tilde{\theta} \leq \frac{C\gamma_k^{2j-2}}{\sin\Theta}\|\nabla^jw_k\|_{L^2(S^2)}^2.
  \end{align*}
  Thus $\widetilde{W}_k\in H^2(\Theta,\pi/2)$ by $w_k\in H^2(S^2)$. Also, by \eqref{E:KL_w_Eq},
  \begin{align*}
    \Delta w_k+6w_k= 6\mu_k\frac{w_k+\delta f_k}{\mu_k-x_3} \quad\text{on}\quad \{(x_1,x_2,x_3) \in S^2 \mid x_3\neq\mu_k\}.
  \end{align*}
  Using $x_3=\cos\theta$, \eqref{Pf_KL:Wk_Fk} with $m_k=m_\infty$, and \eqref{E:Re_SC}, we rewrite this equation as
  \begin{align*}
    W''_k(\theta)+\frac{\cos\theta}{\sin\theta}W'_k(\theta)+\left(6-\frac{m_\infty^2}{\sin^2\theta}\right)W_k(\theta) = 6\mu_k\frac{W_k(\theta)+\delta_kF_k(\theta)}{\mu_k-\cos\theta}, \quad \theta\in(0,\pi)\setminus\{\theta_k\}.
  \end{align*}
  We multiply both sides by $\gamma_k^2$, set $\theta=\gamma_k\tilde{\theta}$, and use \eqref{Pf_KL:Def_Wtk} to get
  \begin{align} \label{Pf_KL:1_Wtk_Eq}
    \widetilde{W}''_k(\tilde{\theta}) = \zeta_{1,k}(\tilde{\theta})\widetilde{W}'_k(\tilde{\theta})+\zeta_{2,k}(\tilde{\theta})\widetilde{W}_k(\tilde{\theta})+\zeta_{3,k}(\tilde{\theta})\frac{\widetilde{W}_k(\tilde{\theta})+\delta_k\widetilde{F}_k(\tilde{\theta})}{\tilde{\theta}-\tilde{\theta}_k}
  \end{align}
  for $\tilde{\theta}\in(0,\pi/2)\setminus\{\tilde{\theta}_k\}$, where
  \begin{align*}
    \zeta_{1,k}(\tilde{\theta}) = -\frac{\gamma_k\cos(\gamma_k\tilde{\theta})}{\sin(\gamma_k\tilde{\theta})}, \quad \zeta_{2,k}(\tilde{\theta}) = \frac{m_\infty^2\gamma_k^2}{\sin^2(\gamma_k\tilde{\theta}_k)}-6\gamma_k^2, \quad \zeta_{3,k}(\tilde{\theta}) = \frac{6\mu_k\gamma_k^2(\tilde{\theta}-\tilde{\theta}_k)}{\mu_k-\cos(\gamma_k\tilde{\theta})}.
  \end{align*}
  Since $\lim_{k\to\infty}\gamma_k=0$ and
  \begin{align*}
    |\cos(\gamma_k\tilde{\theta})-1| \leq \frac{1}{2}\gamma_k^2\tilde{\theta}^2, \quad \left|\frac{\gamma_k}{\sin(\gamma_k\tilde{\theta})}-\frac{1}{\tilde{\theta}}\right| = \frac{|\sin(\gamma_k\tilde{\theta})-\gamma_k\tilde{\theta}|}{\tilde{\theta}\sin(\gamma_k\tilde{\theta})} \leq \frac{\pi}{12}\gamma_k^2\tilde{\theta}
  \end{align*}
  for $\tilde{\theta}\in(0,\pi/2)$ by Taylor's theorem and $\sin\theta\geq2\theta/\pi$ for $\theta\in(0,\pi/2)$,
  \begin{align} \label{Pf_KL:1_ze12}
    \lim_{k\to\infty}\zeta_{j,k}(\tilde{\theta}) = \zeta_{j,\infty}(\tilde{\theta}) =
    \begin{cases}
      -\tilde{\theta}^{-1}, &j=1, \\
      m_\infty^2\tilde{\theta}^{-2}, &j=2,
    \end{cases}
    \quad\text{uniformly on}\quad \left(\Theta,\frac{\pi}{2}\right).
  \end{align}
  Also, since $\mu_k=\cos\theta_k=\cos(\gamma_k\tilde{\theta}_k)$ and
  \begin{align*}
    \cos(\gamma_k\tilde{\theta}) = 1+\sum_{N=1}^\infty\frac{(-1)^N}{(2N)!}\gamma_k^{2N}\tilde{\theta}^{2N}, \quad \tilde{\theta} \in \left(0,\frac{\pi}{2}\right)
  \end{align*}
  by the Taylor series for $\cos\theta$ at $\theta=0$, it follows that
  \begin{align} \label{Pf_KL:1_cosGT}
    \cos(\gamma_k\tilde{\theta})-\mu_k = \sum_{N=1}^\infty\frac{(-1)^N}{(2N)!}\gamma_k^{2N}(\tilde{\theta}^{2N}-\tilde{\theta}_k^{2N}) = -\frac{1}{2}\gamma_k^2(\tilde{\theta}^2-\tilde{\theta}_k^2)\{1+2\gamma_k^2\rho_k(\tilde{\theta})\}
  \end{align}
  for $\tilde{\theta}\in(0,\pi/2)$. Here
  \begin{align*}
    \rho_k(\tilde{\theta}) = -\sum_{N=2}^\infty\frac{(-1)^N}{(2N)!}\gamma_k^{2N-4}\left(\sum_{M=0}^{N-1}\tilde{\theta}^{2M}\tilde{\theta}_k^{2(N-1-M)}\right) = \sum_{M=0}^\infty C_{M,k}\tilde{\theta}^{2M},
  \end{align*}
  where we rearranged the double summation and defined
  \begin{align*}
    C_{M,k} = -\tilde{\theta}_k^{2(1-M)}\sum_{N=\max\{2,M+1\}}^\infty\frac{(-1)^N}{(2N)!}(\gamma_k\tilde{\theta}_k)^{2(N-2)}, \quad M \geq0.
  \end{align*}
  Since $\gamma_k\tilde{\theta}_k=\theta_k=\arccos\mu_k\to0$ as $k\to\infty$ by $\mu_\infty=1$, we may assume that $\gamma_k\tilde{\theta}_k\in(0,1/2)$ for sufficiently large $k\in\mathbb{N}$. By this fact and $\tilde{\theta}_k\in(\pi/6,\pi/3)$, we have
  \begin{align*}
    |C_{M,k}| \leq \frac{\tilde{\theta}_k^{2(1-M)}}{(2M)!}\sum_{N=0}^\infty 2^{-2N} = \frac{4}{3}\cdot\frac{\tilde{\theta}_k^{2(1-M)}}{(2M)!} \leq \frac{4}{3}\left(\frac{\pi}{3}\right)^2\cdot\frac{1}{(2M)!}\left(\frac{\pi}{6}\right)^{-2M}, \quad M\geq0.
  \end{align*}
  Hence we easily find that $\rho_k(\tilde{\theta})$ converges absolutely for all $\tilde{\theta}\in(0,\pi/2)$ and that $\rho_k$ and its derivative $\rho'_k$ are bounded on $(0,\pi/2)$ uniformly in $k\in\mathbb{N}$. Then since
  \begin{align*}
    \zeta_{3,k}(\tilde{\theta}) = \frac{6\mu_k\gamma_k^2(\tilde{\theta}-\tilde{\theta}_k)}{\mu_k-\cos(\gamma_k\tilde{\theta})} = \frac{12\mu_k(\tilde{\theta}-\tilde{\theta}_k)}{(\tilde{\theta}^2-\tilde{\theta}_k^2)\{1+2\gamma_k^2\rho_k(\tilde{\theta})\}} = \frac{12\mu_k}{(\tilde{\theta}+\tilde{\theta}_k)\{1+2\gamma_k^2\rho_k(\tilde{\theta})\}}
  \end{align*}
  by \eqref{Pf_KL:1_cosGT}, there exists a constant $C>0$ independent of $k\in\mathbb{N}$ such that
  \begin{gather}
    |\zeta_{3,k}(\tilde{\theta})|+|\zeta'_{3,k}(\tilde{\theta})| \leq C, \quad \tilde{\theta} \in \left(0,\frac{\pi}{2}\right), \label{Pf_KL:1_ze3_bd} \\
    \lim_{k\to\infty}\zeta_{3,k}(\tilde{\theta}) = \zeta_{3,\infty}(\tilde{\theta}) = \frac{12}{\tilde{\theta}+\tilde{\theta}_\infty} \quad\text{uniformly on}\quad \left(0,\frac{\pi}{2}\right) \label{Pf_KL:1_ze3_co}
  \end{gather}
  by \eqref{Pf_KL:1_GtTht}, $\mu_\infty=1$, and the uniform boundedness of $\rho_k$ and $\rho'_k$.

  Now let $\Psi\in C_c^\infty(\Theta,\pi/2)$. We take the $L^2(\Theta,\pi/2)$-inner product of \eqref{Pf_KL:1_Wtk_Eq} with $\Psi$ and carry out integration by parts for the left-hand side to get
  \begin{multline} \label{Pf_KL:1_Wk_weak}
    -\Bigl(\widetilde{W}'_k,\Psi'\Bigr)_{L^2(\Theta,\pi/2)} = \Bigl(\widetilde{W}'_k,\zeta_{1,k}\Psi\Bigr)_{L^2(\Theta,\pi/2)}+\Bigl(\widetilde{W}_k,\zeta_{2,k}\Psi\Bigr)_{L^2(\Theta,\pi/2)} \\
    +\left(\frac{\widetilde{W}_k+\delta_k\widetilde{F}_k}{\tilde{\theta}-\tilde{\theta}_k},\zeta_{3,k}\Psi\right)_{L^2(\Theta,\pi/2)}.
  \end{multline}
  We send $k\to\infty$, use \eqref{Pf_KL:1_L2Thst}, \eqref{Pf_KL:1_H1Thwe}, and \eqref{Pf_KL:1_ze12}, and then apply H\"{o}lder's inequality and $|\zeta_{1,\infty}(\tilde{\theta})|\leq\Theta^{-1}$ and $|\zeta_{2,\infty}(\tilde{\theta})|\leq m_\infty^2\Theta^{-2}$ for $\tilde{\theta}\in(\Theta,\pi/2)$ to get
  \begin{multline} \label{Pf_KL:1_Win_wbd}
    \left|\Bigl(\widetilde{W}'_\infty,\Psi'\Bigr)_{L^2(\Theta,\pi/2)}\right| \leq \left(\frac{1}{\Theta}\Bigl\|\widetilde{W}'_\infty\Bigr\|_{L^2(\Theta,\pi/2)}+\frac{m_\infty^2}{\Theta^2}\Bigl\|\widetilde{W}_\infty\Bigr\|_{L^2(\Theta,\pi/2)}\right)\|\Psi\|_{L^2(\Theta,\pi/2)} \\
    +\limsup_{k\to\infty}\left|\left(\frac{\widetilde{W}_k+\delta_k\widetilde{F}_k}{\tilde{\theta}-\tilde{\theta}_k},\zeta_{3,k}\Psi\right)_{L^2(\Theta,\pi/2)}\right|.
  \end{multline}
  Let us estimate the last term as in the proof of Lemma \ref{L:KL_vH1}. We define
  \begin{align*}
    \Psi_k(\tilde{\theta}) = \frac{\zeta_{3,k}(\tilde{\theta})\Psi(\tilde{\theta})}{\sin\tilde{\theta}}, \quad \tilde{\theta}\in\left(\Theta,\frac{\pi}{2}\right)
  \end{align*}
  and extend it to $(0,\pi/2)$ by zero outside $(\Theta,\pi/2)$. Then
  \begin{align} \label{Pf_KL:1_Psik}
    \|\Psi_k\|_{W^{1,\infty}(0,\pi/2)} \leq \frac{C}{\sin\Theta}\|\Psi\|_{W^{1,\infty}(\Theta,\pi/2)}, \quad \|\Psi_k\|_{L^2(0,\pi/2)} \leq \frac{C}{\sin\Theta}\|\Psi\|_{L^2(\Theta,\pi/2)}
  \end{align}
  by \eqref{Pf_KL:1_sTh} and \eqref{Pf_KL:1_ze3_bd} with a constant $C>0$ independent of $k\in\mathbb{N}$. Let
  \begin{align*}
    \varepsilon_k = \frac{\sigma^2\kappa\delta_k^2}{\sin\theta_k}, \quad \varepsilon'_k = 2\sigma^2\sin\theta_k, \quad \tilde{\varepsilon}_k = \frac{\varepsilon_k}{\gamma_k} = \frac{\sigma^2\kappa\delta_k^2}{\gamma_k\sin\theta_k}, \quad \tilde{\varepsilon}'_k = \frac{\varepsilon'_k}{\gamma_k} = \frac{2\sigma^2\sin\theta_k}{\gamma_k}
  \end{align*}
  for $\sigma\in(0,1/2)$. Then
  \begin{align*}
    0 < \varepsilon_k < \varepsilon_k' < \frac{1}{2}\sin\theta_k < \frac{1}{2}\theta_k, \quad \gamma_k\sin\theta_k \leq \frac{8}{\pi}(1-\mu_k), \quad \frac{\pi}{4\sqrt{2}} \leq \frac{\sin\theta_k}{\gamma_k} \leq \frac{\pi}{4}
  \end{align*}
  by $0\leq\mu_k\leq1$, $\kappa\delta_k^2\leq1-\mu_k\leq1-\mu_k^2=\sin^2\theta_k$, and \eqref{Pf_KL:1_Def_ga}. Hence
  \begin{align} \label{Pf_KL:1_epst}
    \frac{\pi\kappa}{8}\cdot\frac{\sigma^2\delta_k^2}{1-\mu_k} \leq \tilde{\varepsilon}_k \leq \frac{\pi\kappa}{4}\left(\frac{\sigma\delta_k}{\sin\theta_k}\right)^2, \quad \frac{\pi}{2\sqrt{2}}\sigma^2 \leq \tilde{\varepsilon}'_k \leq \frac{\pi}{2}\sigma^2
  \end{align}
  and $0<\tilde{\varepsilon}_k<\tilde{\varepsilon}'_k<\tilde{\theta}_k/2$ for sufficiently small $\sigma>0$ since $\tilde{\theta}_k\in(\pi/6,\pi/3)$. Let
  \begin{align*}
    \tilde{I}_{1,k} &= (\tilde{\theta}_k-\tilde{\varepsilon}_k,\tilde{\theta}_k+\tilde{\varepsilon}_k), \\
    \tilde{I}_{2,k} &= (\tilde{\theta}_k-\tilde{\varepsilon}'_k,\tilde{\theta}_k-\tilde{\varepsilon}_k]\cup[\tilde{\theta}_k+\tilde{\varepsilon}_k,\tilde{\theta}_k+\tilde{\varepsilon}'_k), \\
    \tilde{I}_{3,k} &= (0,\tilde{\theta}_k-\tilde{\varepsilon}'_k]\cup\left[\tilde{\theta}_k+\tilde{\varepsilon}'_k,\frac{\pi}{2}\right).
  \end{align*}
  Then
  \begin{align} \label{Pf_KL:1_sum_Jt}
    \left(\frac{\widetilde{W}_k+\delta_k\widetilde{F}_k}{\tilde{\theta}-\tilde{\theta}_k},\zeta_{3,k}\Psi\right)_{L^2(\Theta,\pi/2)} = \int_0^{\frac{\pi}{2}}\frac{\widetilde{W}_k(\tilde{\theta})+\delta_k\widetilde{F}_k(\tilde{\theta})}{\tilde{\theta}-\tilde{\theta}_k}\,\overline{\Psi_k(\tilde{\theta})}\sin\tilde{\theta}\,d\tilde{\theta} = \sum_{j=1}^3\tilde{J}_{j,k},
  \end{align}
  where
  \begin{align*}
    \tilde{J}_{j,k} = \int_{\tilde{I}_{j,k}}\frac{\widetilde{W}_k(\tilde{\theta})+\delta_k\widetilde{F}_k(\tilde{\theta})}{\tilde{\theta}-\tilde{\theta}_k}\,\overline{\Psi_k(\tilde{\theta})}\sin\tilde{\theta}\,d\tilde{\theta}, \quad j=1,2,3.
  \end{align*}
  Let us estimate $\tilde{J}_{1,k}$. For $\tilde{\theta}\in(0,\pi/2)$ let $\widetilde{G}_k(\tilde{\theta})=\widetilde{W}_k(\tilde{\theta})+\delta_k\widetilde{F}_k(\tilde{\theta})$ and
  \begin{align*}
     \widetilde{K}_{1,k} = \int_{\tilde{I}_{1,k}}\left|\frac{\widetilde{G}_k(\tilde{\theta})\sqrt{\sin\tilde{\theta}}}{\tilde{\theta}-\tilde{\theta}_k}\right|^2\,d\tilde{\theta}.
  \end{align*}
  Since $\widetilde{G}_k(\tilde{\theta}_k)=0$ by \eqref{Pf_KL:1_tilWF}, we can use Hardy's inequality to get
  \begin{align} \label{Pf_KL:1_Kt1k_01}
    \begin{aligned}
      \widetilde{K}_{1,k} &\leq 4\int_{\tilde{I}_{1,k}}\left|\frac{d}{d\tilde{\theta}}\biggl(\widetilde{G}_k(\tilde{\theta})\sqrt{\sin\tilde{\theta}}\biggr)\right|^2\,d\tilde{\theta} \\
      &\leq C\int_{\tilde{\theta}_k-\tilde{\varepsilon}_k}^{\tilde{\theta}_k+\tilde{\varepsilon}_k}\left(\Bigl|\widetilde{G}'_k(\tilde{\theta})\Bigr|^2+\frac{m_\infty^2}{\sin^2\tilde{\theta}}\Bigl|\widetilde{G}_k(\tilde{\theta})\Bigr|^2\right)\sin\tilde{\theta}\,d\tilde{\theta}.
    \end{aligned}
  \end{align}
  Noting that $\widetilde{G}_k(\tilde{\theta})=W_k(\gamma_k\tilde{\theta})+\delta_kF_k(\gamma_k\tilde{\theta})$, $\theta_k=\gamma_k\tilde{\theta}_k$, and $\varepsilon_k=\gamma_k\tilde{\varepsilon}_k$, we set $\theta=\gamma_k\tilde{\theta}$ in the second line of \eqref{Pf_KL:1_Kt1k_01} and calculate as in \eqref{Pf_KL:1_Int_H1} by using \eqref{E:L2_Mode} and \eqref{Pf_KL:1_sin} to get
  \begin{align} \label{Pf_KL:1_Kt1k_02}
    \begin{aligned}
      \widetilde{K}_{1,k} \leq C\|\nabla(w_k+\delta_kf_k)\|_{L^2(S_k^2)}^2 \leq C\Bigl(\|\nabla w_k\|_{L^2(S^2)}+\delta_k\|\nabla f_k\|_{L^2(S_k^2)}\Bigr),
    \end{aligned}
  \end{align}
  where $S_k^2=S^2(\theta_k-\varepsilon_k,\theta_k+\varepsilon_k)$ (see \eqref{E:Def_Band}). Moreover, noting that
  \begin{align*}
    \|M_{\sin\theta}u_k\|_{L^2(S_k^2)} \leq C\sin\theta_k\|u_k\|_{L^2(S_k^2)} \leq C\sin\theta_k\|u_k\|_{L^2(S^2)}
  \end{align*}
  by \eqref{E:KL_theta}, we apply \eqref{E:KL_w_est} and \eqref{E:KL_gradf} to \eqref{Pf_KL:1_Kt1k_02} to deduce that
  \begin{align*}
    \frac{\delta_k}{\sin\theta_k}\widetilde{K}_{1,k}^{1/2} &\leq C\left(\delta_k\|u_k\|_{L^2(S^2)}+\delta_k\varepsilon_k\|(-\Delta)^{1/2}u_k\|_{L^2(S^2)}+\frac{\delta_k}{\sin\theta_k}\|(-\Delta)^{1/2}v_k\|_{L^2(S^2)}\right).
  \end{align*}
  Also, since $\kappa^{1/2}\delta_k\leq(1-\mu_k)^{1/2}\leq\sin\theta_k$,
  \begin{align*}
    \frac{\delta_k}{\sin\theta_k} \leq \kappa^{-1/2}, \quad \delta_k\varepsilon_k = \frac{\sigma^2\kappa\delta_k^3}{\sin\theta_k} \leq \frac{C\delta_k^3}{(1-\mu_k)^{1/2}}.
  \end{align*}
  Hence $\delta_k\widetilde{K}_{1,k}^{1/2}/\sin\theta_k\leq C(1+R_k)$ by the above inequalities, \eqref{E:KL_u} for $u_k$, and \eqref{Pf_KL:As_u} (recall that $R_k=R(u_k,\delta_k,\mu_k)$ is given by \eqref{E:KL_fR}). By this inequality, $|\tilde{I}_{1,k}|=2\tilde{\varepsilon}_k$, \eqref{Pf_KL:1_Psik}, and \eqref{Pf_KL:1_epst}, we find that
  \begin{align} \label{Pf_KL:1_Jt1k}
    \begin{aligned}
      |\tilde{J}_{1,k}| &\leq \widetilde{K}_{1,k}^{1/2}\left(\int_{\tilde{I}_{1,k}}|\Psi_k(\tilde{\theta})|^2\sin\tilde{\theta}\,d\tilde{\theta}\right)^{1/2} \leq \sqrt{2\tilde{\varepsilon}_k}\,\widetilde{K}_{1,k}^{1/2}\|\Psi_k\|_{L^\infty(0,\pi/2)} \\
      &\leq \frac{C\sigma\delta_k}{\sin\theta_k}\widetilde{K}_{1,k}^{1/2}\|\Psi_k\|_{L^\infty(0,\pi/2)} \leq \frac{C\sigma}{\sin\Theta}(1+R_k)\|\Psi\|_{W^{1,\infty}(\Theta,\pi/2)}.
    \end{aligned}
  \end{align}
  Next we estimate $\tilde{J}_{2,k}$. We split $\tilde{J}_{2,k}=\sum_{l=1}^4\tilde{J}_{2,k}^l$ into
  \begin{align*}
    \tilde{J}_{2,k}^1 &= \int_{\tilde{I}_{2,k}}\frac{\widetilde{W}_k(\tilde{\theta})\sqrt{\sin\tilde{\theta}}-\widetilde{W}_k(\tilde{\theta})\sqrt{\sin\tilde{\theta}_k}}{\tilde{\theta}-\tilde{\theta}_k}\,\overline{\Psi_k(\tilde{\theta})}\sqrt{\sin\tilde{\theta}}\,d\tilde{\theta}, \\
    \tilde{J}_{2,k}^2 &= \widetilde{W}_k(\tilde{\theta}_k)\sqrt{\sin\tilde{\theta}_k}\int_{\tilde{I}_{2,k}}\frac{\overline{\Psi_k(\tilde{\theta})}\sqrt{\sin\tilde{\theta}}-\overline{\Psi_k(\tilde{\theta}_k)}\sqrt{\sin\tilde{\theta}_k}}{\tilde{\theta}-\tilde{\theta_k}}\,d\tilde{\theta}, \\
    \tilde{J}_{2,k}^3 &= \widetilde{W}_k(\tilde{\theta}_k)\,\overline{\Psi_k(\tilde{\theta}_k)}\sin\theta_k\int_{\tilde{I}_{2,k}}\frac{1}{\tilde{\theta}-\tilde{\theta}_k}\,d\tilde{\theta}, \\
    \tilde{J}_{2,k}^4 &= \int_{\tilde{I}_{2,k}}\frac{\delta_k\widetilde{F}_k(\tilde{\theta})}{\tilde{\theta}-\tilde{\theta}_k}\,\overline{\Psi_k(\tilde{\theta})}\sin\tilde{\theta}\,d\tilde{\theta}.
  \end{align*}
  Also, let
  \begin{align*}
    \widetilde{K}_k &=\int_0^{\frac{\pi}{2}}\left|\frac{\widetilde{W}_k(\tilde{\theta})\sqrt{\sin\tilde{\theta}}-\widetilde{W}_k(\tilde{\theta})\sqrt{\sin\tilde{\theta}_k}}{\tilde{\theta}-\tilde{\theta}_k}\right|^2\,d\tilde{\theta}, \\
    \tilde{L}_k &= \int_0^{\frac{\pi}{2}}\left|\frac{\Psi_k(\tilde{\theta})\sqrt{\sin\tilde{\theta}}-\Psi_k(\tilde{\theta})\sqrt{\sin\tilde{\theta}_k}}{\tilde{\theta}-\tilde{\theta}_k}\right|^2\,d\tilde{\theta}.
  \end{align*}
  As in \eqref{Pf_KL:1_Kt1k_01} and \eqref{Pf_KL:1_Kt1k_02}, we apply Hardy's inequality to $\widetilde{K}_k$, set $\theta=\gamma_k\tilde{\theta}$, and use \eqref{E:L2_Mode} and \eqref{Pf_KL:1_sin}. Then we further use \eqref{E:KL_w_est} and \eqref{Pf_KL:As_u} to get
  \begin{align} \label{Pf_KL:1_Ktk_bd}
    \widetilde{K}_k \leq C\|\nabla w_k\|_{L^2(S^2)}^2 \leq C\|(-\Delta)^{1/2}v_k\|_{L^2(S^2)}^2 \leq C.
  \end{align}
  Also, by Hardy's inequality, $\mathrm{supp}\,\Psi_k\subset(\Theta,\pi/2)$, and \eqref{Pf_KL:1_Psik},
  \begin{align} \label{Pf_KL:1_Ltk_bd}
    \tilde{L}_k \leq C\|\Psi_k\|_{W^{1,\infty}(0,\pi/2)}^2\int_\Theta^{\frac{\pi}{2}}\left(\sin\tilde{\theta}+\frac{1}{\sin\tilde{\theta}}\right)\,d\tilde{\theta} \leq \frac{C}{\sin^3\Theta}\|\Psi\|_{W^{1,\infty}(\Theta,\pi/2)}^2
  \end{align}
  We use H\"{o}lder's inequality, $|\tilde{I}_{2,k}|\leq2\tilde{\varepsilon}'_k$, \eqref{Pf_KL:1_Psik}, \eqref{Pf_KL:1_epst}, and \eqref{Pf_KL:1_Ktk_bd} to get
  \begin{align} \label{Pf_KL:1_Jt2k1}
    |\tilde{J}_{2,k}^1| \leq \widetilde{K}_k^{1/2}|\tilde{I}_{2,k}|^{1/2}\|\Psi_k\|_{L^\infty(\Theta,\pi/2)} \leq \frac{C\sigma}{\sin\Theta}\|\Psi\|_{W^{1,\infty}(\Theta,\pi/2)}.
  \end{align}
  Also, by H\"{o}lder's inequality, $|\tilde{I}_{2,k}|\leq2\tilde{\varepsilon}'_k\leq C$, $\sin\tilde{\theta}_k\leq1$, and \eqref{Pf_KL:1_Ltk_bd},
  \begin{align} \label{Pf_KL:1_Jt2k2}
    |\tilde{J}_{2,k}^2| \leq \Bigl|\widetilde{W}_k(\tilde{\theta}_k)\Bigr|\sqrt{\sin\tilde{\theta}_k}\,\tilde{L}_k^{1/2}|\tilde{I}_{2,k}|^{1/2} \leq \frac{C}{\sin^{3/2}\Theta}\Bigl|\widetilde{W}_k(\tilde{\theta}_k)\Bigr|\|\Psi\|_{W^{1,\infty}(\Theta,\pi/2)}.
  \end{align}
  We have $\tilde{J}_{2,k}^3=0$ since
  \begin{align*}
    \int_{\tilde{I}_{2,k}}\frac{1}{\tilde{\theta}-\tilde{\theta}_k}\,d\tilde{\theta} &= \int_{\tilde{\theta}_k-\tilde{\varepsilon}'_k}^{\tilde{\theta}_k-\tilde{\varepsilon}_k}\frac{d}{d\tilde{\theta}}\bigl(\log(\tilde{\theta}_k-\tilde{\theta})\bigr)\,d\tilde{\theta}+\int_{\tilde{\theta}_k+\tilde{\varepsilon}_k}^{\tilde{\theta}_k+\tilde{\varepsilon}'_k}\frac{d}{d\tilde{\theta}}\bigl(\log(\tilde{\theta}-\tilde{\theta}_k)\bigr)\,d\tilde{\theta} = 0.
  \end{align*}
  For $\tilde{J}_{2,k}^4$, we see that
  \begin{align*}
    |\tilde{J}_{2,k}^4| \leq \delta_k\|\Psi_k\|_{L^\infty(0,\pi/2)}\left(\int_0^{\frac{\pi}{2}}\Bigl|\widetilde{F}_k(\tilde{\theta})\Bigr|^2\sin\tilde{\theta}\,d\tilde{\theta}\right)^{1/2}\left(\int_{\tilde{I}_{2,k}}\frac{1}{|\tilde{\theta}-\tilde{\theta}_k|^2}\,d\tilde{\theta}\right)^{1/2}.
  \end{align*}
  As in \eqref{Pf_KL:1_Int_L2}, we use $\widetilde{F}_k(\tilde{\theta})=F_k(\gamma_k\tilde{\theta})$, \eqref{E:L2_Mode}, and \eqref{Pf_KL:1_sin} to get
  \begin{align} \label{Pf_KL:1_Ft_L2}
    \int_0^{\frac{\pi}{2}}\Bigl|\widetilde{F}_k(\tilde{\theta})\Bigr|^2\sin\tilde{\theta}\,d\tilde{\theta} \leq \frac{C}{1-\mu_k}\|f_k\|_{L^2(S^2)}^2.
  \end{align}
  Also, we observe by \eqref{Pf_KL:1_epst} that
  \begin{align*}
    \int_{\tilde{I}_{2,k}}\frac{1}{|\tilde{\theta}-\tilde{\theta}_k|^2}\,d\tilde{\theta} = \left(\int_{\tilde{\theta}_k-\tilde{\varepsilon}'_k}^{\tilde{\theta}_k-\tilde{\varepsilon}_k}+\int_{\tilde{\theta}_k+\tilde{\varepsilon}_k}^{\tilde{\theta}_k+\tilde{\varepsilon}'_k}\right)\frac{d}{d\tilde{\theta}}\left(-\frac{1}{\tilde{\theta}-\tilde{\theta}_k}\right)\,d\tilde{\theta} \leq \frac{2}{\tilde{\varepsilon}_k} \leq \frac{C(1-\mu_k)}{\sigma^2\delta_k^2}.
  \end{align*}
  We deduce from the above inequalities and \eqref{Pf_KL:1_Psik} that
  \begin{align*}
    |\tilde{J}_{2,k}^4| \leq C\sigma^{-1}\|f_k\|_{L^2(S^2)}\|\Psi_k\|_{L^\infty(0,\pi/2)} \leq \frac{C\sigma^{-1}R_k}{\sin\Theta}\|\Psi\|_{W^{1,\infty}(\Theta,\pi/2)}.
  \end{align*}
  Thus, by this inequality, $\tilde{J}_{2,k}^3=0$, \eqref{Pf_KL:1_Jt2k1}, \eqref{Pf_KL:1_Jt2k2}, and $\sin\Theta\leq1$,
  \begin{align} \label{Pf_KL:1_Jt2k}
    |\tilde{J}_{2,k}| \leq \sum_{l=1}^4|\tilde{J}_{2,k}^l| \leq \frac{C}{\sin^{3/2}\Theta}\left(\sigma+\Bigl|\widetilde{W}_k(\tilde{\theta}_k)\Bigr|+\sigma^{-1}R_k\right)\|\Psi\|_{W^{1,\infty}(\Theta,\pi/2)}.
  \end{align}
  To estimate $\tilde{J}_{3,k}$, we decompose $\tilde{J}_{3,k}=\sum_{l=1}^3\tilde{J}_{3,k}^l$ into
  \begin{align*}
    \tilde{J}_{3,k}^1 &= \int_{\tilde{I}_{3,k}}\frac{\widetilde{W}_k(\tilde{\theta})\sqrt{\sin\tilde{\theta}}-\widetilde{W}_k(\tilde{\theta})\sqrt{\sin\tilde{\theta}_k}}{\tilde{\theta}-\tilde{\theta}_k}\,\overline{\Psi_k(\tilde{\theta})}\sqrt{\sin\tilde{\theta}}\,d\tilde{\theta}, \\
    \tilde{J}_{3,k}^2 &= \widetilde{W}_k(\tilde{\theta}_k)\sqrt{\sin\tilde{\theta}_k}\int_{\tilde{I}_{3,k}}\frac{\overline{\Psi_k(\tilde{\theta})}\sqrt{\sin\tilde{\theta}}}{\tilde{\theta}-\tilde{\theta}_k}\,d\tilde{\theta}, \\
    \tilde{J}_{3,k}^3 &= \int_{\tilde{I}_{3,k}}\frac{\delta_k\widetilde{F}_k(\tilde{\theta})}{\tilde{\theta}-\tilde{\theta}_k}\,\overline{\Psi_k(\tilde{\theta})}\sin\tilde{\theta}\,d\tilde{\theta}.
  \end{align*}
  We apply H\"{o}lder's inequality to $\tilde{J}_{3,k}^1$ and use \eqref{Pf_KL:1_Psik} and \eqref{Pf_KL:1_Ktk_bd} to get
  \begin{align*}
    |\tilde{J}_{3,k}^1| &\leq \widetilde{K}_k^{1/2}\|\Psi_k\|_{L^2(0,\pi/2)} \leq \frac{C}{\sin\Theta}\|\Psi\|_{L^2(\Theta,\pi/2)}.
  \end{align*}
  When $\tilde{\theta}\in\tilde{I}_{3,k}$, we have $|\tilde{\theta}-\tilde{\theta}|\geq\tilde{\varepsilon}'_k\geq\pi\sigma^2/2\sqrt{2}$ by \eqref{Pf_KL:1_epst}. Thus
  \begin{align*}
    |\tilde{J}_{3,k}^2| &\leq \Bigl|\widetilde{W}_k(\tilde{\theta}_k)\Bigr|\sqrt{\sin\tilde{\theta}_k}\,\frac{2\sqrt{2}}{\pi\sigma^2}\int_{\tilde{I}_{3,k}}|\Psi_k(\tilde{\theta})|\sqrt{\sin\tilde{\theta}}\,d\tilde{\theta} \\
    &\leq C\sigma^{-2}\Bigl|\widetilde{W}_k(\tilde{\theta}_k)\Bigr|\|\Psi_k\|_{L^2(0,\pi/2)} \leq \frac{C\sigma^{-2}}{\sin\Theta}\Bigl|\widetilde{W}_k(\tilde{\theta}_k)\Bigr|\|\Psi\|_{L^2(\Theta,\pi/2)}.
  \end{align*}
  by $|\tilde{I}_{3,k}|\leq C$, $\sin\tilde{\theta}_k\leq1$, and \eqref{Pf_KL:1_Psik}. Similarly,
  \begin{align*}
    |\tilde{J}_{3,k}^3| &\leq \frac{2\sqrt{2}\,\delta_k}{\pi\sigma^2}\int_{\tilde{I}_{3,k}}\Bigl|\widetilde{F}_k(\tilde{\theta})\Bigr||\Psi_k(\tilde{\theta})|\sin\tilde{\theta}\,d\tilde{\theta} \leq \frac{C\sigma^{-2}R_k}{\sin\Theta}\|\Psi\|_{L^2(\Theta,\pi/2)}
  \end{align*}
  by H\"{o}lder's inequality, \eqref{Pf_KL:1_Psik}, \eqref{Pf_KL:1_Ft_L2}, and $1-\mu_k\geq\kappa\delta_k^2$. Hence
  \begin{align} \label{Pf_KL:1_Jt3k}
    |\tilde{J}_{3,k}| \leq \sum_{l=1}^3|\tilde{J}_{3,k}^l| \leq \frac{C}{\sin\Theta}\left(1+\sigma^{-2}\Bigl|\widetilde{W}_k(\tilde{\theta}_k)\Bigr|+\sigma^{-2}R_k\right)\|\Psi\|_{L^2(\Theta,\pi/2)}.
  \end{align}
  Now we deduce from \eqref{Pf_KL:1_sum_Jt}, \eqref{Pf_KL:1_Jt1k}, \eqref{Pf_KL:1_Jt2k}, and \eqref{Pf_KL:1_Jt3k} that
  \begin{multline*}
    \left|\left(\frac{\widetilde{W}_k+\delta_k\widetilde{F}_k}{\tilde{\theta}-\tilde{\theta}_k},\zeta_{3,k}\Psi\right)_{L^2(\Theta,\pi/2)}\right| \leq \frac{C}{\sin\Theta}\left(1+\sigma^{-2}\Bigl|\widetilde{W}_k(\tilde{\theta}_k)\Bigr|+\sigma^{-2}R_k\right)\|\Psi\|_{L^2(\Theta,\pi/2)} \\
    +\frac{C}{\sin^{3/2}\Theta}\left(\sigma+\Bigl|\widetilde{W}_k(\tilde{\theta}_k)\Bigr|+\sigma^{-1}R_k\right)\|\Psi\|_{W^{1,\infty}(\Theta,\pi/2)}
  \end{multline*}
  for all $\sigma\in(0,1/2)$, where we also used $\sin\Theta\leq1$ and $\sigma<1$. Then we send $k\to\infty$, use \eqref{Pf_KL:As_Rk} and \eqref{Pf_KL:1_tilWF}, and let $\sigma\to0$ to find that
  \begin{align*}
    \limsup_{k\to\infty}\left|\left(\frac{\widetilde{W}_k+\delta_k\widetilde{F}_k}{\tilde{\theta}-\tilde{\theta}_k},\zeta_{3,k}\Psi\right)_{L^2(\Theta,\pi/2)}\right| \leq \frac{C}{\sin\Theta}\|\Psi\|_{L^2(\Theta,\pi/2)}.
  \end{align*}
  Therefore, applying this inequality to \eqref{Pf_KL:1_Win_wbd}, we obtain
  \begin{align*}
    \left|\Bigl(\widetilde{W}'_\infty,\Psi'\Bigr)_{L^2(\Theta,\pi/2)}\right| \leq \left(\frac{1}{\Theta}\Bigl\|\widetilde{W}'_\infty\Bigr\|_{L^2(\Theta,\pi/2)}+\frac{m_\infty^2}{\Theta^2}\Bigl\|\widetilde{W}_\infty\Bigr\|_{L^2(\Theta,\pi/2)}+\frac{C}{\sin\Theta}\right)\|\Psi\|_{L^2(\Theta,\pi/2)}
  \end{align*}
  for all $\Psi\in C_c^\infty(\Theta,\pi/2)$, which gives \eqref{Pf_KL:1_Wtin_H2} since $C_c^\infty(\Theta,\pi/2)$ is dense in $L^2(\Theta,\pi/2)$.

  Next we derive an ODE for $\widetilde{W}_\infty$. For $\Theta\in(0,\tilde{\theta}_\infty)$, let $\Psi\in C_c^\infty(\Theta,\pi/2)$ satisfy $\Psi=0$ near $\tilde{\theta}_\infty$. Then we observe by \eqref{Pf_KL:1_Wk_weak} that
  \begin{multline*}
    -\Bigl(\widetilde{W}'_k,\Psi'\Bigr)_{L^2(\Theta,\pi/2)} = \Bigl(\widetilde{W}'_k,\zeta_{1,k}\Psi\Bigr)_{L^2(\Theta,\pi/2)}+\Bigl(\widetilde{W}_k,\zeta_{2,k}\Psi\Bigr)_{L^2(\Theta,\pi/2)} \\
    +\left(\widetilde{W}_k+\delta_k\widetilde{F}_k,\frac{\zeta_{3,k}\Psi}{\tilde{\theta}-\tilde{\theta}_k}\right)_{L^2(\Theta,\pi/2)}.
  \end{multline*}
  We send $k\to\infty$ in this equality. Then since
  \begin{align*}
    \Bigl\|\delta_k\widetilde{F}_k\Bigr\|_{L^2(\Theta,\pi/2)}^2 \leq \frac{\delta_k^2}{\sin\Theta}\int_\Theta^{\frac{\pi}{2}}\Bigl|\widetilde{F}_k(\tilde{\theta})\Bigr|^2\sin\tilde{\theta}\,d\tilde{\theta} \leq \frac{C}{\sin\Theta}\cdot\frac{\delta_k^2}{1-\mu_k}\|f_k\|_{L^2(S^2)}^2 \leq \frac{CR_k^2}{\sin\Theta} \to 0
  \end{align*}
  as $k\to\infty$ by \eqref{Pf_KL:As_Rk}, \eqref{Pf_KL:1_sTh}, \eqref{Pf_KL:1_Ft_L2}, and $1-\mu_k\geq\kappa\delta_k^2$, and since
  \begin{align*}
    \frac{\zeta_{3,k}\Psi}{\tilde{\theta}-\tilde{\theta}_k} \to \frac{\zeta_{3,\infty}\Psi}{\tilde{\theta}-\tilde{\theta}_\infty} \quad\text{uniformly on}\quad \left(\Theta,\frac{\pi}{2}\right)
  \end{align*}
  by \eqref{Pf_KL:1_GtTht}, \eqref{Pf_KL:1_ze3_co}, and $\Psi=0$ near $\tilde{\theta}_\infty$, we find by \eqref{Pf_KL:1_L2Thst}, \eqref{Pf_KL:1_H1Thwe}, and \eqref{Pf_KL:1_ze12} that
  \begin{multline*}
    -\Bigl(\widetilde{W}'_\infty,\Psi'\Bigr)_{L^2(\Theta,\pi/2)} = \Bigl(\widetilde{W}'_\infty,\zeta_{1,\infty}\Psi\Bigr)_{L^2(\Theta,\pi/2)}+\Bigl(\widetilde{W}_\infty,\zeta_{2,\infty}\Psi\Bigr)_{L^2(\Theta,\pi/2)} \\
    +\left(\widetilde{W}_\infty,\frac{\zeta_{3,\infty}\Psi}{\tilde{\theta}-\tilde{\theta}_\infty}\right)_{L^2(\Theta,\pi/2)}.
  \end{multline*}
  Since this equality is valid for all $\Theta\in(0,\tilde{\theta}_\infty)$ and $\Psi\in C_c^\infty(\Theta,\pi/2)$ vanishing near $\tilde{\theta}_\infty$, and since $\widetilde{W}_\infty\in H_{loc}^2(0,\pi/2)$ by \eqref{Pf_KL:1_Wtin_H2}, we obtain
  \begin{align} \label{Pf_KL:1_Wtin_Eq}
    \widetilde{W}''_\infty(\tilde{\theta}) = \zeta_{1,\infty}(\tilde{\theta})\widetilde{W}'_\infty(\tilde{\theta})+\left\{\zeta_{2,\infty}(\tilde{\theta})+\frac{\zeta_{3,\infty}(\tilde{\theta})}{\tilde{\theta}-\tilde{\theta}_\infty}\right\}\widetilde{W}_\infty(\tilde{\theta}), \quad \tilde{\theta}\in\left(0,\frac{\pi}{2}\right)\setminus\{\tilde{\theta}_\infty\}.
  \end{align}
  Now we observe that $\sin\tilde{\theta}\geq\sin\tilde{\theta}_k\geq1/2$ for $\tilde{\theta}\in(\tilde{\theta}_k,\pi/2)$ since $\tilde{\theta}_k\in(\pi/6,\pi/3)$. We apply this inequality and calculate as in \eqref{Pf_KL:1_Int_L2} by using \eqref{Pf_KL:1_sin} and $\theta_k=\gamma_k\tilde{\theta}_k$. Then
  \begin{align*}
    \Bigl\|\widetilde{W}_k\Bigr\|_{L^2(\tilde{\theta}_k,\pi/2)}^2 \leq 2\int_{\tilde{\theta}_k}^{\frac{\pi}{2}}\Bigl|\widetilde{W}_k(\tilde{\theta})\Bigr|^2\sin\tilde{\theta}\,d\tilde{\theta} \leq \frac{C}{\gamma_k^2}\int_{\theta_k}^\pi|W_k(\theta)|^2\sin\theta\,d\theta.
  \end{align*}
  Moreover, we apply \eqref{E:L2_Mode} and \eqref{Pf_KL:1_Def_ga} to the last term and use \eqref{Pf_KL:1_wLS} to get
  \begin{align*}
    \Bigl\|\widetilde{W}_k\Bigr\|_{L^2(\tilde{\theta}_k,\pi/2)}^2 \leq \frac{C}{1-\mu_k}\|w_k\|_{L^2(S^2(\theta_k,\pi))}^2 \to 0 \quad\text{as}\quad k\to\infty.
  \end{align*}
  By this fact, \eqref{Pf_KL:1_GtTht}, and \eqref{Pf_KL:1_L2Thst} with $\Theta=\pi/6$, we find that
  \begin{align*}
    \Bigl\|\widetilde{W}_\infty\Bigr\|_{L^2(\tilde{\theta}_\infty,\pi/2)} = \lim_{k\to\infty}\Bigl\|\widetilde{W}_k\Bigr\|_{L^2(\tilde{\theta}_k,\pi/2)} = 0.
  \end{align*}
  Thus $\widetilde{W}_\infty=0$ on $(\tilde{\theta}_\infty,\pi/2)$ and, since $\widetilde{W}_\infty\in C^1(0,\pi/2)$ by \eqref{Pf_KL:1_Wtin_H2}, we further get
  \begin{align} \label{Pf_KL:1_Wtin_zero}
    \widetilde{W}_\infty(\tilde{\theta}_\infty) = \widetilde{W}'_\infty(\tilde{\theta}_\infty) = 0.
  \end{align}
  Then we easily find that any solution of \eqref{Pf_KL:1_Wtin_Eq}--\eqref{Pf_KL:1_Wtin_zero} in $C^1(0,\pi/2)$ must be trivial by the standard theory of ODEs, since $\zeta_{1,\infty}$, $\zeta_{2,\infty}$, and $\zeta_{3,\infty}$ given by \eqref{Pf_KL:1_ze12} and \eqref{Pf_KL:1_ze3_co} are smooth on $(0,\pi/2)$ and the singularity of $(\tilde{\theta}-\tilde{\theta}_\infty)^{-1}$ at $\tilde{\theta}=\tilde{\theta}_\infty$ is of order one. Hence $\widetilde{W}_\infty=0$, i.e. $\tilde{w}_\infty=0$, which contradicts $\tilde{w}_\infty\neq0$. This completes the proof in Case 2, and we conclude that \eqref{E:Ko_Lowu} is valid.
\end{proof}

\begin{proof}[Proof of \eqref{E:Ko_LowS}]
  We assume $\mu\geq0$ for simplicity, since the case $\mu\leq0$ can be handled similarly. Let $\kappa$ be the constant given by \eqref{Pf_KM:kappa} and $\delta\in(0,1]$, $\mu\in[0,1-\kappa\delta^2]$, and $\theta_\mu=\arccos\mu\in(0,\pi/2]$. Also, for $u\in\mathcal{Y}_m\cap H^1(S^2)$ with $m\in\mathbb{Z}\setminus\{0\}$, let $v$, $w$, and $f$ be given by \eqref{E:KL_vw} and \eqref{E:KL_fR}. We set $\varepsilon=\kappa^{1/2}\delta/2$, which satisfies \eqref{E:KL_eps} by \eqref{E:KL_sinmu}. Since $u,v,w,f\in\mathcal{X}_m$, we can write
  \begin{align*}
    u = U(\theta)e^{im\varphi}, \quad v = V(\theta)e^{im\varphi}, \quad w = W(\theta)e^{im\varphi}, \quad f = F(\theta)e^{im\varphi}.
  \end{align*}
  For $\theta\in(\theta_\mu-\varepsilon,\theta_\mu+\varepsilon)$, we observe by \eqref{E:IF_sinU} and \eqref{E:KL_theta} that
  \begin{align*}
    |\sin\theta\,U(\theta)|^2 \leq C\sin^2\theta_\mu|U(\theta)|^2 \leq C\|M_{\sin\theta}u\|_{L^2(S^2)}\|(-\Delta)^{1/2}u\|_{L^2(S^2)}.
  \end{align*}
  By this inequality, \eqref{E:KL_cosep}, and Young's inequality,
  \begin{align} \label{Pf_KLS:In}
    \begin{aligned}
      \int_{\theta_\mu-\varepsilon}^{\theta_\mu+\varepsilon}|\sin\theta\,U(\theta)|^2\sin\theta\,d\theta &\leq C\varepsilon\sin\theta_\mu\|M_{\sin\theta}u\|_{L^2(S^2)}\|(-\Delta)^{1/2}u\|_{L^2(S^2)} \\
      &\leq \frac{1}{2}\|M_{\sin\theta}u\|_{L^2(S^2)}^2+C\varepsilon^2\sin^2\theta_\mu\|(-\Delta)^{1/2}u\|_{L^2(S^2)}^2.
    \end{aligned}
  \end{align}
  Since $u+6v=B_{2,m}u=(w+\delta f)/(\mu-x_3)$ for $x_3\neq\mu$ by \eqref{E:KL_B2u_Eq},
  \begin{align*}
    U(\theta) = -6V(\theta)+\frac{W(\theta)}{\mu-\cos\theta}+\frac{\delta F(\theta)}{\mu-\cos\theta}, \quad \theta\in(0,\pi)\setminus\{\theta_\mu\}.
  \end{align*}
  By this equality and $|\mu-\cos\theta|\geq\frac{1}{2}|\theta-\theta_\mu|\sin\theta$ (see \eqref{E:KL_mucos}),
  \begin{align*}
    |U(\theta)| \leq C\left(|V(\theta)|+\left|\frac{W(\theta)}{\theta-\theta_\mu}\right|+\frac{\delta}{|\theta-\theta_\mu|}|F(\theta)|\right), \quad \theta\in(0,\pi)\setminus\{\theta_\mu\}.
  \end{align*}
  Moreover, if $|\theta-\theta_\mu|\geq\varepsilon$, then $\delta/|\theta-\theta_\mu|\leq\delta/\varepsilon=2/\kappa^{1/2}$. By these inequalities, $|\sin\theta|\leq1$, and \eqref{E:L2_Mode}, we find that
  \begin{multline} \label{Pf_KLS:Out}
    \int_{|\theta-\theta_\mu|\geq\varepsilon}|\sin\theta\,U(\theta)|^2\sin\theta\,d\theta \\
    \leq C\left(\|v\|_{L^2(S^2)}^2+\|f\|_{L^2(S^2)}^2+\int_{|\theta-\theta_\mu|\geq\varepsilon}\left|\frac{W(\theta)\sqrt{\sin\theta}}{\theta-\theta_\mu}\right|^2\,d\theta\right).
  \end{multline}
  To estimate the last term, we see that
  \begin{align*}
    \int_{|\theta-\theta_\mu|\geq\varepsilon}\left|\frac{W(\theta)\sqrt{\sin\theta}}{\theta-\theta_\mu}\right|^2\,d\theta \leq C\left(J+|W(\theta_\mu)|^2\sin\theta_\mu\int_{|\theta-\theta_\mu|\geq\varepsilon}\frac{1}{|\theta-\theta_\mu|^2}\,d\theta\right),
  \end{align*}
  where
  \begin{align*}
    J = \int_{|\theta-\theta_\mu|\geq\varepsilon}\left|\frac{W(\theta)\sqrt{\sin\theta}-W(\theta_\mu)\sqrt{\sin\theta_\mu}}{\theta-\theta_\mu}\right|^2\,d\theta.
  \end{align*}
  We use Hardy's inequality, \eqref{E:L2_Mode}, and \eqref{E:KL_w_est} to get
  \begin{align*}
    J &\leq 4\int_0^\pi\left|\frac{d}{d\theta}\Bigl(W(\theta)\sqrt{\sin\theta}\Bigr)\right|^2\,d\theta \leq C\int_0^\pi\left(|W'(\theta)|^2+\frac{m^2}{\sin^2\theta}|W(\theta)|^2\right)\sin\theta\,d\theta \\
    &\leq C\|\nabla w\|_{L^2(S^2)}^2 \leq C\|(-\Delta)^{1/2}v\|_{L^2(S^2)}^2.
  \end{align*}
  Also, $\int_{|\theta-\theta_\mu|\geq\varepsilon}|\theta-\theta_\mu|^{-2}\,d\theta\leq2/\varepsilon$ by $|\theta-\theta_\mu|^{-2}=-\frac{d}{d\theta}(\theta-\theta_\mu)^{-1}$. Hence
  \begin{align*}
    \int_{|\theta-\theta_\mu|\geq\varepsilon}\left|\frac{W(\theta)\sqrt{\sin\theta}}{\theta-\theta_\mu}\right|^2\,d\theta \leq C\left(\|(-\Delta)^{1/2}v\|_{L^2(S^2)}^2+\frac{\sin\theta_\mu}{\varepsilon}|W(\theta_\mu)|^2\right).
  \end{align*}
  We combine this inequality, \eqref{Pf_KLS:In}, and \eqref{Pf_KLS:Out} and use \eqref{E:H1_Equiv} to $v$ to get
  \begin{multline} \label{Pf_KLS:Sinu}
    \|M_{\sin\theta}u\|_{L^2(S^2)}^2 \leq C\varepsilon^2\sin^2\theta_\mu\|(-\Delta)^{1/2}u\|_{L^2(S^2)}^2 \\
    +C\left(\|(-\Delta)^{1/2}v\|_{L^2(S^2)}^2+\|f\|_{L^2(S^2)}^2+\frac{\sin\theta_\mu}{\varepsilon}|W(\theta_\mu)|^2\right).
  \end{multline}
  Let us estimate the last term. By Lemma \ref{L:KL_WF} we have $W(\theta_\mu)+\delta F(\theta_\mu)=0$. Fix $\sigma\in(0,1)$. Since $\sigma\varepsilon$ satisfies \eqref{E:KL_eps}, there exists $\theta'_\mu\in(\theta_\mu-\sigma\varepsilon,\theta_\mu+\sigma\varepsilon)$ such that
  \begin{align*}
    |F(\theta'_\mu)|^2 &\leq \frac{1}{\sigma\varepsilon\sin\theta_\mu}\int_{\theta_\mu-\sigma\varepsilon}^{\theta_\mu+\sigma\varepsilon}|F(\theta)|^2\sin\theta\,d\theta \leq \frac{C}{\sigma\varepsilon\sin\theta_\mu}\|f\|_{L^2(S^2)}^2, \\
    |F(\theta_\mu)-F(\theta'_\mu)|^2 &\leq \frac{C\sigma\varepsilon}{\sin\theta_\mu}\int_{\theta_\mu-\sigma\varepsilon}^{\theta_\mu+\sigma\varepsilon}|F'(\theta)|^2\sin\theta\,d\theta \leq \frac{C\sigma\varepsilon}{\sin\theta_\mu}\|\nabla f\|_{L^2(S_{\mu,\sigma\varepsilon}^2)}^2
  \end{align*}
  as in \eqref{Pf_WF:prime} and \eqref{Pf_WF:diff}, where $S_{\mu,\sigma\varepsilon}^2$ is given by \eqref{E:KL_Sme}. Thus
  \begin{align*}
    |W(\theta_\mu)|^2 = \delta^2|F(\theta_\mu)|^2 &\leq \frac{C\sigma\varepsilon}{\sin\theta_\mu}\left(\sigma^{-2}\|f\|_{L^2(S^2)}^2+\delta^2\|\nabla f\|_{L^2(S_{\mu,\sigma\varepsilon}^2)}^2\right),
  \end{align*}
  where we also used $\delta/\varepsilon=2/\kappa^{1/2}$. We apply this inequality to \eqref{Pf_KLS:Sinu} and then use \eqref{E:KL_gradf}, $\sigma<1$, and $\|M_{\sin\theta}u\|_{L^2(S_{\mu,\sigma\varepsilon}^2)}\leq\|M_{\sin\theta}u\|_{L^2(S^2)}$ to get
  \begin{multline*}
    (1-C\sigma)\|M_{\sin\theta}u\|_{L^2(S^2)}^2 \\
    \leq C\Bigl(\varepsilon^2\sin^2\theta_\mu\|(-\Delta)^{1/2}u\|_{L^2(S^2)}^2+\|(-\Delta)^{1/2}v\|_{L^2(S^2)}^2+\sigma^{-2}\|f\|_{L^2(S^2)}^2\Bigr).
  \end{multline*}
  In this inequality, we fix $\sigma\in(0,1)$ so that $C\sigma\leq1/2$, apply \eqref{E:Ko_Lowu} to $(-\Delta)^{1/2}v$, and use $\varepsilon=\kappa^{1/2}\delta/2$ and \eqref{E:KL_sinmu} to obtain \eqref{E:Ko_LowS}.
\end{proof}

Now we are ready to verify Assumption \ref{As:Est} for $A_m$ and $\Lambda_m$. Let $\kappa$ be the constant given by \eqref{Pf_KM:kappa}. For $m\in\mathbb{Z}\setminus\{0\}$, $\xi\in(0,\infty)$, and $\mu\in\mathbb{R}$, we define
\begin{align*}
  h_{1,m}(\xi,\mu) &=
  \begin{cases}
    0 &\text{if}\quad |\mu| > 1+\kappa\xi^{-1}, \\
    |m|^{-1/2}\xi^{-1/2} &\text{if}\quad 1-\kappa\xi^{-1} < |\mu| \leq 1+\kappa\xi^{-1}, \\
    \xi^{-1}(1-|\mu|)^{-1/2} &\text{if}\quad |\mu| \leq 1-\kappa\xi^{-1},
  \end{cases} \\
  h_{2,m}(\xi,\mu) &=
  \begin{cases}
    0 &\text{if}\quad |\mu| > 1+\kappa\xi^{-2}, \\
    |m|^{-1/2}\xi^{-2} &\text{if}\quad 1-\kappa\xi^{-2} < |\mu| \leq 1+\kappa\xi^{-2}, \\
    \xi^{-1}(1-|\mu|)^{1/2} &\text{if}\quad |\mu| \leq 1-\kappa\xi^{-2}.
  \end{cases}
\end{align*}
Note that $h_{j,m}(\xi,\mu)\geq0$ and $\lim_{\xi\to\infty}\sup_{\mu\in\mathbb{R}}h_j(\xi,\mu)=0$ for $j=1,2$.

\begin{lemma} \label{L:Kom_As4}
  Let $m\in\mathbb{Z}\setminus\{0\}$. Then $A_m$ and $\Lambda_m$ satisfy Assumption \ref{As:Est} with the functions $h_{1,m}$ and $h_{2,m}$ given above and the operator $B_{3,m}=M_{\sin\theta}$ on $\mathcal{B}_m=L^2(S^2)$.
\end{lemma}

\begin{proof}
  Since $B_{3,m}$ is a bounded operator on $\mathcal{B}_m$, it is closed in $\mathcal{B}_m$ and
  \begin{align*}
    D_{\mathcal{X}_m}((-A_m)^{1/2}) \subset \mathcal{X}_m \subset L^2(S^2) = \mathcal{B}_m = D_{\mathcal{B}_m}(B_{3,m}).
  \end{align*}
  Also, \eqref{E:ALam_B3} holds by \eqref{E:Ko_B3}. Let us verify \eqref{E:u_h} and \eqref{E:B3u_h}.
  If $\xi\in(0,\infty)$ and $|\mu|>1+\kappa\xi^{-1}$, then \eqref{E:u_h} is valid by \eqref{E:Ko_Hiu} and $(|\mu|-1)^{-1}<\kappa^{-1}\xi$. Also, if $\xi\in[1,\infty)$ and $|\mu|\leq 1+\kappa\xi^{-1}$, then \eqref{E:u_h} follows from \eqref{E:Ko_Amh}, \eqref{E:Ko_Midu}, and \eqref{E:Ko_Lowu} with $\delta=\xi^{-1/2}$. Let $\xi\in(0,1)$ and $|\mu|\leq1+\kappa\xi^{-1}$. For $u\in\mathcal{Y}_m\cap H^1(S^2)$ of the form \eqref{E:Ko_Ym}, we have
  \begin{align*}
    \|u\|_{L^2(S^2)}^2 \leq \frac{1}{|m|^2}\|(-\Delta)^{1/2}u\|_{L^2(S^2)}^2 \leq \frac{C}{|m|^2}\|(-A_m)^{1/2}u\|_{L^2(S^2)}^2
  \end{align*}
  by \eqref{E:Def_Laps}, $\lambda_n=n(n+1)\geq|m|^2$ for $n\geq|m|$, and \eqref{E:Ko_Amh}. Moreover,
  \begin{align*}
    |m|^{-2} \leq
    \begin{cases}
      |m|^{-1}\xi^{-1} &\text{if}\quad 1-\kappa\xi^{-1} < |\mu| \leq 1+\kappa\xi^{-1}, \\
      \xi^{-2}(1-|\mu|)^{-1} &\text{if}\quad |\mu| \leq 1-\kappa\xi^{-1}
    \end{cases}
  \end{align*}
  since $|m|\geq1$, $\xi\in(0,1)$, and $(1-|\mu|)^{-1}\geq1$. Hence \eqref{E:u_h} is valid.

  The inequality \eqref{E:B3u_h} follows from \eqref{E:Ko_HiS} and $(|\mu|-1)^{-1}<\kappa^{-1}\xi^{-2}$ when $\xi\in(0,\infty)$ and $|\mu|>1+\kappa\xi^{-2}$ and from \eqref{E:Ko_Amh}, \eqref{E:Ko_MidS}, and \eqref{E:Ko_LowS} with $\delta=\xi^{-1}$ when $\xi\in[1,\infty)$ and $|\mu|\leq1+\kappa\xi^{-2}$. Also, when $\xi\in(0,1)$ and $|\mu|\leq1+\kappa\xi^{-2}$, we have \eqref{E:B3u_h} by
  \begin{align*}
    \|B_{3,m}u\|_{L^2(S^2)}^2 = \|M_{\sin\theta}u\|_{L^2(S^2)}^2 \leq \|u\|_{L^2(S^2)}^2 \leq \frac{C}{|m|^2}\|(-A_m)^{1/2}u\|_{L^2(S^2)}^2
  \end{align*}
  for $u\in\mathcal{Y}_m\cap H^1(S^2)$ and
  \begin{align*}
    |m|^{-2} \leq
    \begin{cases}
      |m|^{-1}\xi^{-4} &\text{if}\quad 1-\kappa\xi^{-2} < |\mu| \leq 1+\kappa\xi^{-2}, \\
      \kappa^{-1}\xi^{-2}(1-|\mu|) &\text{if}\quad |\mu| \leq 1-\kappa\xi^{-2}
    \end{cases}
  \end{align*}
  due to $|m|\geq1$, $\xi\in(0,1)$, and $\kappa^{-1}\xi^2(1-|\mu|)\geq1$ when $|\mu|\leq1-\kappa\xi^{-2}$.
\end{proof}

\subsection{Estimates for the semigroup} \label{SS:Ko_Semi}
By Lemmas \ref{L:Kom_As13} and \ref{L:Kom_As4}, $A_m$ and $\Lambda_m$ satisfy Assumptions \ref{As:A}--\ref{As:La_02} and \ref{As:Est}. Moreover, $\|B_{2,m}\|_{\mathcal{X}_m\to\mathcal{X}_m}\leq1$ by \eqref{E:B2m_Xmk} and the constants in \eqref{E:uB2u}, \eqref{E:AB2u}, and \eqref{E:u_h}--\eqref{E:B3u_h} can be taken independently of $m$. Hence we can apply the abstract results in Section \ref{S:Abst} to $L_{\alpha,m}=A_m-i\alpha m\Lambda_m$ to get the following results.

\begin{theorem} \label{T:Kom_PSB}
  Let $m\in\mathbb{Z}\setminus\{0\}$. Then
  \begin{align} \label{E:Kom_PS_Set}
    \{\zeta\in\mathbb{C} \mid \mathrm{Re}\,\zeta\geq0\} \subset \rho_{\mathcal{X}_m}(L_{\alpha,m}) \subset \rho_{\mathcal{Y}_m}(\mathbb{Q}_mL_{\alpha,m})
  \end{align}
  for all $\alpha\in\mathbb{R}$. Moreover, for all $\alpha\in\mathbb{R}\setminus\{0\}$ and $\lambda\in\mathbb{R}$ we have
  \begin{align} \label{E:Kom_PS_Bo}
    \|(i\lambda+\mathbb{Q}_mL_{m,\alpha})^{-1}\|_{\mathcal{Y}_m\to\mathcal{Y}_m} \leq CG_m\left(\alpha,\frac{\lambda}{\alpha m}\right).
  \end{align}
  Here $C>0$ is a constant independent of $m$, $\alpha$, and $\lambda$. Also,
  \begin{align} \label{E:Kom_PS_Gm}
    G_m(\alpha,\mu) =
    \begin{cases}
      |\alpha m|^{-1}(|\mu|-1)^{-1} &\text{if} \quad |\mu| > 1+|\alpha|^{-1/2}, \\
      |\alpha|^{-1/2}|m|^{-1} &\text{if} \quad 1-|\alpha|^{-1/2} < |\mu| \leq 1+|\alpha|^{-1/2}, \\
      |\alpha m|^{-2/3}(1-|\mu|)^{-1/3} &\text{if} \quad |\mu| \leq 1-|\alpha|^{-1/2}
    \end{cases}
  \end{align}
  for $\alpha\in\mathbb{R}\setminus\{0\}$ and $\mu\in\mathbb{R}$.
\end{theorem}

\begin{proof}
  By Lemma \ref{L:PS_Pre} and Theorem \ref{T:PS_Bound} with $\alpha$ replaced by $\alpha m$ we have \eqref{E:Kom_PS_Set} for all $\alpha\in\mathbb{R}$ and
  \begin{align} \label{Pf_KmP:Bound}
    \|(i\lambda+\mathbb{Q}_mL_{\alpha,m})^{-1}\|_{\mathcal{Y}_m\to\mathcal{Y}_m} \leq CF_m\left(\alpha m,\frac{\lambda}{\alpha m}\right)
  \end{align}
  for all $\alpha\in\mathbb{R}\setminus\{0\}$ and $\lambda\in\mathbb{R}$. Here $C>0$ is a constant independent of $m$, $\alpha$, and $\lambda$. Also,
  \begin{align} \label{Pf_KmP:Fm}
    F_m(\alpha m,\mu) = \inf_{\xi_1,\xi_2>0}\left(\frac{\xi_1}{|\alpha m|}+\frac{\xi_1^2\xi_2^2}{(\alpha m)^2}+\frac{\xi_1^2h_2(\xi_2,\mu)}{|\alpha m|}+h_1(\xi_1,\mu)^2\right)
  \end{align}
  for $\mu=\lambda/\alpha m\in\mathbb{R}$. Let us estimate $F_m(\alpha m,\mu)$. In what follows, we write $C$ for general positive constant depending only on the constant $\kappa$ given by \eqref{Pf_KM:kappa}.

  When $|\mu|>1+|\alpha|^{-1/2}$, let $\xi_1=\xi_2^2=2\kappa(|\mu|-1)^{-1}$. Then $1+\kappa\xi_j^{-j}=(|\mu|+1)/2<|\mu|$ and thus $h_j(\xi_j,\mu)=0$ for $j=1,2$. Also,
  \begin{align*}
    \frac{\xi_1}{|\alpha m|} = \frac{2\kappa}{|\alpha m|(|\mu|-1)}, \quad \frac{\xi_1^2\xi_2^2}{(\alpha m)^2} = \frac{8\kappa^3}{(\alpha m)^2(|\mu|-1)^3} \leq \frac{8\kappa^3}{|m|}\cdot\frac{1}{|\alpha m|(|\mu|-1)}.
  \end{align*}
  By these relations and $|m|\geq1$, we find that
  \begin{align} \label{Pf_KmP:mu_hi}
    F_m(\alpha m,\mu) \leq \frac{C}{|\alpha m|(|\mu|-1)}.
  \end{align}
  When $1-|\alpha|^{-1/2}<|\mu|\leq 1+|\alpha|^{-1/2}$, we set $\xi_1=\xi_2^2=\kappa|\alpha|^{1/2}/2$. Then
  \begin{align*}
    \kappa\xi_j^{-j} = 2|\alpha|^{-1/2} > |\alpha|^{-1/2} \geq \bigl|\,1-|\mu|\,\bigr|, \quad\text{i.e.}\quad 1-\kappa\xi_j^{-j}<|\mu|<1+\kappa\xi_j^{-j}
  \end{align*}
  for $j=1,2$. Hence $h_1(\xi_1,\mu)=|m|^{-1/2}\xi_1^{-1/2}$, $h_2(\xi_2,\mu)=|m|^{-1/2}\xi_2^{-2}$, and
  \begin{gather*}
    \frac{\xi_1}{|\alpha m|} = \frac{\kappa}{2|\alpha|^{1/2}|m|}, \quad \frac{\xi_1^2\xi_2^2}{(\alpha m)^2} = \frac{\kappa^3}{8|\alpha|^{1/2}m^2}, \\
    \frac{\xi_1^2h_2(\xi_2,\mu)}{|\alpha m|} = \frac{\kappa}{2|\alpha|^{1/2}|m|^{3/2}}, \quad h_1(\xi_1,\mu)^2 = \frac{2}{\kappa|\alpha|^{1/2}|m|}.
  \end{gather*}
  By these equalities and $|m|\geq1$, we find that
  \begin{align} \label{Pf_KmP:mu_md}
    F_m(\alpha m,\mu) \leq \frac{C}{|\alpha|^{1/2}|m|}.
  \end{align}
  When $|\mu|\leq1-|\alpha|^{-1/2}$, we take
  \begin{align*}
    \xi_1 = |\alpha m|^{1/3}(1-|\mu|)^{-1/3}, \quad \xi_2 = |\alpha m|^{1/3}(1-|\mu|)^{1/6}.
  \end{align*}
  Then for $j=1,2$ we see by $(1-|\mu|)^{-1/3}\leq|\alpha|^{1/6}$, $\kappa\in(0,1)$, and $|m|\geq1$ that
  \begin{align*}
    \frac{\kappa}{\xi_j^j(1-|\mu|)} = \frac{\kappa}{|\alpha m|^{j/3}(1-|\mu|)^{2j/3}} \leq \frac{\kappa}{|m|^{j/3}} < 1, \quad\text{i.e.}\quad |\mu| < 1-\kappa\xi_j^{-j}.
  \end{align*}
  Hence $h_1(\xi_1,\mu)=\xi_1^{-1}(1-|\mu|)^{-1/2}$, $h_2(\xi_2,\mu)=\xi_2^{-1}(1-|\mu|)^{1/2}$, and
  \begin{align*}
    \frac{\xi_1}{|\alpha m|} = \frac{\xi_1^2\xi_2^2}{(\alpha m)^2} = \frac{\xi_1^2h_2(\xi_2,\mu)}{|\alpha m|} = h_1(\xi_1,\mu)^2 = \frac{1}{|\alpha m|^{2/3}(1-|\mu|)^{1/3}}.
  \end{align*}
  Therefore,
  \begin{align} \label{Pf_KmP:mu_lo}
    F_m(\alpha m,\mu) \leq \frac{4}{|\alpha m|^{2/3}(1-|\mu|)^{1/3}}.
  \end{align}
  Hence we get \eqref{E:Kom_PS_Bo}--\eqref{E:Kom_PS_Gm} by \eqref{Pf_KmP:Bound} and \eqref{Pf_KmP:mu_hi}--\eqref{Pf_KmP:mu_lo}.
\end{proof}

\begin{theorem} \label{T:Kom_Dec_Q}
  Let $m\in\mathbb{Z}\setminus\{0\}$. Then $L_{\alpha,m}$ generates an analytic semigroup $\{e^{tL_{\alpha,m}}\}_{t\geq0}$ in $\mathcal{X}_m$ for all $\alpha\in\mathbb{R}$. Moreover,
  \begin{align} \label{E:Kom_Dec_Q}
    \|\mathbb{Q}_me^{tL_{\alpha,m}}f\|_{L^2(S^2)} \leq C_1e^{-C_2|\alpha|^{1/2}|m|^{2/3}t}\|\mathbb{Q}_mf\|_{L^2(S^2)}, \quad t\geq0, \, f\in\mathcal{X}_m
  \end{align}
  for all $\alpha\in\mathbb{R}\setminus\{0\}$, where $C_1$ and $C_2$ are positive constants independent of $m$, $\alpha$, and $t$.
\end{theorem}

\begin{proof}
  Theorem \ref{T:Semi_BoAl} implies that $L_{\alpha,m}$ generates an analytic semigroup $\{e^{tL_{\alpha,m}}\}_{t\geq0}$ in $\mathcal{X}_m$ for all $\alpha\in\mathbb{R}$. Moreover, for all $\alpha\in\mathbb{R}\setminus\{0\}$ we have
  \begin{align} \label{Pf_KmDQ:QeL}
    \|\mathbb{Q}_me^{tL_{\alpha,m}}f\|_{L^2(S^2)} \leq C_1e^{-C_2t/F_m(\alpha m)}\|\mathbb{Q}_mf\|_{L^2(S^2)}, \quad t\geq0, \, f\in\mathcal{X}_m
  \end{align}
  by \eqref{E:Decay_QL} with $\alpha$ replaced by $\alpha m$. Here $C_1,C_2>0$ are constants independent of $m$, $\alpha$, and $t$. Also, $F_m(\alpha m)=\sup_{\lambda\in\mathbb{R}}F_m\left(\alpha m,\mu\right)$ with $F_m(\alpha m,\mu)$ given by \eqref{Pf_KmP:Fm}.

  Let us estimate $F_m(\alpha m)$. In what follows, $C$ denotes a general positive constant independent of $m$, $\alpha$, and $\mu$. If $|\mu|>1+|\alpha|^{-1/2}$, then $F_m(\alpha m,\mu)\leq C|\alpha|^{-1/2}|m|^{-1}$ by \eqref{Pf_KmP:mu_hi} and $|\mu|-1>|\alpha|^{-1/2}$. The same inequality holds when $1-|\alpha|^{-1/2}<|\mu|\leq1+|\alpha|^{-1/2}$ by \eqref{Pf_KmP:mu_md}. When $|\mu|\leq1-|\alpha|^{-1/2}$, we have $F_m(\alpha m,\mu)\leq 4|\alpha|^{-1/2}|m|^{-2/3}$ by \eqref{Pf_KmP:mu_lo} and $1-|\mu|\geq|\alpha|^{-1/2}$. By these facts and $|m|\geq1$, we get $F_m(\alpha m)\leq C|\alpha|^{-1/2}|m|^{-2/3}$. Hence we apply this inequality to \eqref{Pf_KmDQ:QeL} to obtain \eqref{E:Kom_Dec_Q}.
\end{proof}

When $|m|\geq3$, $\mathbb{Q}_m=I$ on $\mathcal{X}_m$ by \eqref{E:Ko_Qm} and thus \eqref{E:Kom_Dec_Q} gives a decay estimate for $e^{tL_{\alpha,m}}$. On the other hand, when $|m|=1,2$, $\mathbb{P}_mf=(I-\mathbb{Q}_m)f=(f,Y_2^m)_{L^2(S^2)}Y_2^m$ does not vanish for $f\in\mathcal{X}_m$, so we need to estimate the $\mathbb{P}_m$-part of $e^{tL_{\alpha,m}}$.

\begin{theorem} \label{T:Kom_Dec_P}
  Let $|m|=1,2$. Then for all $t\geq0$, $f\in\mathcal{X}_m$, and $\alpha\in\mathbb{R}$ such that $|\alpha|>4$ is sufficiently large, we have
  \begin{multline} \label{E:Kom_Dec_P}
    \|\mathbb{P}_me^{tL_{\alpha,m}}f\|_{L^2(S^2)} \leq e^{-4t}\|\mathbb{P}_mf\|_{L^2(S^2)} \\
    +C_1|\alpha m|^{1/3}\log\Bigl(C_2|\alpha m|^{2/3}\Bigr)e^{-2t}\|\mathbb{Q}_mf\|_{L^2(S^2)},
  \end{multline}
  where $C_1,C_2>0$ are constants independent of $t$, $f$, $\alpha$, and $m$.
\end{theorem}

\begin{proof}
  Let $\mathcal{N}_m=N_{\mathcal{X}_m}(\Lambda_m)=\{cY_2^m\mid c\in\mathbb{C}\}$ (see \eqref{Pf_K13:Ker}). Then
  \begin{align} \label{Pf_KmDP:PA_Res}
    (\zeta-\mathbb{P}_mA_m)^{-1}\mathbb{P}_mf = (\zeta+4)^{-1}\mathbb{P}_mf, \quad \zeta\in\rho_{\mathcal{N}_m}(\mathbb{P}_mA_m)=\mathbb{C}\setminus\{-4\}
  \end{align}
  for $f\in\mathcal{X}_m$ since $\mathbb{P}_mA_mY_2^m=-4Y_2^m$ by \eqref{E:Ko_AL_Y} and $\lambda_2=6$. Also, by \eqref{E:PS_ReSet},
  \begin{align} \label{Pf_KmDP:L_ReSet}
    \rho_{\mathcal{X}_m}(L_{\alpha,m}) = \rho_{\mathcal{Y}_m}(\mathbb{Q}_mL_{\alpha,m})\cap\rho_{\mathcal{N}_m}(\mathbb{P}_mA_m) = \rho_{\mathcal{Y}_m}(\mathbb{Q}_mL_{\alpha,m})\cup(\mathbb{C}\setminus\{-4\}).
  \end{align}
  Since $A_m$ is self-adjoint and satisfies \eqref{E:Ko_A_Po} in $\mathcal{X}_m$ and $\Lambda_m$ is $A_m$-compact in $\mathcal{X}_m$, we see by a perturbation theory of semigroups (see \cite{EngNag00}) that $e^{tL_{\alpha,m}}$ is expressed by the Dunford integral
 \begin{align*}
    e^{tL_{\alpha,m}}f = \frac{1}{2\pi i}\int_\Gamma e^{t\zeta}(\zeta-L_{\alpha,m})^{-1}f\,d\zeta, \quad t>0, \, f\in\mathcal{X}_m,
  \end{align*}
  where $\Gamma$ is any piecewise smooth curve in $\rho_{\mathcal{X}_m}(L_{\alpha,m})$ going from $\infty e^{-i\delta}$ to $\infty e^{i\delta}$ with any $\delta\in(\pi/2,\pi)$ and located on the right-hand side of $\sigma_{\mathcal{X}_m}(L_{\alpha,m})=\sigma_{\mathcal{Y}_m}(\mathbb{Q}_mL_{\alpha,m})\cup\{-4\}$. We apply $\mathbb{P}_m$ to the above expression to get
  \begin{align} \label{Pf_KmDP:Pef_Dec}
    \mathbb{P}_me^{tL_{\alpha,m}}f = \frac{1}{2\pi i}\int_\Gamma e^{t\zeta}(\zeta-\mathbb{P}_mA_m)^{-1}\mathbb{P}_mf\,d\zeta-\frac{\alpha m}{2\pi}J
  \end{align}
  by \eqref{E:PS_ReOp} with $\alpha$ replaced by $\alpha m$, where
  \begin{align*}
    J = \int_\Gamma e^{t\zeta}(\zeta-\mathbb{P}_mA_m)^{-1}\mathbb{P}_m\Lambda_m(\zeta-\mathbb{Q}_mL_{\alpha,m})^{-1}\mathbb{Q}_mf\,d\zeta.
  \end{align*}
  Noting that $\zeta=-4$ is located on the left-hand side of $\Gamma$, we use \eqref{Pf_KmDP:PA_Res}  and the residue theorem (after replacing $\Gamma$ by a small circle centered at $\zeta=-4$) to deduce that
  \begin{align} \label{Pf_KmDP:Int_PA}
    \frac{1}{2\pi i}\int_\Gamma e^{t\zeta}(\zeta-\mathbb{P}_mA_m)^{-1}\mathbb{P}_mf\,d\zeta = \left(\frac{1}{2\pi i}\int_\Gamma \frac{e^{t\zeta}}{\zeta+4}\,d\zeta\right)\mathbb{P}_mf = e^{-4t}\mathbb{P}_mf.
  \end{align}
  Next we estimate $J$. For $\alpha\in\mathbb{R}\setminus\{0\}$ and $\lambda\in\mathbb{R}$ let
  \begin{align*}
    H_m(\alpha,\lambda) =
    \begin{cases}
      |m|^{1/3}(|\lambda|-|\alpha m|)^{-1} &\text{if}\quad |\lambda/\alpha m| > 1+|\alpha|^{-1/2}, \\
      |\alpha|^{-1/2}|m|^{-2/3} &\text{if}\quad 1-|\alpha|^{-1/2} < |\lambda/\alpha m| \leq 1+|\alpha|^{-1/2}, \\
      |\alpha m|^{-1/3}(|\alpha m|-|\lambda|)^{-1/3} &\text{if}\quad |\lambda/\alpha m| \leq 1-|\alpha|^{-1/2}.
    \end{cases}
  \end{align*}
  Then since $G_m(\alpha,\lambda/\alpha m)\leq H_m(\alpha,\lambda)$, where $G_m$ is the function given by \eqref{E:Kom_PS_Gm}, we see by \eqref{E:Kom_PS_Bo} with $\lambda$ replaced by $-\lambda$ that
  \begin{align} \label{Pf_KmDP:Bd_Im}
    \|(i\lambda-\mathbb{Q}_mL_{\alpha,m})^{-1}\|_{\mathcal{Y}_m\to\mathcal{Y}_m} \leq C_1H_m(\alpha,-\lambda) = C_1H_m(\alpha,\lambda)
  \end{align}
  with a constant $C_1>0$ independent of $m$, $\alpha$, and $\lambda$.
  Also,
  \begin{align} \label{Pf_KmDP:Bd_Res}
    \begin{gathered}
      \Omega = \{\zeta\in\mathbb{C} \mid |\mathrm{Re}\,\zeta| \leq 1/2C_1H_m(\alpha,\mathrm{Im}\,\zeta)\} \subset \rho_{\mathcal{Y}_m}(\mathbb{Q}_mL_{\alpha,m}), \\
      \|(\zeta-\mathbb{Q}_mL_{\alpha,m})^{-1}\|_{\mathcal{Y}_m\to\mathcal{Y}_m} \leq 2C_1H_m(\alpha,\mathrm{Im}\,\zeta), \quad \zeta\in\Omega
    \end{gathered}
  \end{align}
  by \eqref{Pf_KmDP:Bd_Im} and a standard Neumann series argument. Let
  \begin{align*}
    S_1 = \frac{1}{4C_1}\frac{|\alpha m|^{2/3}}{2^{1/3}}, \quad S_2 = \frac{|\alpha|^{1/2}|m|^{2/3}}{4C_1}.
  \end{align*}
  We take $|\alpha|>4$ sufficiently large so that $S_1>S_2>2$ and set
  \begin{align*}
    \zeta_{1,\pm}(s) &= -s\pm i\frac{|\alpha m|}{2}\frac{s-2}{S_1-2}, \quad s\in I_1 = [2,S_1], \\
    \zeta_{2,\pm}(s) &= -s\pm i\left\{|\alpha m|-\frac{1}{|\alpha m|}(4C_1s)^3\right\}, \quad s\in I_2 = [S_2,S_1], \\
    \zeta_{3,\pm}(s) &= -S_2\pm is, \quad s\in I_3 = \left[|\alpha m|-|\alpha|^{1/2}|m|,|\alpha m|+|\alpha|^{1/2}|m|\right], \\
    \zeta_{4,\pm}(s) &= -s\pm i\left(|\alpha m|+4C_1|m|^{1/3}s\right), \quad s\in I_4 = [S_2,\infty).
  \end{align*}
  Then we define a piecewise smooth curve
  \begin{align*}
    \Gamma = \bigcup_{k=1}^4(\Gamma_{k,+}\cup\Gamma_{k,-})\subset\mathbb{C}\setminus\{-4\}, \quad \Gamma_{k,\pm} = \{\zeta_{k,\pm}(s) \mid s\in I_k\}.
  \end{align*}
  For $s\in I_1$ we have $|\mathrm{Im}\,\zeta_{1,\pm}(s)/\alpha m|\leq1/2\leq1-|\alpha|^{-1/2}$ by $|\alpha|>4$.
  Hence
  \begin{align} \label{Pf_KmDP:Hz1}
    H_m(\alpha,\mathrm{Im}\,\zeta_{1,\pm}(s)) = \frac{1}{|\alpha m|^{1/3}(|\alpha m|-|\mathrm{Im}\,\zeta_{1,\pm}(s)|)^{1/3}} \leq \frac{2^{1/3}}{|\alpha m|^{2/3}}
  \end{align}
  and
  \begin{align*}
    \frac{1}{2C_1H_m(\alpha,\mathrm{Im}\,\zeta_{1,\pm}(s))} \geq \frac{1}{2C_1}\frac{|\alpha m|^{2/3}}{2^{1/3}} = 2S_1 > S_1 \geq |\mathrm{Re}\,\zeta_{1,\pm}(s)|,
  \end{align*}
  which shows $\Gamma_{1,\pm}\subset\Omega$. Similarly, for $k=2,3,4$ and $s\in I_k$ we have
  \begin{align} \label{Pf_KmDP:Hz234}
    2C_1H_m(\alpha,\mathrm{Im}\,\zeta_{k,\pm}(s)) =
    \begin{cases}
      (2s)^{-1}, \quad k=2,4, \\
      2C_1|\alpha|^{-1/2}|m|^{-2/3} = (2S_2)^{-1}, \quad k=3
    \end{cases}
  \end{align}
  and thus $\Gamma_{k,\pm}\subset\Omega$. Hence it follows from \eqref{Pf_KmDP:L_ReSet} and \eqref{Pf_KmDP:Bd_Res} that
  \begin{align*}
    \Gamma \subset \Omega\cap(\mathbb{C}\setminus\{-4\}) \subset \rho_{\mathcal{Y}_m}(\mathbb{Q}_mL_{\alpha,m})\cap\rho_{\mathcal{N}_m}(\mathbb{P}_mA_m) = \rho_{\mathcal{X}_m}(L_{\alpha,m})
  \end{align*}
  and we can take this $\Gamma$ as the path of the integral for $J$ to get $J=\sum_{k=1}^4J_k$, where
  \begin{align*}
    J_k = \left(\int_{\Gamma_{k,+}}+\int_{\Gamma_{k,-}}\right)e^{t\zeta}(\zeta-\mathbb{P}_mA_m)^{-1}\mathbb{P}_m\Lambda_m(\zeta-\mathbb{Q}_mL_{\alpha,m})^{-1}\mathbb{Q}_mf\,d\zeta.
  \end{align*}
  Let us estimate each $J_k$. In what follows, we write $C$ and $C'$ for general positive constant depending only on $C_1$. For $s\in I_1$ we have
  \begin{align*}
    |\zeta_{1,\pm}(s)+4| = \sqrt{(4-s)^2+b^2(s-2)^2} = \sqrt{(1+b^2)\{(s-c_1)^2+c_2^2\}}
  \end{align*}
  and $|e^{t\zeta_{1,\pm}(s)}|=e^{-ts}\leq e^{-2t}$ and $|\frac{d}{ds}\zeta_{1,\pm}(s)|=\sqrt{1+b^2}$, where
  \begin{align*}
    b = \frac{|\alpha m|}{2}\frac{1}{S_1-2}, \quad c_1 = 2\left(1+\frac{1}{1+b^2}\right), \quad c_2 = \frac{2b}{1+b^2}.
  \end{align*}
  By these relations, \eqref{E:Ko_Lm_Bo}, \eqref{Pf_KmDP:PA_Res}, \eqref{Pf_KmDP:Bd_Res}, \eqref{Pf_KmDP:Hz1}, we find that
  \begin{align*}
    \|J_1\|_{L^2(S^2)} &\leq \frac{Ce^{-2t}}{|\alpha m|^{2/3}}\left(\int_2^{S_1}\frac{1}{\sqrt{(s-c_1)^2+c_2^2}}\,ds\right)\|\mathbb{Q}_mf\|_{L^2(S^2)} \\
    &= \frac{Ce^{-2t}}{|\alpha m|^{2/3}}\left[\log\left(s-c_1+\sqrt{(s-c_1)^2+c_2^2}\right)\right]_{s=2}^{S_1}\|\mathbb{Q}_mf\|_{L^2(S^2)}.
  \end{align*}
  Moreover, since $S_1=C|\alpha m|^{2/3}$, we have $C|\alpha m|^{1/3}\leq b\leq C'|\alpha m|^{1/3}$ and thus $b\geq1$ when $|\alpha|>4$ is large. Then since $c_1-2=2(1+b^2)^{-1}$ and $b^{-1}\leq c_2\leq 2b^{-1}$,
  \begin{align*}
    1 \leq S_1-c_1 \leq S_1, \quad \frac{C}{|\alpha m|^{1/3}} \leq c_2 \leq S_1, \quad (c_1-2)^2+c_2^2 \leq 1
  \end{align*}
  for sufficiently large $|\alpha|>4$ and thus
  \begin{align*}
    1 &\leq S_1-c_1+\sqrt{(S_1-c_1)^2+c_2^2} \leq CS_1 = C|\alpha m|^{2/3}, \\
    1 &\geq \sqrt{(c_1-2)^2+c_2^2}-(c_1-2) \geq \frac{c_2^2}{2} \geq \frac{C}{|\alpha m|^{2/3}}
  \end{align*}
  by the mean value theorem for $\sqrt{(c_1-2)^2+s}$ with $s\geq0$. Hence
  \begin{align} \label{Pf_KmDP:J1}
    \|J_1\|_{L^2(S^2)} \leq \frac{Ce^{-2t}}{|\alpha m|^{2/3}}\log\Bigl(C|\alpha m|^{2/3}\Bigr)\|\mathbb{Q}_mf\|_{L^2(S^2)}.
  \end{align}
  For $s\in I_2$ we have $|e^{t\zeta_{2,\pm}(s)}|=e^{-ts}\leq e^{-2t}$ and
  \begin{align*}
    |\zeta_{2,\pm}(s)+4| \geq |\mathrm{Im}\,\zeta_{2,\pm}(s)| \geq \frac{|\alpha m|}{2}, \quad \left|\frac{d\zeta_{2,\pm}}{ds}(s)\right| = \sqrt{1+\frac{Cs^4}{|\alpha m|^2}} \leq \frac{Cs^2}{|\alpha m|}
  \end{align*}
  by $s^2/|\alpha m|\geq S_2^2/|\alpha m|=C|m|^{1/3}\geq C$. It follows from these inequalities, \eqref{E:Ko_Lm_Bo}, \eqref{Pf_KmDP:PA_Res}, \eqref{Pf_KmDP:Bd_Res}, \eqref{Pf_KmDP:Hz234}, and $S_1=C|\alpha m|^{2/3}$ that
  \begin{align} \label{Pf_KmDP:J2}
    \|J_2\|_{L^2(S^2)} \leq \frac{Ce^{-2t}}{|\alpha m|^2}\left(\int_{S_2}^{S_1}s\,ds\right)\|\mathbb{Q}_mf\|_{L^2(S^2)} \leq \frac{Ce^{-2t}}{|\alpha m|^{2/3}}\|\mathbb{Q}_mf\|_{L^2(S^2)}.
  \end{align}
  Similarly, since $|e^{t\zeta_{3,\pm}(s)}|=e^{-S_2t}\leq e^{-2t}$, $|\frac{d}{ds}\zeta_{3,\pm}(s)|=1$, and
  \begin{align*}
    |\zeta_{3,\pm}(s)+4| \geq |\mathrm{Im}\,\zeta_{3,\pm}(s)| \geq |\alpha m|-|\alpha|^{1/2}|m| \geq \frac{|\alpha m|}{2}, \quad s\in I_3
  \end{align*}
  by $|\alpha|>4$, we see by \eqref{E:Ko_Lm_Bo}, \eqref{Pf_KmDP:PA_Res}, \eqref{Pf_KmDP:Bd_Res}, and \eqref{Pf_KmDP:Hz234} that
  \begin{align} \label{Pf_KmDP:J3}
    \|J_3\|_{L^2(S^2)} \leq \frac{C|I_3|e^{-2t}}{|\alpha|^{3/2}|m|^{5/3}}\|\mathbb{Q}_mf\|_{L^2(S^2)} = \frac{Ce^{-2t}}{|\alpha||m|^{2/3}}\|\mathbb{Q}_mf\|_{L^2(S^2)},
  \end{align}
  where $|I_3|=2|\alpha|^{1/2}|m|$ is the length of $I_3$. For $s\in I_4$ we have
  \begin{gather*}
    |e^{t\zeta_{4,\pm}(s)}| = e^{-st} \leq e^{-2t}, \quad |\zeta_{4,\pm}(s)+4| \geq |\mathrm{Im}\,\zeta_{4,\pm}(s)| \geq C(|\alpha m|+s), \\
    \left|\frac{d\zeta_{4,\pm}}{ds}(s)\right| = \sqrt{1+(4C_1|m|^{1/3})^2} \leq C.
  \end{gather*}
  Note that here $|m|=1,2$. By these inequalities, \eqref{E:Ko_Lm_Bo}, \eqref{Pf_KmDP:PA_Res}, \eqref{Pf_KmDP:Bd_Res}, and \eqref{Pf_KmDP:Hz234},
  \begin{align} \label{Pf_KmDP:J4}
    \begin{aligned}
      \|J_4\|_{L^2(S^2)} &\leq Ce^{-2t}\left(\int_{S_2}^\infty\frac{1}{s(|\alpha m|+s)}\,ds\right)\|\mathbb{Q}_mf\|_{L^2(S^2)} \\
      &= Ce^{-2t}\left\{\frac{1}{|\alpha m|}\log\left(\frac{|\alpha m|+S_2}{S_2}\right)\right\}\|\mathbb{Q}_mf\|_{L^2(S^2)} \\
      &\leq \frac{Ce^{-2t}}{|\alpha m|}\log\Bigl(C|\alpha m|^{1/2}\Bigr)\|\mathbb{Q}_mf\|_{L^2(S^2)},
    \end{aligned}
  \end{align}
  where we also used $(|\alpha m|+S_2)/S_2=C|\alpha|^{1/2}|m|^{1/3}+1\leq C|\alpha m|^{1/2}$ in the last inequality. Now, noting that $|\alpha|>4$ is sufficiently large, we deduce from \eqref{Pf_KmDP:J1}--\eqref{Pf_KmDP:J4} that
  \begin{align*}
    \|J\|_{L^2(S^2)} \leq \sum_{k=1}^4\|J_k\|_{L^2(S^2)} \leq \frac{Ce^{-2t}}{|\alpha m|^{2/3}}\log\Bigl(C|\alpha m|^{2/3}\Bigr)\|\mathbb{Q}_mf\|_{L^2(S^2)}.
  \end{align*}
  Hence we get \eqref{E:Kom_Dec_P} by applying this inequality and \eqref{Pf_KmDP:Int_PA} to \eqref{Pf_KmDP:Pef_Dec}.
\end{proof}

Now we recall that $L_\alpha=A-i\alpha\Lambda$ on $\mathcal{X}$ is diagonalized as
\begin{align*}
  L_\alpha = \oplus_{m\in\mathbb{Z}\setminus\{0\}}L_{\alpha,m}, \quad L_{\alpha,m} = L_\alpha|_{\mathcal{X}_m} = L_\alpha|_{\mathcal{P}_m\mathcal{X}},
\end{align*}
where $\mathcal{P}_m$ is the operator given by \eqref{E:Def_Proj}. Moreover, for $f\in\mathcal{X}$ we have
\begin{align*}
  \mathbb{Q}_m\mathcal{P}_mf = \mathbb{Q}\mathcal{P}_mf, \quad m\in\mathbb{Z}\setminus\{0\}, \quad \mathbb{P}_m\mathcal{P}_mf = (I-\mathbb{Q})\mathcal{P}_mf, \quad |m| = 1,2,
\end{align*}
where $\mathbb{Q}$ is the orthogonal projection from $\mathcal{X}$ onto $\mathcal{Y}=N_{\mathcal{X}}(\Lambda)^\perp$ (see Lemma \ref{L:Ko_NLam}). Hence by Theorems \ref{T:Kom_Dec_Q} and \ref{T:Kom_Dec_P} we get the next result which implies Theorem \ref{T:OL_Dec}.

\begin{theorem} \label{T:Ko_Dec}
  There exist constants $C_1,C_2>0$ such that
  \begin{align*}
    \|\mathbb{Q}\mathcal{P}_me^{tL_\alpha}f\|_{L^2(S^2)} \leq C_1e^{-C_2|\alpha|^{1/2}|m|^{2/3}t}\|\mathbb{Q}\mathcal{P}_mf\|_{L^2(S^2)}
  \end{align*}
  for all $t\geq0$, $f\in\mathcal{X}$, $\alpha\in\mathbb{R}$, and $m\in\mathbb{Z}\setminus\{0\}$ (note that $\mathbb{Q}\mathcal{P}_m=\mathcal{P}_m$ on $\mathcal{X}$ if $|m|\geq3$). Also, if $|m|=1,2$ and $|\alpha|>4$ is sufficiently large, then for all $t\geq0$ and $f\in\mathcal{X}$ we have
  \begin{multline*}
    \|(I-\mathbb{Q})\mathcal{P}_me^{tL_\alpha}f\|_{L^2(S^2)} \leq e^{-4t}\|(I-\mathbb{Q})\mathcal{P}_mf\|_{L^2(S^2)} \\
    + C_3|\alpha m|^{1/3}\log\Bigl(C_4|\alpha m|^{2/3}\Bigr)e^{-2t}\|\mathbb{Q}\mathcal{P}_mf\|_{L^2(S^2)},
  \end{multline*}
  where $C_3,C_4>0$ are constants independent of $t$, $f$, $\alpha$, and $m$.
\end{theorem}

\begin{proof}[Proof of Theorem \ref{T:OL_Dec}]
  For $\nu>0$ and $a\in\mathbb{R}$ we have $\mathcal{L}^{\nu,a}|_{\mathcal{X}}=\nu L_{a/\nu}$ and thus $e^{t\mathcal{L}^{\nu,a}}|_{\mathcal{X}}=e^{\nu tL_{a/\nu}}$ in $\mathcal{X}$. Hence Theorem \ref{T:OL_Dec} follows from Theorem \ref{T:Ko_Dec}.
\end{proof}

\section{Abstract results} \label{S:Abst}
In this section we present abstract results for a perturbed operator.

For a linear operator $T$ on a Banach space $\mathcal{B}$, we denote by $D_{\mathcal{B}}(T)$, $\rho_{\mathcal{B}}(T)$, and $\sigma_{\mathcal{B}}(T)$ the domain, the resolvent set, and the spectrum of $T$ in $\mathcal{B}$. Also, let $N_{\mathcal{B}}(T)$ and $R_{\mathcal{B}}(T)$ be the kernel and range of $T$ in $\mathcal{B}$.

Let $(\mathcal{X},(\cdot,\cdot)_{\mathcal{X}})$ be a Hilbert space and $A$ and $\Lambda$ linear operators on $\mathcal{X}$. We make the following assumptions.

\begin{assumption} \label{As:A}
  The operator $A$ is self-adjoint in $\mathcal{X}$ and satisfies
  \begin{align} \label{E:Ab_A_Po}
    (-Au,u)_{\mathcal{X}} \geq C_A\|u\|_{\mathcal{X}}^2, \quad u\in D_{\mathcal{X}}(A).
  \end{align}
  with some constant $C_A>0$.
\end{assumption}

\begin{assumption} \label{As:La_01}
  The following conditions hold:
  \begin{enumerate}
    \item The operator $\Lambda$ is densely defined, closed, and $A$-compact in $\mathcal{X}$.
    \item Let $\mathcal{Y}=N_{\mathcal{X}}(\Lambda)^\perp$ be the orthogonal complement of $N_{\mathcal{X}}(\Lambda)$ in $\mathcal{X}$ and $\mathbb{Q}$ the orthogonal projection from $\mathcal{X}$ onto $\mathcal{Y}$. Then $\mathbb{Q}A\subset A\mathbb{Q}$ in $\mathcal{X}$.
  \end{enumerate}
\end{assumption}

\begin{assumption} \label{As:La_02}
  There exist a Hilbert space $(\mathcal{H},(\cdot,\cdot)_{\mathcal{H}})$, a closed symmetric operator $B_1$ on $\mathcal{H}$, and a bounded self-adjoint operator $B_2$ on $\mathcal{X}$ such that the following conditions hold:
  \begin{enumerate}
    \item The inclusion $\mathcal{X}\subset\mathcal{H}$ holds and $(u,v)_{\mathcal{X}}=(u,v)_{\mathcal{H}}$ for all $u,v\in\mathcal{X}$.
    \item The relation $N_{\mathcal{X}}(\Lambda)=N_{\mathcal{X}}(B_2)$ holds in $\mathcal{X}$ and
    \begin{align*}
      B_2u \in D_{\mathcal{H}}(B_1), \quad B_1B_2u = \Lambda u \in\mathcal{X} \quad\text{for all}\quad u\in D_{\mathcal{X}}(\Lambda).
    \end{align*}
    \item There exists a constant $C>0$ such that
    \begin{alignat}{3}
      (u,B_2u)_{\mathcal{X}} &\geq C\|u\|_{\mathcal{X}}^2, &\quad &u\in \mathcal{Y}, \label{E:uB2u} \\
      \mathrm{Re}(-Au,B_2u)_{\mathcal{X}} &\geq C\|(-A)^{1/2}u\|_{\mathcal{X}}^2, &\quad &u\in D_{\mathcal{X}}(A)\cap\mathcal{Y}. \label{E:AB2u}
    \end{alignat}
  \end{enumerate}
\end{assumption}

Note that $B_2$ is a linear operator on the original space $\mathcal{X}$, not on the auxiliary space $\mathcal{H}$. Also, the operator $B_1$ on $\mathcal{H}$ does not necessarily map $\mathcal{X}$ into itself.

By $\mathbb{Q}A\subset A\mathbb{Q}$ in Assumption \ref{As:La_01} we can consider $\mathbb{Q}A$ as a linear operator
\begin{align*}
  \mathbb{Q}A\colon D_{\mathcal{Y}}(\mathbb{Q}A) \subset \mathcal{Y}\to\mathcal{Y}, \quad D_{\mathcal{Y}}(\mathbb{Q}A) = D_{\mathcal{X}}(A)\cap\mathcal{Y}.
\end{align*}
In what follows, we use the notation $\mathcal{N}=N_{\mathcal{X}}(\Lambda)$ for simplicity. Let $\mathbb{P}=I-\mathbb{Q}$ be the orthogonal projection from $\mathcal{X}$ onto $\mathcal{N}$ (note that $\mathcal{N}$ is closed in $\mathcal{X}$ since $\Lambda$ is closed). Then $\mathbb{P}A\subset A\mathbb{P}$ and we can also consider $\mathbb{P}A$ as a linear operator
\begin{align*}
  \mathbb{P}A\colon D_{\mathcal{N}}(\mathbb{P}A) \subset \mathcal{N} \to \mathcal{N}, \quad D_{\mathcal{N}}(\mathbb{P}A) = D_{\mathcal{X}}(A)\cap\mathcal{N}.
\end{align*}
Note that $\mathbb{Q}A$ and $\mathbb{P}A$ are closed in $\mathcal{Y}$ and in $\mathcal{N}$, respectively. Also, $\mathbb{Q}\Lambda$ is $\mathbb{Q}A$-compact in $\mathcal{Y}$. For $\alpha\in\mathbb{R}$ we define a linear operator $L_\alpha$ on $\mathcal{X}$ by
\begin{align*}
  L_\alpha=A-i\alpha\Lambda, \quad D_{\mathcal{X}}(L_\alpha)=D_{\mathcal{X}}(A)
\end{align*}
and consider $\mathbb{Q}L_\alpha=\mathbb{Q}A-i\alpha\mathbb{Q}\Lambda$ on $\mathcal{Y}$ with domain $D_{\mathcal{Y}}(\mathbb{Q}L_\alpha)=D_{\mathcal{Y}}(\mathbb{Q}A)$.

Our purpose is to derive a decay estimate for the $\mathbb{Q}$-part of the semigroup generated by $L_\alpha$ in terms of $\alpha$. The following results were obtained in \cite[Lemmas 6.4--6.6]{Miu21pre}.

\begin{lemma} \label{L:PS_Pre}
  Under Assumptions \ref{As:A}--\ref{As:La_02}, we have
  \begin{align} \label{E:QB2}
    \mathbb{Q}B_2 = B_2 \quad\text{on}\quad \mathcal{X}, \quad \mathrm{Im}(\Lambda u,\mathbb{Q}B_2u)_{\mathcal{X}} = \mathrm{Im}(\Lambda u,B_2u)_{\mathcal{X}} = 0, \quad u\in D_{\mathcal{X}}(\Lambda).
  \end{align}
  Moreover, $L_\alpha$ and $\mathbb{Q}L_\alpha$ are closed in $\mathcal{X}$ and in $\mathcal{Y}$, respectively, and
  \begin{align} \label{E:PS_ReSet}
    \{\zeta\in\mathbb{C} \mid \mathrm{Re}\,\zeta\geq0\} \subset \rho_{\mathcal{X}}(L_\alpha) = \rho_{\mathcal{Y}}(\mathbb{Q}L_\alpha)\cap\rho_{\mathcal{N}}(\mathbb{P}A)
  \end{align}
  for all $\alpha\in\mathbb{R}$, and if $\zeta\in\rho_{\mathcal{X}}(L_\alpha)$ and $f\in\mathcal{X}$, then
  \begin{align} \label{E:PS_ReOp}
    \begin{aligned}
      \mathbb{Q}(\zeta-L_\alpha)^{-1}f &= (\zeta-\mathbb{Q}L_\alpha)^{-1}\mathbb{Q}f, \\
      \mathbb{P}(\zeta-L_\alpha)^{-1}f &= (\zeta-\mathbb{P}A)^{-1}\mathbb{P}f-i\alpha(\zeta-\mathbb{P}A)^{-1}\mathbb{P}\Lambda(\zeta-\mathbb{Q}L_\alpha)^{-1}\mathbb{Q}f.
    \end{aligned}
  \end{align}
\end{lemma}

Moreover, the next result was shown in \cite[Theorem 6.7]{Miu21pre} by application of the Gearhart--Pr\"{u}ss type theorem given by Wei \cite{Wei21} to the $m$-accretive operator $-\mathbb{Q}L_\alpha$ on the Hilbert space $\mathcal{Y}'=\mathcal{Y}$ equipped with the weighted inner product $(u,v)_{\mathcal{Y}'}=(u,B_2v)_{\mathcal{X}}$.

\begin{theorem} \label{T:Semi_BoAl}
  Under Assumptions \ref{As:A}--\ref{As:La_02}, the operator $L_\alpha$ generates an analytic semigroup $\{e^{tL_\alpha}\}_{t\geq0}$ in $\mathcal{X}$ for all $\alpha\in\mathbb{R}$. Moreover, there exist positive constants $C_1$ and $C_2$ depending only on $\|B_2\|_{\mathcal{X}\to\mathcal{X}}$ and the constants appearing in \eqref{E:uB2u} and \eqref{E:AB2u} (and in particular independent of the constant $C_A$ appearing in \eqref{E:Ab_A_Po}) such that
  \begin{align} \label{E:Semi_BoAl}
    \|\mathbb{Q}e^{tL_\alpha}f\|_{\mathcal{X}} \leq C_1e^{-C_2t/\Phi_{\mathcal{Y}}(-\mathbb{Q}L_\alpha)}\|\mathbb{Q}f\|_{\mathcal{X}}, \quad t\geq0, \, f\in\mathcal{X}
  \end{align}
  for all $\alpha\in\mathbb{R}$, where $\Phi_{\mathcal{Y}}(-\mathbb{Q}L_\alpha)=\sup_{\lambda\in\mathbb{R}}\|(i\lambda-\mathbb{Q}L_\alpha)^{-1}\|_{\mathcal{Y}\to\mathcal{Y}}$.
\end{theorem}

In order to get an upper bound of $\Phi_{\mathcal{Y}}(-\mathbb{Q}L_\alpha)$ by a function of $\alpha$, we make an addition assumption on $A$ and $\Lambda$.

\begin{assumption} \label{As:Est}
  There exist bounded nonnegative functions
  \begin{align*}
    h_1,h_2\colon(0,\infty)\times\mathbb{R}\to[0,\infty) \quad\text{satisfying}\quad \lim_{\xi\to\infty}\sup_{\mu\in\mathbb{R}}h_j(\xi,\mu) = 0, \quad j=1,2,
  \end{align*}
  a constant $C>0$, a Banach space $(\mathcal{B},\|\cdot\|_{\mathcal{B}})$, and a closed operator $B_3$ on $\mathcal{B}$ such that the following conditions hold:
  \begin{enumerate}
    \item For all $\xi\in(0,\infty)$ and $\mu\in\mathbb{R}$ we have
    \begin{align} \label{E:u_h}
      \|u\|_{\mathcal{X}}^2 \leq C\left(\xi^2\|\mathbb{Q}(\mu-\Lambda)u\|_{\mathcal{X}}^2+h_1(\xi,\mu)^2\|(-A)^{1/2}u\|_{\mathcal{X}}^2\right), \quad u\in D_{\mathcal{X}}(A)\cap\mathcal{Y}.
    \end{align}
    \item The inclusions $\mathcal{X}\subset\mathcal{B}$ and $D_{\mathcal{X}}((-A)^{1/2})\subset D_{\mathcal{B}}(B_3)$ hold and
    \begin{align} \label{E:ALam_B3}
      |\mathrm{Im}(Au,\Lambda u)_{\mathcal{X}}| \leq C\|(-A)^{1/2}u\|_{\mathcal{X}}\|B_3u\|_{\mathcal{B}}, \quad u\in D_{\mathcal{X}}(A)\cap\mathcal{Y}.
    \end{align}
    Moreover, for all $\xi\in(0,\infty)$ and $\mu\in\mathbb{R}$ we have
    \begin{align} \label{E:B3u_h}
    \|B_3u\|_{\mathcal{B}}^2 \leq C\left(\xi^2\|\mathbb{Q}(\mu-\Lambda)u\|_{\mathcal{X}}^2+h_2(\xi,\mu)^2\|(-A)^{1/2}u\|_{\mathcal{X}}^2\right), \quad u\in D_{\mathcal{X}}(A)\cap\mathcal{Y}.
    \end{align}
  \end{enumerate}
\end{assumption}

Let us give an estimate for the resolvent of $\mathbb{Q}L_\alpha$ along the imaginary axis which was originally shown in \cite[Theorem 2.9]{IbMaMa19} under assumptions on coercive estimates for $\mu-\Lambda$ on $\mathcal{X}$ with $\mu\in\mathbb{R}$ away from zero and for $\mathbb{Q}(\mu-\Lambda)$ on $\mathcal{Y}$ with $\mu\in\mathbb{R}$ close to zero.

\begin{theorem} \label{T:PS_Bound}
  Under Assumptions \ref{As:A}--\ref{As:La_02} and \ref{As:Est}, we have
  \begin{align} \label{E:PS_Bound}
    \|\mathbb{Q}(i\lambda+L_\alpha)^{-1}\|_{\mathcal{X}\to\mathcal{X}} = \|(i\lambda+\mathbb{Q}L_\alpha)^{-1}\|_{\mathcal{Y}\to\mathcal{Y}} \leq CF\left(\alpha,\frac{\lambda}{\alpha}\right)
  \end{align}
  for all $\alpha\in\mathbb{R}\setminus\{0\}$ and $\lambda\in\mathbb{R}$. Here
  \begin{align} \label{E:PS_F}
    F(\alpha,\mu) = \inf_{\xi_1,\xi_2>0}\left(\frac{\xi_1}{|\alpha|}+\frac{\xi_1^2\xi_2^2}{\alpha^2}+\frac{\xi_1^2h_2(\xi_2,\mu)}{|\alpha|}+h_1(\xi_1,\mu)^2\right)
  \end{align}
  for $\alpha\in\mathbb{R}\setminus\{0\}$ and $\mu\in\mathbb{R}$. Also, $C>0$ is a constant depending only on $\|B_2\|_{\mathcal{X}\to\mathcal{X}}$ and the constants appearing in \eqref{E:uB2u}, \eqref{E:AB2u}, and \eqref{E:u_h}--\eqref{E:B3u_h} (and in particular independent of the constant $C_A$ appearing in \eqref{E:Ab_A_Po}).
\end{theorem}

\begin{proof}
  Let $\alpha\in\mathbb{R}\setminus\{0\}$ and $\lambda\in\mathbb{R}$. Then
  \begin{align*}
    -i\lambda\in\rho_{\mathcal{X}}(L_\alpha)\subset\rho_{\mathcal{Y}}(\mathbb{Q}L_\alpha), \quad \mathbb{Q}(i\lambda+L_\alpha)^{-1}f = (i\lambda+\mathbb{Q}L_\alpha)^{-1}\mathbb{Q}f, \quad f\in\mathcal{X}
  \end{align*}
  by \eqref{E:PS_ReSet} and \eqref{E:PS_ReOp}. By these relations, we easily find that the first equality of \eqref{E:PS_Bound} holds. Let us prove the second inequality of \eqref{E:PS_Bound}. Without loss of generality, we may assume $\alpha>0$. We set $\mu=\lambda/\alpha$ and abbreviate $h_j(\xi_j,\mu)$ to $h_j$ for $\xi_1,\xi_2>0$ and $j=1,2$. In what follows, we denote by $C$ a general positive constant depending only on $\|B_2\|_{\mathcal{X}\to\mathcal{X}}$ and the constants appearing in \eqref{E:uB2u}, \eqref{E:AB2u}, and \eqref{E:u_h}--\eqref{E:B3u_h}.
  Let $u\in D_{\mathcal{Y}}(\mathbb{Q}L_\alpha)$ and
  \begin{align*}
    f = (i\lambda+\mathbb{Q}L_{\alpha})u = \mathbb{Q}Au+i\alpha\mathbb{Q}(\mu-\Lambda)u \in \mathcal{Y} \subset \mathcal{X}.
  \end{align*}
  We take the inner products of both sides with $\mathbb{Q}(\mu-\Lambda)u$ in $\mathcal{X}$. Then
  \begin{align} \label{Pf_PS:fQ}
    (f,\mathbb{Q}(\mu-\Lambda)u)_{\mathcal{X}} = \mu(\mathbb{Q}Au,\mathbb{Q}u)_{\mathcal{X}}-(\mathbb{Q}Au,\mathbb{Q}\Lambda u)_{\mathcal{X}}+i\alpha\|\mathbb{Q}(\mu-\Lambda)u\|_{\mathcal{X}}^2.
  \end{align}
  Since $\mathbb{Q}A\subset A\mathbb{Q}$, $\mathbb{Q}u=u$, and $A$ is self-adjoint in $\mathcal{X}$, we have
  \begin{align*}
    (\mathbb{Q}Au,\mathbb{Q}\Lambda u)_{\mathcal{X}} = (Au,\Lambda u)_{\mathcal{X}}, \quad \mathrm{Im}(\mathbb{Q}Au,\mathbb{Q}u)_{\mathcal{X}} = \mathrm{Im}(Au,u)_{\mathcal{X}} = 0.
  \end{align*}
  Thus, taking the imaginary part of \eqref{Pf_PS:fQ}, we obtain
  \begin{align*}
    \mathrm{Im}(f,\mathbb{Q}(\mu-\Lambda)u)_{\mathcal{X}} = -\mathrm{Im}(Au,\Lambda u)_{\mathcal{X}}+\alpha\|\mathbb{Q}(\mu-\Lambda)u\|_{\mathcal{X}}^2.
  \end{align*}
  By this equality and \eqref{E:ALam_B3} we get
  \begin{align*}
    \alpha\|\mathbb{Q}(\mu-\Lambda)u\|_{\mathcal{X}}^2 &\leq \|f\|_{\mathcal{X}}\|\mathbb{Q}(\mu-\Lambda)u\|_{\mathcal{X}}+C\|(-A)^{1/2}u\|_{\mathcal{X}}\|B_3u\|_{\mathcal{B}} \\
    &\leq \frac{\alpha}{2}\|\mathbb{Q}(\mu-\Lambda)u\|_{\mathcal{X}}^2+\frac{1}{2\alpha}\|f\|_{\mathcal{X}}^2+C\|(-A)^{1/2}u\|_{\mathcal{X}}\|B_3u\|_{\mathcal{B}}
  \end{align*}
  and therefore
  \begin{align} \label{Pf_PS:alpha}
    \alpha\|\mathbb{Q}(\mu-\Lambda)u\|_{\mathcal{X}}^2 \leq \frac{1}{\alpha}\|f\|_{\mathcal{X}}^2+C\|(-A)^{1/2}u\|_{\mathcal{X}}\|B_3u\|_{\mathcal{B}}.
  \end{align}
  Moreover, for each $\xi_2>0$ it follows from \eqref{E:B3u_h} and \eqref{Pf_PS:alpha} that
  \begin{align*}
    \|B_3u\|_{\mathcal{B}}^2 &\leq C\left(\xi_2^2\|\mathbb{Q}(\mu-\Lambda)u\|_{\mathcal{X}}^2+h_2^2\|(-A)^{1/2}u\|_{\mathcal{X}}^2\right) \\
    &\leq C\left(\frac{\xi_2^2}{\alpha^2}\|f\|_{\mathcal{X}}^2+\frac{\xi_2^2}{\alpha}\|(-A)^{1/2}u\|_{\mathcal{X}}\|B_3u\|_{\mathcal{B}}+h_2^2\|(-A)^{1/2}u\|_{\mathcal{X}}^2\right) \\
    &\leq \frac{1}{2}\|B_3u\|_{\mathcal{B}}^2+C\left\{\frac{\xi_2^2}{\alpha^2}\|f\|_{\mathcal{X}}^2+\left(\frac{\xi_2^4}{\alpha^2}+h_2^2\right)\|(-A)^{1/2}u\|_{\mathcal{X}}^2\right\}.
  \end{align*}
  Hence, subtracting $\frac{1}{2}\|B_2u\|_{\mathcal{X}}^2$ from both sides and taking the square root of the resulting inequality, we obtain
  \begin{align*}
    \|B_3u\|_{\mathcal{B}} \leq C\left\{\frac{\xi_2}{\alpha}\|f\|_{\mathcal{X}}+\left(\frac{\xi_2^2}{\alpha}+h_2\right)\|(-A)^{1/2}u\|_{\mathcal{X}}\right\}.
  \end{align*}
  We apply this inequality to \eqref{Pf_PS:alpha} to find that
  \begin{align*}
    \|\mathbb{Q}(\mu-\Lambda)u\|_{\mathcal{X}}^2 &\leq \frac{1}{\alpha^2}\|f\|_{\mathcal{X}}^2+\frac{C}{\alpha}\|(-A)^{1/2}u\|_{\mathcal{X}}\left\{\frac{\xi_2}{\alpha}\|f\|_{\mathcal{X}}+\left(\frac{\xi_2^2}{\alpha}+h_2\right)\|(-A)^{1/2}u\|_{\mathcal{X}}\right\} \\
    &\leq C\left\{\frac{1}{\alpha^2}\|f\|_{\mathcal{X}}^2+\left(\frac{\xi_2^2}{\alpha^2}+\frac{h_2}{\alpha}\right)\|(-A)^{1/2}u\|_{\mathcal{X}}^2\right\}.
  \end{align*}
  For each $\xi_1>0$, we deduce from \eqref{E:u_h} and the above inequality that
  \begin{align} \label{Pf_PS:u_fA}
    \begin{aligned}
      \|u\|_{\mathcal{X}}^2 &\leq C\left(\xi_1^2\|\mathbb{Q}(\mu-\Lambda)u\|_{\mathcal{X}}^2+h_1^2\|(-A)^{1/2}u\|_{\mathcal{X}}^2\right) \\
      &\leq C\left\{\frac{\xi_1^2}{\alpha^2}\|f\|_{\mathcal{X}}^2+\left(\frac{\xi_1^2\xi_2^2}{\alpha^2}+\frac{\xi_1^2h_2}{\alpha}+h_1^2\right)\|(-A)^{1/2}u\|_{\mathcal{X}}^2\right\}.
    \end{aligned}
  \end{align}
  To estimate $\|(-A)^{1/2}u\|_{\mathcal{X}}$, we observe by $f=(i\lambda+\mathbb{Q}L_\alpha)u=\mathbb{Q}(i\lambda+L_\alpha)u$ that
  \begin{align*}
    (f,B_2u)_{\mathcal{X}} = \bigl((i\lambda+L_\alpha)u,\mathbb{Q}B_2u\bigr)_{\mathcal{X}} = i\lambda(u,\mathbb{Q}B_2u)_{\mathcal{X}}+(Au,\mathbb{Q}B_2u)_{\mathcal{X}}-i\alpha(\Lambda u,\mathbb{Q}B_2u)_{\mathcal{X}}.
  \end{align*}
  We take the real part of this equality and use \eqref{E:QB2} and $\mathrm{Im}(u,B_2u)_{\mathcal{X}}=0$ by the self-adjointness of $B_2$ in $\mathcal{X}$ to get $\mathrm{Re}(f,B_2u)_{\mathcal{X}} = \mathrm{Re}(Au,B_2u)_{\mathcal{X}}$. Hence
  \begin{align*}
    \|(-A)^{1/2}u\|_{\mathcal{X}}^2 \leq C\,\mathrm{Re}(-Au,B_2u)_{\mathcal{X}} = -C\,\mathrm{Re}(f,B_2u)_{\mathcal{X}} \leq C\|f\|_{\mathcal{X}}\|B_2u\|_{\mathcal{X}} \leq C\|f\|_{\mathcal{X}}\|u\|_{\mathcal{X}}
  \end{align*}
  by \eqref{E:AB2u} and the boundedness of $B_2$. Applying this inequality to \eqref{Pf_PS:u_fA} we obtain
  \begin{align*}
    \|u\|_{\mathcal{X}}^2 &\leq C\left\{\frac{\xi_1^2}{\alpha^2}\|f\|_{\mathcal{X}}^2+\left(\frac{\xi_1^2\xi_2^2}{\alpha^2}+\frac{\xi_1^2h_2}{\alpha}+h_1^2\right)\|f\|_{\mathcal{X}}\|u\|_{\mathcal{X}}\right\} \\
    &\leq \frac{1}{2}\|u\|_{\mathcal{X}}^2+C\left\{\frac{\xi_1^2}{\alpha^2}+\left(\frac{\xi_1^2\xi_2^2}{\alpha^2}+\frac{\xi_1^2h_2}{\alpha}+h_1^2\right)^2\right\}\|f\|_{\mathcal{X}}^2.
  \end{align*}
  We further subtract $\frac{1}{2}\|u\|_{\mathcal{X}}^2$ from both sides and take the square root of the resulting inequality to find that
  \begin{align*}
    \|u\|_{\mathcal{X}} \leq C\left(\frac{\xi_1}{\alpha}+\frac{\xi_1^2\xi_2^2}{\alpha^2}+\frac{\xi_1^2h_2}{\alpha}+h_1^2\right)\|f\|_{\mathcal{X}}, \quad f = (i\lambda+\mathbb{Q}L_\alpha)u
  \end{align*}
  for all $u\in D_{\mathcal{Y}}(\mathbb{Q}L_\alpha)$. Since $-i\lambda\in\rho_{\mathcal{Y}}(\mathbb{Q}L_\alpha)$, the above inequality shows that
  \begin{align*}
    \|(i\lambda+\mathbb{Q}L_\alpha)^{-1}\|_{\mathcal{Y}\to\mathcal{Y}} \leq C\left(\frac{\xi_1}{|\alpha|}+\frac{\xi_1^2\xi_2^2}{\alpha^2}+\frac{\xi_1^2h_2(\xi_2,\mu)}{|\alpha|}+h_1(\xi_1,\mu)^2\right)
  \end{align*}
  for all $\xi_1,\xi_2>0$ with $\mu=\lambda/\alpha$. Therefore, the second inequality of \eqref{E:PS_Bound} is valid.
\end{proof}

\begin{remark} \label{R:PS_Bound}
  Contrary to \cite[Assumption 1]{IbMaMa19}, we do not assume that $A$ has a compact resolvent in Assumption \ref{As:A}. Thus the essential spectrum $\tilde{\sigma}_{\mathcal{X}}(A)$ of $A$ in $\mathcal{X}$ may be not empty, but we must have $\tilde{\sigma}_{\mathcal{Y}}(\mathbb{Q}A)=\emptyset$ for the essential spectrum of $\mathbb{Q}A$ in $\mathcal{Y}$ under Assumptions \ref{As:A}--\ref{As:La_02} and \ref{As:Est}. Indeed, $0\in\rho_{\mathcal{Y}}(\mathbb{Q}L_\alpha)$ and $\|(\mathbb{Q}L_\alpha)^{-1}\|_{\mathcal{Y}\to\mathcal{Y}} \leq CF(\alpha,0)$ for all $\alpha\in\mathbb{R}\setminus\{0\}$ by Theorem \ref{T:PS_Bound}. Thus, by a standard Neumann series argument,
  \begin{align*}
    \{\zeta\in\mathbb{C} \mid |\zeta| < 1/CF(\alpha,0)\} \subset \rho_{\mathcal{Y}}(\mathbb{Q}L_\alpha).
  \end{align*}
  On the other hand, since $\mathbb{Q}\Lambda$ is $\mathbb{Q}A$-compact in $\mathcal{Y}$ and the essential spectrum is invariant under a relatively compact perturbation (see \cite[Theorem IV.5.35]{Kato76}),
  \begin{align*}
    \tilde{\sigma}_{\mathcal{Y}}(\mathbb{Q}A) = \tilde{\sigma}_{\mathcal{Y}}(\mathbb{Q}A-i\alpha\mathbb{Q}\Lambda) = \tilde{\sigma}_{\mathcal{Y}}(\mathbb{Q}L_\alpha) \subset \sigma_{\mathcal{Y}}(\mathbb{Q}L_\alpha).
  \end{align*}
  Hence, if $\tilde{\sigma}_{\mathcal{Y}}(\mathbb{Q}A)$ contains some $\zeta\in\mathbb{C}$, then $\zeta\in\sigma_{\mathcal{Y}}(\mathbb{Q}L_\alpha)$ and thus $|\zeta|\geq1/CF(\alpha,0)$ for all $\alpha\in\mathbb{R}\setminus\{0\}$, but this is impossible since $\lim_{|\alpha|\to\infty}F(\alpha,0)=0$ by $\lim_{\xi\to\infty}h_j(\xi,0)=0$ for $j=1,2$. Thus $\tilde{\sigma}_{\mathcal{Y}}(\mathbb{Q}A)=\emptyset$. In particular, $\tilde{\sigma}_{\mathcal{X}}(A)=\emptyset$ if $\mathcal{N}=N_{\mathcal{X}}(\Lambda)=\{0\}$.
\end{remark}

Combining Theorems \ref{T:Semi_BoAl} and \ref{T:PS_Bound}, we obtain the following result.

\begin{theorem} \label{T:Decay_QL}
  Under Assumptions \ref{As:A}--\ref{As:La_02} and \ref{As:Est}, there exist positive constants $C_1$ and $C_2$ depending only on $\|B_2\|_{\mathcal{X}\to\mathcal{X}}$ and the constants appearing in \eqref{E:uB2u}, \eqref{E:AB2u}, and \eqref{E:u_h}--\eqref{E:B3u_h} (and in particular independent of the constant $C_A$ appearing in \eqref{E:Ab_A_Po}) such that
  \begin{align} \label{E:Decay_QL}
    \|\mathbb{Q}e^{tL_\alpha}f\|_{\mathcal{X}} \leq C_1e^{-C_2t/F(\alpha)}\|\mathbb{Q}f\|_{\mathcal{X}} \quad t\geq0, \, f\in\mathcal{X}
  \end{align}
  for all $\alpha\in\mathbb{R}\setminus\{0\}$, where $F(\alpha)=\sup_{\mu\in\mathbb{R}}F(\alpha,\mu)$ with $F(\alpha,\mu)$ given by \eqref{E:PS_F}.
\end{theorem}

\begin{proof}
  By \eqref{E:PS_Bound} we have
  \begin{align*}
    \Phi_{\mathcal{Y}}(-\mathbb{Q}L_\alpha) &= \sup_{\lambda\in\mathbb{R}}\|(i\lambda-\mathbb{Q}L_\alpha)^{-1}\|_{\mathcal{Y}\to\mathcal{Y}} = \sup_{\lambda\in\mathbb{R}}\|(i\lambda+\mathbb{Q}L_\alpha)^{-1}\|_{\mathcal{Y}\to\mathcal{Y}} \\
    &\leq C\sup_{\lambda\in\mathbb{R}}F\left(\alpha,\frac{\lambda}{\alpha}\right) = C\sup_{\mu\in\mathbb{R}}F(\alpha,\mu) = CF(\alpha).
  \end{align*}
  Hence \eqref{E:Decay_QL} follows from this inequality and \eqref{E:Semi_BoAl}.
\end{proof}

\section{Appendix: basic formulas of differential geometry} \label{S:DG}
This section gives some notations and basic formulas of differential geometry. We refer to e.g. \cite{Lee13,Lee18} for details.

Let $(M,g)$ be a two-dimensional Riemannian manifold. For a local coordinate system $(x^1,x^2)$ of $M$, let $(\partial_1,\partial_2)$ and $(dx^1,dx^2)$ be the coordinate frame and its dual coframe. We set $g_{jk}=g(\partial_j,\partial_k)$ for $j,k=1,2$ and denote by $(g^{jk})_{j,k=1,2}$ the inverse matrix of $(g_{jk})_{j,k=1,2}$ so that the inner products of one-forms and $(0,2)$-tensor fields on $M$ are given by $g(dx^j,dx^k)=g^{jk}$ and
\begin{align*}
  g(dx^{j_1}\otimes dx^{k_1},dx^{j_2}\otimes dx^{k_2}) = g(dx^{j_1},dx^{j_2})g(dx^{k_1},dx^{k_2}) = g^{j_1j_2}g^{k_1k_2}.
\end{align*}
Let $\Gamma_{jk}^l$ be the Christoffel symbols of $g$ given by
\begin{align*}
  \Gamma_{jk}^l = \frac{1}{2}\sum_{m=1,2}g^{lm}(\partial_jg_{km}+\partial_kg_{jm}-\partial_mg_{jk}), \quad j,k,l=1,2.
\end{align*}
For a (complex-valued) function $u$ on $M$, we write $\nabla u$ and $\nabla^2u$ for the gradient and the covariant Hessian of $u$, respectively, which are locally expressed as
\begin{align*}
  \nabla u = \sum_{j,k=1,2}g^{jk}(\partial_ju)\partial_k, \quad \nabla^2u = \sum_{j,k=1,2}u_{;j;k}dx^j\otimes dx^k
\end{align*}
with $u_{;j;k}=\partial_k\partial_ju-\sum_{l=1,2}\Gamma_{kj}^l\partial_lu$ for $j,k=1,2$. Then
\begin{align*}
  |\nabla u|^2 &= g(\nabla u,\nabla\bar{u}) = \sum_{j,k=1,2}g^{jk}\partial_ju\,\overline{\partial_ku}, \\
  |\nabla^2u|^2 &= g(\nabla^2u,\nabla^2\bar{u}) = \sum_{j_1,j_2,k_1,k_2=1,2}g^{j_1j_2}g^{k_1k_2}u_{;j_1;k_1}\,\overline{u_{;j_2;k_2}}.
\end{align*}
Also, the Laplace--Beltrami operator $\Delta$ on $M$ is locally given by
\begin{align*}
  \Delta u = \frac{1}{\det g}\sum_{j,k=1,2}\partial_j\Bigl(g^{jk}\partial_ku\sqrt{\det g}\Bigr), \quad \det g = \det\bigl((g_{jk})_{j,k=1,2}\bigr).
\end{align*}
We use these expressions under the spherical coordinate system of $M=S^2$.

\begin{lemma} \label{L:Re_SC}
  Let $(x^1,x^2)=(\theta,\varphi)$ be the spherical coordinate system
  \begin{align*}
    [0,\pi]\times[0,2\pi) \ni (\theta,\varphi) \mapsto \mathbf{x}(\theta,\varphi) = (\sin\theta\cos\varphi,\sin\theta\sin\varphi,\cos\theta) \in S^2.
  \end{align*}
  Then for a function $u$ on $S^2$ we have
  \begin{align} \label{E:Re_SC}
    \begin{aligned}
      |\nabla u|^2 &= |\partial_\theta u|^2+\frac{1}{\sin^2\theta}|\partial_\varphi u|^2, \\
      |\nabla^2u|^2 &= |\partial_\theta^2u|^2+\frac{2}{\sin^2\theta}\left|\partial_\varphi\partial_\theta u-\frac{\cos\theta}{\sin\theta}\partial_\varphi u\right|^2+\frac{1}{\sin^4\theta}|\partial_\varphi^2u+\sin\theta\cos\theta\,\partial_\theta u|^2, \\
      \Delta u &= \frac{1}{\sin\theta}\partial_\theta\bigl(\sin\theta\,\partial_\theta u\bigr)+\frac{1}{\sin^2\theta}\partial_\varphi^2u.
    \end{aligned}
  \end{align}
\end{lemma}

\begin{proof}
  We use the index $j=\theta,\varphi$ instead of $j=1,2$. Since
  \begin{align*}
    \partial_\theta\mathbf{x} =
    \begin{pmatrix}
      \cos\theta\cos\varphi \\
      \cos\theta\sin\varphi \\
      -\sin\theta
    \end{pmatrix},
    \quad \partial_\varphi\mathbf{x} =
    \begin{pmatrix}
      -\sin\theta\sin\varphi \\
      \sin\theta\cos\varphi \\
      0
    \end{pmatrix},
  \end{align*}
  we observe by direct calculations that
  \begin{align*}
    g_{\theta\theta} &= 1, \quad g_{\varphi\varphi} = \sin^2\theta, \quad \det g = \sin^2\theta, \\
    g^{\theta\theta} &= 1, \quad g^{\varphi\varphi} = \frac{1}{\sin^2\theta}, \quad \Gamma_{\varphi\varphi}^\theta = -\sin\theta\cos\theta, \quad \Gamma_{\theta\varphi}^\varphi = \Gamma_{\varphi\theta}^\varphi = \frac{\cos\theta}{\sin\theta}
  \end{align*}
  and the other $g_{jk}$, $g^{jk}$, and $\Gamma_{jk}^l$ vanish identically. Hence
  \begin{align*}
    u_{;\theta;\theta} = \partial_\theta^2u, \quad u_{;\theta;\varphi} = u_{;\varphi;\theta} = \partial_\varphi\partial_\theta u-\frac{\cos\theta}{\sin\theta}\partial_\varphi u, \quad u_{;\varphi;\varphi} = \partial_\varphi^2u+\sin\theta\cos\theta\,\partial_\theta u
  \end{align*}
  and we obtain \eqref{E:Re_SC} by the above expressions.
\end{proof}

\section*{Acknowledgments}
The work of the first author was supported by JSPS KAKENHI Grant Numbers 20K03698, 19H05597, 20H00118. Also, the work of the second author was supported by Grant-in-Aid for JSPS Fellows No. 19J00693.

\bibliographystyle{abbrv}
\bibliography{NSK_Rate_Ref}

\end{document}